\documentclass[onefignum,onetabnum]{siamart190516}

\usepackage[T1]{fontenc}

\allowdisplaybreaks

\usepackage{braket,amsfonts}
\usepackage{amsmath,amssymb,amsxtra}
\usepackage{bm,mathrsfs,marvosym,eurosym}
\usepackage{mathtools,enumerate}
\usepackage{caption}
\usepackage{booktabs}
\usepackage{threeparttable}
\usepackage{comment}
\usepackage{nccmath}
\usepackage{bbm}
\usepackage{array}
\usepackage{multirow}
\usepackage{subcaption} 
\usepackage[normalem]{ulem}
\usepackage{upgreek}
\usepackage{tikz}
\usetikzlibrary{arrows,calc,decorations.pathreplacing}
\usetikzlibrary{arrows.meta, positioning}
\usepackage{graphicx}
\usepackage{epstopdf}
\usepackage{algorithmic}
\usepackage{hyperref}

\def\0{{\textbf{0}}}

\def\gl{\gamma_\mathrm{lin}}

\def \fy{\varphi}
\def\<{{\langle }}
\def\>{{\rangle }}

\def\II{(\Omega)}

\definecolor{darkred}{RGB}{200,0,0}
\definecolor{darkblue}{RGB}{0,0,200}
\definecolor{darkgreen}{RGB}{0,160,0}

\usepackage{lipsum}

\newsiamremark{remark}{Remark}
\newsiamremark{hypothesis}{Hypothesis}
\crefname{hypothesis}{Hypothesis}{Hypotheses}
\newsiamthm{claim}{Claim}

\newcommand{\Tau}{\bm{\tau}}

\newtheorem{assumption}{\textit{Assumption}}

\makeatletter
\newcommand*{\addFileDependency}[1]{
  \typeout{(#1)}
  \@addtofilelist{#1}
  \IfFileExists{#1}{}{\typeout{No file #1.}}
}
\makeatother

\title{Convergence analysis of a parareal algorithm with \\
multistep fine propagator\thanks{The work of Z. Zhou is supported by by National Natural Science Foundation of China (Project 12422117),
Hong Kong Research Grants Council (15302323) and an internal grant of Hong Kong Polytechnic University (Project ID: P0053938, Work Programme: 4-ZZVA).}}
\author{ 
Georgios Akrivis\thanks{Department of Computer Science and Engineering, University of Ioannina, 451 10 Ioannina,
Greece.  (\texttt{akrivis@cse.uoi.gr}).}
\and Qingle Lin\thanks{Department of Applied Mathematics, The Hong Kong Polytechnic University, Kowloon, Hong Kong, P.R. China (\texttt{qingle.lin@connect.polyu.hk}, \texttt{zhizhou@polyu.edu.hk})}
\and Zhi Zhou\footnotemark[3] 
}

\headers{Parareal with multistep FP}{G. Akrivis, Q. Lin, and Z. Zhou}

\def\e{{\rm e}}

\ifpdf
\hypersetup{
  pdftitle={Convergence analysis of a parareal algorithm with multistep fine propagator},
  pdfauthor={Georgios Akrivis, Qingle Lin, Zhi Zhou}
}
\fi

\graphicspath{{pics/}}
\begin{document}

\maketitle
 
\begin{abstract}
The parareal algorithm is a powerful parallel-in-time integration method that accelerates the numerical solution of evolution equations
by iteratively combining a fine propagator and a coarse propagator. Although the convergence of the parareal algorithm has been extensively 
studied, most existing analyses assume that the fine propagator is either an exact solver or a single-step method. In this paper, we construct 
and analyze a parareal algorithm for solving parabolic equations, where the fine propagator is based on the two-step backward 
differentiation formula (BDF2), while the coarse propagator remains a single-step method. We propose a novel approach to design an effective 
correction for the initialization steps and establish linear convergence of the iteration. Numerical results fully support the theoretical
findings, show clear improvements over existing multistep parareal strategies, and indicate that the proposed approach extends effectively to higher-order BDF methods and to nonlinear problems.
\end{abstract}



\section{Introduction}
Parallel-in-time (PinT) methods have been an active area of research for over 60 years, dating back to the pioneering work of Nievergelt in 1964~\cite{nievergelt1964parallel}.
Over the decades, a wide variety of algorithms have been developed to address the challenges of solving time-dependent problems more efficiently. Representative works include 
parareal algorithms  \cite{Bal:2005, GanderJiang:SISC2013, GanderVandewalle:2007,
MadayTurinici:2002,WuZhou:2015}, Laplace transform based methods \cite{Sheen:2013, McLeanSloanThomee:2006,  SheenSloanThomee:2000,Weideman:2007}, diagonalization based methods (ParaDiag) \cite{McDonald:2018,GanderWu:2019,Banjai:2012,GanderHalpern:2017,WuZhouZhou:2022,WuZhou:2021,MR2385067}, space-time multigrid methods
\cite{Hackbusch:book, LubichOstermann:1987, Vandewalle:1992} and so on. For a broader overview, we refer readers to the comprehensive review in \cite{Gander:2015,OngSchroder:2020,gander2025time} and references
therein. These methods are particularly advantageous for solving partial differential equations (PDEs) over long time intervals, where traditional sequential algorithms often become computationally infeasible.  

Among the various PinT approaches, the parareal algorithm \cite{lions2001resolution} has garnered significant attention due to its simplicity, flexibility, and demonstrated success across numerous applications. The parareal algorithm divides the time domain into $N$ subintervals, enabling the concurrent solution of $N$ subproblems using existing numerical methods. The algorithm operates iteratively, combining a fast but low-accuracy  coarse propagator (CP), which is applied sequentially, with a high-accuracy fine propagator (FP), which is applied in parallel. In each iteration, the CP generates an initial approximation, which the FP refines through parallel corrections, gradually improving the solution over successive iterations.

Notably, most existing studies on parareal development and convergence assume the FP is either an exact solver \cite{GanderVandewalle:2007} or a single-step method \cite{gander2023unified,Dobrev:2017,WuZhou:2015}. However, linear multistep methods, such as the backward differentiation formula (BDF) methods, offer significant advantages for solving PDEs. These methods achieve higher accuracy with fewer function evaluations compared to single-step methods and can efficiently handle implicit formulations, which are crucial for solving parabolic equations. Importantly, they do not suffer from order reduction, a common issue with high-order single-step methods when solving PDEs; see, e.g., \cite{MR875165,MR1153166,MR4316580,MR1098868} and Crouzeix’s thesis \cite{Crouzeix:1975}. However, multistep schemes require multiple previous approximations for initialization, making their integration into the parareal algorithm more complicated. To ensure the convergence, it is crucial to  design proper corrections for the initialization steps within each time window.

Only a few works address parareal-type temporal parallelism with linear multistep methods. In \cite{falgout2017multigrid,falgout2019parallel}, the authors introduce multigrid reduction in time (MGRIT) and incorporate BDF methods to achieve temporal parallelism. A key idea is to group the unknowns into packets containing as many time levels as the BDF order. This “packetization” reformulates the linear multistep scheme as an equivalent single-step method, which can then be incorporated straightforwardly into the MGRIT framework. However, the resulting approach can become unstable when high-order BDF schemes are used, as discussed in \cite[Section 4.1]{falgout2019parallel}. 
In \cite{ait2018multi}, Ait-Ameur, Maday and Tajchman propose a variant of the parareal algorithm that successfully incorporates a second-order BDF method (BDF2) as the FP, demonstrating how proper initialization corrections preserve both accuracy and convergence. 
Very recently, a rigorous convergence analysis has been established in \cite{ait2024multi}, providing a state-of-the-art contribution to the field.
However, that analysis is restricted to ordinary differential equations~(ODEs) and relies on consistency assumptions that are not generally valid for PDE flows. 
In particular, \cite[Assumption (H)(7)]{ait2024multi} requires (for a nonautonomous system) that the coarse propagator satisfies
\begin{equation*}
\|G_{\bm{\tau}}(u)-E_{\bm{\tau}}(u)\|\leqslant C\bm{\tau}\|u\|,
\end{equation*}
in some norm, where $G_{\bm{\tau}}$ denotes the coarse propagator and $E_{\bm{\tau}}$ the exact evolution operator over one step.
Such a bound is typically violated even for linear parabolic equations. As an illustration, consider the model problem $u' + Au = 0$ with $A$ an elliptic operator (so that $E_{\bm{\tau}}=\e^{-\bm{\tau} A}$).  For parabolic problems, sharp time-discretization error estimates depend on higher (graph) norms of the initial data, e.g., $\|Au\|$ (or, more generally, $\|A^s u\|$), rather than on $\|u\|$ alone; see \cite[Theorems~3.1 and~3.2]{Thomee:2006}. 
Accordingly, one typically obtains bounds of the form
\begin{equation*}
\|G_{\bm{\tau}}u - E_{\bm{\tau}}u\|\leqslant  C\bm{\tau}^s \|A^s u\|\qquad \forall\,s\in(0,1],
\end{equation*}
but not an $\mathcal{O}(\bm{\tau})\|u\|$ estimate on the natural state space. Similar restrictive consistency requirements are imposed on the fine propagator and on the error propagators; see \cite[relations (10), (11), (13)]{ait2024multi}.  These observations motivate us to study the parareal with a multistep FP (hereafter, \emph{F-multistep parareal}) applied to parabolic equations under sharp consistency assumptions.

In this work, we establish a sharp linear convergence estimate for the {F-multistep parareal} applied to linear parabolic equations, using BDF2 as the FP. 
Motivated by this estimate, we propose a more principled update scheme for the auxiliary variables introduced in the F-multistep parareal framework. 
The resulting method has the same computational cost as the approach in \cite{ait2024multi}, but converges substantially faster. 
Moreover, we prove that, as the coarsening factor $J\to\infty$, the convergence factor of the F-multistep parareal method approaches that of the plain parareal algorithm with the same single-step CP and an exact FP. 
We present several numerical experiments on both linear and nonlinear problems to demonstrate the effectiveness of the proposed approach. 
The results show that the new update scheme significantly improves performance compared with the one proposed in \cite{ait2024multi}. Finally, our analysis provides a foundation for future work on higher-order multistep FPs and on nonlinear problems \cite{linear_conv}; consistent with this, our experiments indicate that the proposed update performs well in both settings.

The paper is organized as follows. In Section \ref{sec:multistep parareal}, we describe the parareal algorithm and its multistep 
variant, followed by the introduction of our proposed update scheme. In Section \ref{sec:conv}, we derive a convergence estimate for the 
F-multistep parareal method when the stability function of the CPs is invertible and illustrate and compare our proposed scheme with the 
original one using two examples. We then demonstrate the connection between the F-multistep parareal and the plain parareal when the coarsening 
factor $J$ is large in Sections \ref{subsec:exam_inv} and \ref{sec:connection}. In Section \ref{sec:numerical}, we illustrate the F-multistep parareal on a linear parabolic problem and a semilinear 
parabolic problem. Throughout, let
$A$ be a positive definite, self-adjoint, linear operator on a Hilbert space $(H,(\cdot, \cdot))$ with domain $D(A)$, and a compact inverse 
$A^{-1}$. The notation  $(\cdot,\cdot)$ denotes the inner product of the $H$ norm.

\section{Multistep variant of the parareal algorithm}\label{sec:multistep parareal}
Let $T >0$ be a fixed terminal time. Let $f:[0,T] \rightarrow H$ be a given forcing term.
Let $\{(\lambda_p, \varphi_p)\}_{p=1}^\infty$ denote the eigenpairs of $A$ with $(\varphi_p)_{p=1}^\infty$ forming an orthonormal basis of $H$. 
For a given initial value $u_0\in H$, we consider the following initial value problem for $u \in C((0,T];D(A))\cap C([0,T];H)$:
\begin{equation}\label{eqn:pde}
\left \{
\begin{split}
&u' (t)+ A u(t) = f(t), \quad 0<t<T,\\
&u(0)=u_0 .
\end{split}
\right .
\end{equation}

Let $N,J,N_c \in \mathbb{N}$ be such that $N = N_c J$. To discretize problem \eqref{eqn:pde} in time by a parareal algorithm, we consider two 
uniform partitions of the interval $[0,T],$ namely, a coarse one, $\Tau_n = n \Tau, n=0,\dotsc,N_c,$ with the coarse time step $\Tau = 
T/N_c,$ and a fine one, $t_n = n\tau, n=0,\dotsc,N,$ with the fine time step $\tau = T/N.$ Obviously, $\Tau = J\tau.$

Now we present the parareal algorithm for the initial value problem \eqref{eqn:pde}. Typically, the CP $G$ is an 
inexpensive, low-order numerical method, whereas the FP $F$ is a high-order but time-consuming integrator. In most work on 
the parareal algorithm, both $G$ and $F$ are single-step schemes. Given the initial data $v\in H$ and $f\in C([0,T];H)$, the CP, denoted by $ 
G(t+\bm{\tau},t,v)$, evolves the initial state $v$ from time $t$ to $t+\bm{\tau}$. Similarly, the FP, represented by $F(t+\tau , t,v)$, 
evolves a starting value $v$ from time $t$ to $t + \tau $. For details on the development, analysis, and implementation of the parareal algorithm, we refer to \cite{gander2025time}.

For linear parabolic problems, the study of the parareal algorithm is well established; see, e.g., \cite{GanderVandewalle:2007,WuZhou:2015}. With the exact solver as the FP and the stability function $R$ of  the CP, we define the convergence function $\gamma$ and the convergence factor $\gl$ by
\begin{equation}\label{eqn:gamma}
\gamma(s)=\frac{\e^{-s}-R(s)}{1-|R(s)|} \quad\mbox{and}\quad \gl =\sup_{s\in \mathbb{R}^+} |\gamma(s)|.
\end{equation}

\subsection{Parareal algorithm with multistep FPs}

Next, we discuss the development of the multistep variant of the parareal algorithm. Here, we adopt the notation from \cite{ait2024multi}, 
where Ait-Ameur and Maday proposed a multistep version of the parareal algorithm and established its convergence for nonlinear ODEs. Unlike single-step methods, linear multistep methods require multiple starting values. Let $F$ denote a $q$-step FP, such
that for any given time points $t_1$ and $t_2$, and $q$ starting approximations $w^1, w^2, \dotsc, w^q \in H$ at time levels $t_2 - (q-
1)\tau,t_2 -(q-2)\tau,\ldots, t_2$, the operator $F(t_1, t_2, w^1, \dots, w^q)$ propagates these approximations to time $t_1$. Here, $\tau $ 
represents the time step of the fine propagator $F$.

Throughout, we assume that the CPs $G$ are single-step methods. Specifically, we shall develop CPs $G_{i}$ with $i=1,\dotsc,q$, where $G_i$ is 
used to update the correction of initialization at $q$ fine time grid points $\Tau_n - (i-1)\tau, i=q-1,\dotsc,0$. In particular, for any time 
grid $t$ and function $w \in H$, the operator ${G}_i(t, w)$ propagates a single starting value $w$ at time $t$ to time $t+\bm{\tau} - i\tau, i=q-1,\dotsc,0$.

Then, we introduce the following multistep variant of the parareal algorithm. In particular, we set $U_n^k := U_{n,0}^k$.
\begin{algorithm}[htbp!]
\center
\caption{F-multistep parareal algorithm}
\begin{algorithmic}[1]\label{alg:m_para}
\STATE 
\textbf{Initialization}:
Compute $U^0_{n+1,0} = G_{1}( \Tau_{n}, U^0_{n})$ with $U_{0}^0 =u_0$, $n =0,1,\dotsc,N_c - 1$.

\STATE Set $U_{n+1,-i}^0 = U_{n+1}^0$ for $0\leqslant n \leqslant N_c -1$ and $i=1,\dotsc,q-1$.

\STATE Initialize with starting approximations $U_{0,-i}^0=u_{-i}, i=1,\dotsc,q-1$.
\FOR {$k=0,1,\ldots,K$}
\STATE \textbf{Parfor}: On each subinterval $[\Tau_n,\Tau_{n+1}]$, sequentially compute for $i=q-1,\dots,1,0$:
\begin{equation*}
v_{n,-i} = F(\Tau_{n+1}-i\tau ,\Tau_n, U_{n, -(q-1)}^{k},\dots,U_n^k).
\end{equation*}
\vskip5pt
\STATE Perform sequential corrections, i.e., compute $U^{k+1}_{n+1,-i}$ for $i=0,1,\dotsc,q-1$ by
\begin{equation}\label{eqn:multi-parareal}
U^{k+1}_{n+1,-i} = G_{i+1}( \Tau_{n}, U^{k+1}_{n})  + v_{n,-i} -  G_{i+1}(\Tau_n, U^{k}_{n}), 
\end{equation}
with $U^{k+1}_{0,-i} = u_{-i}$ for $i=0,\dotsc,q-1$ and $n =  0,1,\dotsc,N_c - 1$.\vskip5pt
\STATE Check the stopping criterion.
\ENDFOR
\end{algorithmic}
\end{algorithm}

In \cite{ait2024multi,ait2018multi},  the CPs $G_i$ are chosen as 
\begin{equation}\label{eqn:CP-I}
\text{Type (I):}\quad  {G}_i(t,w) = {G}({t+\bm{\tau},t,w}),\quad \text{for}~~ i = 1, \dots, q,
\end{equation}
where $G$ is a given single-step CP. Then, each iteration in Algorithm \ref{alg:m_para} can be written as 
\begin{equation*}
\begin{cases}
\begin{aligned}
U_{n+1}^{k+1} &= G(\Tau_{n+1}, \Tau_n, U_n^{k+1}) - G(\Tau_{n+1}, \Tau_n, U_n^{k}) \\
&\quad + F(\Tau_{n+1}, \Tau_n, U_{n,-(q-1)}^{k}, \dots, U_{n,0}^{k}), && 0 \leqslant n \leqslant N_c - 1, \\[6pt]
U_{n+1,-i}^{k+1} &= F(\Tau_{n+1} - i\tau , T_n, U_{n,-(q-1)}^{k}, \dots, U_{n,0}^{k}) + U_{n+1}^{k+1} &&1 \leqslant i \leqslant q-1, \\
&\quad - F(\Tau_{n+1}, T_n, U_{n,-(q-1)}^{k}, \dots, U_{n,0}^{k}), &&  0 \leqslant n \leqslant N_c - 1.
\end{aligned}
\end{cases}
\end{equation*}
Notably, in this algorithm, the term $G(\Tau_{n+1}, \Tau_n, U_n^{k+1}) - G(\Tau_{n+1}, \Tau_n, U_n^k)$ acts as a correction for all subsequent 
values $U_{n+1,-i}^{k+1}$, $i=0 ,\dots,q-1$. However, in Section \ref{subsec:exam_inv}, we show that this correction introduces extra consistency 
error, which affects the convergence rate. This effect becomes more pronounced when high-resolution CPs are used.

To improve the performance of the parareal algorithm, we propose a novel type of CPs:
\begin{equation}\label{eqn:CP-II}
\text{Type (II):}\quad {G}_i(t,w) = {G}(t+\bm{\tau} - (i-1)\tau , t, w),\quad \text{for}~~ i = 1, \dots, q,
\end{equation}
where ${G}_i$ is a single-step scheme that solves  equation \eqref{eqn:pde} at $t+\bm{\tau} - (i-1)\tau $, starting from the initial data 
$u(t) = w$. Unlike Type (I) CPs, the proposed Type (II) CPs require performing $q$ corrections, $G_i(\Tau_n, U_n^{k+1}) - G_i(\Tau_n, U_n^k)$ 
for $i = 1, 2, \dots, q$, within each coarse time interval. Nevertheless, these corrections can still be computed in parallel, preserving the 
efficiency of the method.

In Section \ref{sec:conv}, we present a detailed convergence analysis that demonstrates that the proposed Type (II) CPs outperform Type (I) 
CPs. Both types of F-multistep parareal schemes satisfy the consistency condition of the block iteration \cite{gander2023unified}, ensuring that 
the parareal solution satisfies $U_n^k = U_n$ when $k \geqslant n$. This property, known as the finite iteration convergence of parareal, guarantees the 
reliability of the approach. 

In this work, we employ BDF$q$ methods, one of the most widely used classes of linear multistep methods, as the FPs.  By \cite[Lemma 10.3]
{Thomee:2006}, we have the  explicit formula of the $q$-step method FP:
\begin{align}
&F \left( \Tau_{0}+n\tau,\Tau_{0}+(q-1)\tau,u_{0},\dotsc,u_{q-1} \right) \nonumber \\
&= \sum_{s=0}^{q-1} \beta_{ns} ( \tau A) u_{s} + \tau \sum_{j=q}^{n} \beta_{n-j} ( \tau A ) f( \Tau_{0}+j \tau) \nonumber \\
&=: \sum_{s=0}^{q-1} F_{{s+1}}^{\left( n-q+1 \right)}(\tau A) u_{s} + N^{( n-q+1)}( f)( \Tau_{0}+(q-1)\tau ),\label{eqn:fp}
\end{align}
where the \(\beta_j(\lambda)\) and \(\beta_{ns}(\lambda)\) are defined by
\begin{equation}\label{eqn:beta}
\widetilde{\beta}(\lambda)=\sum_{j = 0}^{\infty}\beta_j(\lambda)\zeta^j := (\widetilde{\alpha}(\zeta)+\lambda)^{-1},\, \beta_{ns}(\lambda)=-\sum_{j = q - s}^{q}\beta_{n - s - j}(\lambda)\alpha_j,
\end{equation}
with $\tilde{\alpha} ( \zeta) = \zeta^q \alpha (1/ \zeta)$, where $\alpha_j$ are the coefficients 
of the generating function $\alpha (\zeta)= \sum_{j=1}^q\frac{1}{j}(1-\zeta)^j=\sum_{j=0}^q\alpha_j\zeta^j$
of the $q$-step BDF method. The numerical differentiation is given by $ {\partial_{q}} U^{n}=\frac{1}{\tau} \sum_{j=0}^{q} \alpha_{j} U^{n-j}$.

We consider the following explicit formulas for the CPs for $i=1,\dots,q$,
\begin{equation}
\begin{aligned}
\text{Type (I):} \quad {G}_i(t,u) = R(\bm{\tau} A)u; \qquad
\text{Type (II):} \quad  {G}_i(t,u) = R((\bm{\tau} - (i-1)\tau )A)u,
\end{aligned}
\label{eqn:cp}
\end{equation}
where $R(s)$ is a stability function, an approximation to $\e^{-s}$. Then, we impose some assumptions on $R$ and $\gamma(s)$. 
\begin{assumption}\label{assum:1}
The rational function $R$ in~\eqref{eqn:cp} and the convergence function $\gamma(s)$ in~\eqref{eqn:gamma} satisfy the following conditions:
\begin{itemize}
    \item[(i)] $R(0)=1$ and $R'(0)=-1$. $R(s) \in (0,1)$ for all $s\in \mathbb{R}^+$.
    \item[(ii)]{ $|\gamma(s)|$ achieves its unique maximum at $s_0$ when $s \in [0,\infty)$ and $s_0 \neq 0$. }
\end{itemize}
\end{assumption}
We need this assumption to analyze the connection between the F-multistep parareal and the plain parareal in Theorem~\ref{thm:conv1}. Note that
the CPs only propagate the homogeneous part of the differential equation in \eqref{eqn:pde}, as the inhomogeneous part cancels out in the 
error analysis and does not influence the convergence of the parareal for the linear initial value problem \eqref{eqn:pde}. 
Since the homogeneous equation is autonomous, the variable $t$  does not affect ${G}_i(t,u)$,  and we simplify the 
notation 
as $G_i(u) = {G}_i(t,u)$. Moreover, we introduce $R_i (s)= R(s)$ for Type (I) and $R_i (s) = R({(J-(i-1))}/{J}s)$ for Type (II), for $s\in 
\mathbb{R}^+$ and $i=1,\dotsc,q$.

\section{Convergence analysis}\label{sec:conv} 
In this section, we analyze the convergence of the F-multistep parareal Algorithm \ref{alg:m_para} for the linear parabolic 
problem \eqref{eqn:pde} with nonsmooth data $u_0 \in H$. We assume that the FP is the BDF2: 
\begin{align*}
F\big( \Tau_{n+1},\Tau_{n},U_{n,-1}^{k},U_{n}^{k} \big) 
&=F_{1}^{( J )}( \tau A ) U_{n,-1}^{k}+F_{2}^{( J )}( \tau A ) U_{n}^{k}+N^{(J)}(f)(\Tau_n),\\
F\big( \Tau_{n+1}-\tau,\Tau_{n},U_{n,-1}^{k},U_{n}^{k} \big) 
&=F_{1}^{( J-1)}( \tau A) U_{n,-1}^{k}+F_{2}^{( J-1)}( \tau A)U_{n}^{k}+N^{(J-1)}(f)(\Tau_n), 
\end{align*}
where the operators $F_1^{(J)}$, $F_2^{(J)}$ and $N^{(J)}(f)$ can be computed through \eqref{eqn:fp} and \eqref{eqn:beta}. 

\subsection{Recursion derivation}\label{sec:recursion_1}
Next, we derive a recursive relation for the error of the parareal Algorithm \ref{alg:m_para}. Note that the $k$-th iteration of the algorithm
can be written as:
{\small \begin{equation}\label{eqn:m-para}
\left\{
\begin{split}
U_{n+1}^{k+1}&=G_{1}\big( U_{n}^{k+1} \big) +F_{1}^{(J)}\big( U_{n,-1}^{k} \big) +F_{2}^{(J)}\big( U_{n}^{k} \big)+N^{(J)}(f)(\Tau_n) -G_{1}\big( U_{n}^{k} \big), \\
U_{n+1,-1}^{k+1}&=G_{2}\big( U_{n}^{k+1} \big) +F_{1}^{(J-1)}\big( U_{n,-1}^{k} \big) +F_{2}^{(J-1)}\big( U_{n}^{k} \big) +N^{(J-1)}(f)(\Tau_n) -G_{2}\big( U_{n}^{k} \big).
\end{split}
\right.
\end{equation}}
Meanwhile, the numerical solution of the BDF2 scheme satisfies
\begin{equation*}
\left\{
\begin{aligned}
U_{n+1} &= F_{1}^{(J)}\big( U_{n,-1} \big) + F_{2}^{(J)}\big( U_{n} \big)+N^{(J)}(f)(\Tau_n) , \\
U_{n+1,-1} &= F_{1}^{(J-1)}\big( U_{n,-1} \big) + F_{2}^{(J-1)}\big( U_{n} \big)+N^{(J-1)}(f)(\Tau_n) .
\end{aligned}
\right.
\end{equation*} 
Thus, the errors $E_n^k := U_n^k - U_n$ and $E_{n,-1}^k := U_{n,-1}^k - U_{n,-1}$ satisfy the relation
\begin{equation}\label{eqn:para_err}
\left\{
\begin{aligned}
E_{n+1}^{k+1}&=G_{1}\big( E_{n}^{k+1} \big) +F_{1}^{(J)}\big( E_{n,-1}^{k} \big) +F_{2}^{(J)}\big( E_{n}^{k} \big) -G_{1}\big( E_{n}^{k} \big), \\
E_{n+1,-1}^{k+1}&=G_{2}\big( E_{n}^{k+1} \big) +F_{1}^{(J-1)}\big( E_{n,-1}^{k} \big) +F_{2}^{(J-1)}\big( E_{n}^{k} \big) -G_{2}\big( E_{n}^{k} \big).
\end{aligned}
\right.
\end{equation}

Throughout, we assume that the stability function $R$ is positive in $\mathbb{R}^+$; thus, $G_1$ and $G_2$ are invertible. 
We multiply the second equation in \eqref{eqn:para_err} by $G_1 G_2^{-1}$ and obtain
\begin{equation*}
G_{1}G_{2}^{-1}( E_{n+1,-1}^{k+1} ) =G_{1}( E_{n}^{k+1} ) +G_{1}G_{2}^{-1}F_{1}^{(J-1)}( E_{n,-1}^{k} ) +G_{1}G_{2}^{-1}F_{2}^{(J-1)}( E_{n}^{k} ) -G_{1}( E_{n}^{k} ).
\end{equation*}
Given that $F_1^{(J-1)},~F_2^{(J-1)}$, and $G_1 G_2^{-1}$ mutually commute, this can be rewritten as 
\begin{equation*}
G_{1}G_{2}^{-1}( E_{n+1,-1}^{k+1} ) =G_{1}( E_{n}^{k+1} )
    +F_{1}^{(J-1)}G_{1}G_{2}^{-1}( E_{n,-1}^{k} ) +F_{2}^{(J-1)}G_{1}G_{2}^{-1}( E_{n}^{k} ) -G_{1}( E_{n}^{k} ).
\end{equation*}
We further introduce a new variable, $ {D}_{n}^{k}=G_{1}G_{2}^{-1}\big( E_{n,-1}^{k} \big) -E_{n}^{k}$, which satisfies
\begin{equation} \label{eqn:g_0}
\begin{split}
 {D}_{n+1}^{k+1} &= \big( F_{1}^{(J-1)}G_{1}G_{2}^{-1}-F_{1}^{(J)} \big) ( E_{n,-1}^{k} ) 
 +\big( F_{2}^{(J-1)}{ G_{1}G_{2}^{-1}}-F_{2}^{(J)} \big) ( E_{n}^{k} )  \\
&= \big( F_{1}^{(J-1)}-F_{1}^{(J)}G_{2}G_{1}^{-1} \big) \big( G_{1}G_{2}^{-1}( E_{n,-1}^{k}) -E_{n}^{k} \big)   \\
&\quad +\big( F_{1}^{(J-1)}+F_{2}^{(J-1)}G_{1}G_{2}^{-1}-F_{1}^{(J)}G_{2}G_{1}^{-1}-F_{2}^{(J)} \big) ( E_{n}^{k} )   \\
& =: \eta_1 {D}_{n}^{k}+ \eta_2 E_{n}^{k}.
\end{split}
\end{equation}
Meanwhile, the error $E_{n+1}^{k+1}$ can also be related with ${D}_n^k$ through
\begin{align*}
E_{n+1}^{k+1} &= G_{1}( E_{n}^{k+1}) + F_{1}^{(J)}G_{2}G_{1}^{-1}\big( G_{1}G_{2}^{-1}( E_{n,-1}^{k}) - E_{n}^{k} \big) \\
 &\quad + \big( F_{1}^{(J)}G_{2}G_{1}^{-1} + F_{2}^{(J)} - G_1 \big) \big( E_{n}^{k} \big) \\
& =: G_{1} E_{n}^{k+1} +\eta_3 D_{n}^{k} + \eta_4 E_{n}^{k}.
\end{align*}
As a result, we obtain the following coupled equations between $E_n^k$ and $ {D}_n^k$:
\begin{equation}\label{eqn:coupled}
\left\{
\begin{aligned}
 {D}_{n+1}^{k+1} &= \eta_{1}  {D}_{n}^{k} + \eta_2 E_{n}^{k}, \\
E_{n+1}^{k+1} &= G_{1} E_{n}^{k+1} + \eta_3 D_{n}^{k} + \eta_4 E_{n}^{k}.
\end{aligned}
\right.
\end{equation}
The benefit of this transformation is that the recursive relations are now entirely formulated on the coarse time grids, simplifying the analysis significantly.

Next, recall that $\varphi_p$ is an eigenfunction of the operator $A $ corresponding to an eigenvalue $\lambda_p$. Let
 $e_{n,p}^k:=(E_{n}^k,\fy_p)$ and $d_{n,p}^k := (D_n^k,\fy_p)$. Testing \eqref{eqn:coupled} by $\fy_p$, we arrive at
 \begin{equation} \label{eqn:coupled_p}
    \left\{
\begin{aligned}
  {d}_{n+1,p}^{k+1}  &= \eta_{1,p} {d}_{n,p}^{k} + \eta_{2,p} e_{n,p}^{k}, \\
e_{n+1,p}^{k+1}  &= G_{1,p} e_{n,p}^{k+1} + \eta_{3,p} d_{n,p}^{k} + \eta_{4,p} e_{n,p}^{k},
\end{aligned}
\right.
\end{equation}
where we let $z_p = \tau  \lambda_p$ and define
\begin{align}
{\eta}_{1,p} &=F_{1}^{(J-1)}(z_{p})-F_{1}^{(J)}(z_{p})G_{2,p} G^{-1}_{1,p},~G_{1,p}=R_1(Jz_p),~G_{2,p}=R_2(Jz_p), \label{eqn:eta_1} \\
{\eta}_{2,{p}} &= F_{1}^{(J-1)}(z_p) + F_{2}^{(J-1)} (z_p)G_{1,p} G_{2,p}^{-1} - F_{1}^{(J)}(z_p) G_{2,p} G_{1,p}^{-1} - F_{2}^{(J)}
(z_p), \nonumber\\
{\eta}_{3,p} &= F_{1}^{(J)}(z_p) G_{2,p} G_{1,p}^{-1},\nonumber\\
{\eta}_{4,p}&= F_{1}^{(J)}(z_p) G_{2,p} G_{1,p}^{-1} + F_{2}^{(J)}(z_p) - G_{1,p}. \label{eqn:g_0_p}
\end{align}

\subsection{Error estimate}\label{ssec:Resursion derivation}
Next, we derive error estimates for the case $k \geqslant 3$. For the initial iterations with $k = 0, 1, 2$, the errors $\|E_n^k\|_H$ 
in \eqref{eqn:coupled_p}  remain bounded. As observed in the numerical experiments in Section~\ref{sec:numerical}, the error 
$\|E_n^k\|_H$ may initially increase (see, e.g., Fig.~\ref{fig:Ex1_nonsmooth}). However, as $k$ increases, the error decay stabilizes.

The following theorem shows a relation between $\max_i |e_{i,p}^{k+1}|$, $\max_i|e_{i,p}^{k-1}|$, and $\max_i |e_{i,p}^{k-3}|$.
\begin{theorem}\label{thm:err_recursion}
For $k\geqslant 3$, let $d_{n+1,p}^{k+1}$ and $e_{n+1,p}^{k+1}$ be as  in  \eqref{eqn:coupled_p}. Let $x_j = \max_{0\leqslant i\leqslant n-k+j}|e_{i,p}^{j}|$ 
and $y_j = \max_{0\leqslant i\leqslant n-k+j} |d_{i,p}^j|$. Then, the following estimates hold:
\begin{equation*} 
\left\{
\begin{split}
x_{k+1} &\leqslant \gamma_{a,p} x_{k-1} +  \gamma_{b,p} x_{k-3} + \gamma_{c,p} y_{k-3}, \\
y_{k-1} &\leqslant  \gamma_{d,p} x_{k-3} +  \gamma_{e,p} y_{k-3},
\end{split}
\right.
\end{equation*}
where the functions $ \gamma_{a,p}$, $ \gamma_{b,p}$, $ \gamma_{c,p}$, $ \gamma_{d,p}$, and $ \gamma_{e,p}$ 
are defined in \eqref{eqn:g_a}, \eqref{eqn:g_b}, \eqref{eqn:g_c}, \eqref{eqn:g_d}, and \eqref{eqn:g_e}, respectively.
\end{theorem}

\begin{proof}
We start from eliminating the index $k+1$ of the right-hand side of the second equation in \eqref{eqn:coupled_p} by iterating itself on $n$,
\begin{equation*}
e_{n+1,p}^{k+1}  = G_{1,p}^{n+1} e_{0,p}^{k+1}  + \sum_{i=0}^{n} G_{1,p}^{i}\big( \eta_{3,p} d_{n-i,p}^{k} + \eta_{4,p} e_{n-i,p}^{k} \big),
\end{equation*}
where $e_{0,p}^{k+1}=(U_{0}^{k+1},\fy_p)-(U_{0},\fy_p)=0$. Then, using the first relation 
in \eqref{eqn:coupled_p}, we eliminate $d_{i,p}^k$ on the right-hand side,
\begin{equation}\label{eqn:k-1}
\begin{split}
e_{n+1,p}^{k+1} &= {\eta}_{4,p}  \sum_{i=0}^{n} G_{1,p}^{i}e_{n-i,p}^{k} + {\eta}_{3,p} \sum_{i=0}^{n} G_{1,p}^{i}d_{n-i,p}^{k} \\
&= {\eta}_{4,p}  \sum_{i=0}^{n} G_{1,p}^{i}e_{n-i,p}^{k} + {\eta}_{3,p}  \sum_{i=0}^{n-1} G_{1,p}^{i}\big( \eta_{2,p} e_{n-i-1,p}^{k-1} + \eta_{1,p}  d_{n-i-1,p}^{k-1} \big),
\end{split}
\end{equation}
where $d_{0,p}^k = 0$.
Our next objective is to further expand $e_{i,p}^k$ on the right-hand side, to express $e_{n+1,p}^{k+1}$ in terms of values from iteration $k-1$. 
Terminating the expansion at this stage and applying the absolute norm would yield a non-sharp error bound. This conclusion is further 
corroborated by the numerical experiments presented in Section \ref{sec:numerical}.
We have the expression for $e_{n-i,p}^k$ by the second relation in \eqref{eqn:coupled_p},
\begin{align*}
e_{n-i,p}^{k} &= G_{1,p} e_{n-1-i,p}^{k} + {\eta}_{3,p} d_{n-1-i,p}^{k-1} + \eta_{4,p} e_{n-1-i,p}^{k-1} \\
&= G_{1,p}^{n-i} e_{0,p}^{k} + \sum_{j=0}^{n-1-i} G_{1,p}^{j} \big( \eta_{4,p} e_{n-1-i-j,p}^{k-1} + \eta_{3,p} d_{n-1-i-j,p}^{k-1} \big)\\
&= \sum_{j=0}^{n-1-i} G_{1,p}^{j} \big( \eta_{4,p} e_{n-1-i-j,p}^{k-1} + \eta_{3,p} d_{n-1-i-j,p}^{k-1} \big),
\end{align*}
where we again applied the fact that $e_{0,p}^k =0$. 
Then, we derive
\begin{equation}\label{eqn:e_d}
\begin{aligned}
e_{n+1,p}^{k+1} &= \eta_{4,p} \sum_{i=0}^{n} G_{1,p}^{i} \sum_{j=0}^{n-1-i} G_{1,p}^{j} \big( \eta_{4,p} e_{n-1-i-j,p}^{k-1} + \eta_{3,p}d_{n-1-i-j,p}^{k-1} \big)  \\
&\quad + {\eta}_{3,p} {\eta}_{2,p} \sum_{i=0}^{n-1} G_{1,p}^{i} e_{n-1-i,p}^{k-1} + {\eta}_{3,p} {\eta}_{1,p} \sum_{i=0}^{n-1} G_{1,p}^{i} d_{n-1-i,p}^{k-1}  \\
&= \sum_{m=0}^{n-1} G_{1,p}^{m} \big( (m+1)  (\eta_{4,p} )^2  + \eta_{3,p}  \eta_{2,p}  \big) e_{n-1-m,p}^{k-1}\\
&\quad + \sum_{m=0}^{n-1} G_{1,p}^{m} \eta_{3,p}  \big( (m+1) \eta_{4,p}  + \eta_{1,p} \big) d_{n-1-m,p}^{k-1}.
\end{aligned}
\end{equation}
Note that the right-hand side of \eqref{eqn:e_d} is only related to the iteration $k-1$. Now, we further expand $d_{n-1-m,p}^{k-1}$ 
to iteration $k-3$ based on \eqref{eqn:coupled_p}: 
\begin{align*}
d_{n-1-m,p}^{k-1} &= \eta_{2,p} e_{n-2-m,p}^{k-2} + \eta_{1,p} d_{n-2-m,p}^{k-2} \\
&= \eta_{2,p} \sum_{j=0}^{n-3-m} G_{1,p}^{j} \big( \eta_{4,p} e_{n-3-m-j,p}^{k-3} + \eta_{3,p} d_{n-3-m-j,p}^{k-3} \big) + \eta_{1,p} d_{n-2-m,p}^{k-2}.
\end{align*}
Then, we further expand $d_{n-2-m,p}^{k-2}$ to eliminate iteration $k-2$ on the right-hand side of the above relation, 
\begin{align*}
d_{n-1-m,p}^{k-1} &= \eta_{2,p} \sum_{j=0}^{n-3-m} G_{1,p}^{j} \big( \eta_{4,p} e_{n-3-m-j,p}^{k-3} +  \eta_{3,p} d_{n-3-m-j,p}^{k-3} \big) \\
&\quad + \eta_{1,p} \eta_{2,p} e_{n-3-m,p}^{k-3} + \eta_{1,p}^{2} d_{n-3-m,p}^{k-3} \\
&= \eta_{2,p} \eta_{4,p} \sum_{j=0}^{n-3-m} G_{1,p}^{j} e_{n-3-m-j,p}^{k-3} + \eta_{1,p} \eta_{2,p} e_{n-3-m,p}^{k-3} \\
&\quad + \eta_{2,p} \eta_{3,p} \sum_{j=0}^{n-3-m} G_{1,p}^{j} d_{n-3-m-j,p}^{k-3} + \eta_{1,p}^{2} d_{n-3-m,p}^{k-3}. 
\end{align*}
Finally, we take the expression of $d_{n-1-m,p}^{k-1}$ back to the expression of  $e_{n+1,p}^{k+1}$ in \eqref{eqn:e_d},  
{\small\begin{equation}\label{eqn:three}
\begin{aligned}
& e_{n+1,p}^{k+1} = \sum_{m=0}^{n-1} G_{1,p}^{m} \big( (m+1) \eta_{4,p}^{2} + \eta_{3,p} \eta_{2,p} \big) e_{n-1-m,p}^{k-1} \\
&\quad + \sum_{m=0}^{n-3} G_{1,p}^{m} \eta_{2,p}  \eta_{3,p} \big( (m+1) \eta_{4,p} + \eta_{1,p} \big) \big(\sum_{j=0}^{n-3-m} \eta_{4,p} G_{1,p}^{j} e_{n-3-m-j,p}^{k-3} + \eta_{1,p} e_{n-3-m,p}^{k-3} \big)   \\
&\quad + \sum_{m=0}^{n-3} G_{1,p}^{m}\eta_{3,p} \big( (m+1) \eta_{4,p} + \eta_{1,p} \big) \big(\sum_{j=0}^{n-3-m} \eta_{2,p} \eta_{3,p} G_{1,p}^{j} d_{n-3-m-j,p}^{k-3} + \eta_{1,p}^{2} d_{n-3-m,p}^{k-3} \big) \\
&=: {\rm I} + {\rm II} + {\rm III}.
\end{aligned}
\end{equation}}

We have decomposed the error $e_{n+1,p}^{k+1}$ into three sums, each analyzed separately. The first term corresponds 
to the error $e_{i,p}^{k-1}$ from iteration $k-1$. The second term involves the error $e_{i,p}^{k-3}$ from iteration $k-3$, 
while the third term accounts for the error $d_{i,p}^{k-3}$ from iteration $k-3$. 

\vspace{5pt}
\textbf{Step (I):} We estimate the first sum ${\rm I}$ in~\eqref{eqn:three}. This sum represents the primary component 
of the error $e_{n+1,p}^{k+1}$, as it incorporates the complete error propagation from the $(k-1)$-th iteration. The sum can be bounded as follows,
\begin{align*}
|{\rm I}|&\leqslant \max_{0\leqslant i\leqslant n-1} |e_{i,p}^{k-1}|\sum_{m=0}^{\infty} |G_{1,p}|^{m}\left| (m+1) \eta_{4,p}^{2} +\eta_{2,p} \eta_{3,p}\right| =: \gamma_{a,p} \max_{0\leqslant i\leqslant n-1} |e_{i,p}^{k-1}|,
\end{align*}
 where the convergence function $\gamma_{a,p}$ is defined as
\begin{equation}\label{eqn:g_a}
\gamma_{a,p} :=\sum_{m=0}^{\infty} |G_{1,p}|^{m}\left| (m+1) \eta_{4,p}^{2} +\eta_{2,p} \eta_{3,p}\right|.
\end{equation}

\textbf{Step (II):} We bound the second sum in \eqref{eqn:three}, 
\begin{align*}
|{\rm II}|&=\left| \eta_{2,p} \eta_{3,p} \eta_{4,p} \sum_{m=0}^{n-3} \sum_{j=0}^{n-3-m} G_{1,p}^{m+j}\left( \left( m+1 \right) \eta_{4,p} +\eta_{1,p} \right) e_{n-3-m-j,p}^{k-3} \right. \\
&\quad \left. + \eta_{1,p} \eta_{2,p} \eta_{3,p} \sum_{m=0}^{n-3} G_{1,p}^{m}\left( \left( m+1 \right) \eta_{4,p} +\eta_{1,p} \right) e_{n-3-m,p}^{k-3} \right| \\
&\leqslant \max_{0\leqslant i\leqslant n-3} |e_{i,p}^{k-3}| \left( |\eta_{2,p}\eta_{3,p} \eta_{4,p} |\sum_{m=0}^{n-3} \sum_{j=0}^{n-3-m} |G_{1,p}|^{m+j}|\left( m+1 \right) \eta_{4,p} +\eta_{1,p} | \right) \\
&\quad + \max_{0\leqslant i \leqslant n-3} |e_{i,p}^{k-3}| \left( |\eta_{1,p}\eta_{2,p} \eta_{3,p} |\sum_{m=0}^{n-3} |G_{1,p}|^{m}|\left( m+1 \right) \eta_{4,p} +\eta_{1,p} | \right).
\end{align*}

Differentiation of the relation $\sum_{m=0}^\infty x^m=1/(1-x)$ yields $\sum_{m=0}^\infty (m+1)x^m=1/(1-x)^2$ and thus
\begin{equation}\label{geom-series}
\sum_{m=0}^n (m+1)x^m\leqslant \frac 1{(1-x)^2},\quad 0\leqslant x<1.
\end{equation}
Since $|G_{1,p}|<1,$ using \eqref{geom-series}, 
we further simplify the sum above
\begin{align*}
  |{\rm II}|  &\leqslant \left( \frac{|\eta_{2,p} \eta_{3,p} \eta_{4,p}^{2} |}{\left( 1-|G_{1,p}| \right)^{3}} 
  + \frac{|\eta_{1,p} \eta_{2,p} \eta_{3,p} \eta_{4,p} |}{\left( 1-|G_{1,p}| \right)^{2}} \right) \max_{0\leqslant i\leqslant n-3} |e_{i}^{k-3}| \\
&\quad + \left( \frac{|\eta_{1,p} \eta_{2,p} \eta_{3,p} \eta_{4,p} |}{\left( 1-|G_{1,p}| \right)^{2}} 
+ \frac{|\eta_{1,p}^2 \eta_{2,p} \eta_{3,p} |}{1-|G_{1,p}|} \right) \max_{0\leqslant i \leqslant n-3} |e_{i}^{k-3}| \\
&\leqslant  \frac{|\eta_{2,p} \eta_{3,p}|}{1-|G_{1,p}|} \left( |\eta_{1,p} |+\frac{|\eta_{4,p} |}{1-|G_{1,p}|} \right)^{2} \max_{0\leqslant i\leqslant n-3} |e_{i}^{k-3}| 
=: \gamma_{b,p} \max_{0\leqslant i\leqslant n-3} |e_{i}^{k-3}|,
\end{align*}
where the convergence function $\gamma_{b,p}$ is defined as 
\begin{equation}\label{eqn:g_b}
\gamma_{b,p} :=\frac{|\eta_{2,p} \eta_{3,p}|}{1-|G_{1,p} |} \left( |\eta_{1,p}| +\frac{|\eta_{4,p}|}{1-|G_{1,p} |} \right)^{2}.
\end{equation}

\textbf{Step (III):} 
We derive a bound for the last sum ${\rm III}$ in \eqref{eqn:three}, 
\begin{align*}
|{\rm III}|&\leqslant \left| \eta_{2,p} \eta_{3,p}^{2} \sum_{m=0}^{n-3} \sum_{j=0}^{n-3-m} G_{1,p}^{m+j} \left( (m+1)\eta_{4,p} + \eta_{1,p} \right) d_{n-3-m-j,p}^{k-3} \right| \\
&\quad + \left| \eta_{1,p}^{2} \eta_{3,p} \sum_{m=0}^{n-3} G_{1,p}^{m} \left( (m+1)\eta_{4,p} + \eta_{1,p} \right) d_{n-3-m,p}^{k-3} \right|. 
\end{align*}
With $|G_{1,p}|<1$ and  \eqref{geom-series},
we obtain
\begin{align*}
|{\rm III}|&\leqslant  \left( \frac{|\eta_{2,p} \eta_{3,p}^{2} \eta_{4,p}|}{(1-|G_{1,p}|)^{3}} + \frac{|\eta_{1,p} \eta_{2,p} \eta_{3,p}^{2} |}{(1-|G_{1,p}|)^{2}} 
+ \frac{|\eta_{1,p}^{2} \eta_{3,p} \eta_{4,p} |}{(1-|G_{1,p}|)^{2}} + \frac{|\eta_{1,p}^{3} \eta_{3,p} |}{1-|G_{1,p}|} \right) \max_{0\leqslant i \leqslant n-3} |d_{i,p}^{k-3}|\\
&=: \gamma_{c,p}  \max_{0\leqslant i \leqslant n-3} |d_{i,p}^{k-3}|,
\end{align*}
where the convergence factor $\gamma_{c,p}$ is defined as 
\begin{equation}\label{eqn:g_c}
\gamma_{c,p} :=\frac{|\eta_{2,p} \eta_{3,p}^{2} \eta_{4,p}|}{(1-|G_{1,p}|)^{3}} + \frac{|\eta_{1,p} \eta_{2,p} \eta_{3,p}^{2} |}{(1-|G_{1,p}|)^{2}} 
+ \frac{|\eta_{1,p}^{2} \eta_{3,p} \eta_{4,p} |}{(1-|G_{1,p}|)^{2}} + \frac{|\eta_{1,p}^{3} \eta_{3,p} |}{1-|G_{1,p}|}.
\end{equation}

{\textbf{Step (IV):}}
We now turn to the misalignment-induced error $d_{i,p}^{k-3}$, which is present in the last sum in \eqref{eqn:three}. 
Our approach involves expressing $d_{i,p}^{k+1}$ in terms of $d_{i,p}^{k-1}$ and $e_{i,p}^{k-1}$ from \eqref{eqn:coupled_p}; this enables 
us to establish an upper bound for $d_{i,p}^{k+1}$. We start from \eqref{eqn:coupled},
\begin{align*}
d_{n+1,p}^{k+1} &= \eta_{2,p} e_{n,p}^{k} + \eta_{1,p} d_{n,p}^{k} \\
&= \eta_{2,p} \big( G_{1,p} e_{n-1,p}^{k} + \eta_{4,p} e_{n-1,p}^{k-1} +  \eta_{3,p} d_{n-1,p}^{k-1} \big) + \eta_{1,p} d_{n,p}^{k} \\
&= \eta_{2,p} \sum_{j=0}^{n-1} G_{1,p}^{j} \big( \eta_{4,p} e_{n-1-j,p}^{k-1} +  \eta_{3,p} d_{n-1-j,p}^{k-1} \big) +\eta_{1,p} d_{n,p}^{k}.
\end{align*}
Then, we further expand $d_{n,p}^k$ with $d_{n,p}^{k}=\eta_{2,p} e_{n-1,p}^{k-1}+\eta_{1,p} d_{n-1,p}^{k-1}$,
\begin{align*}
d_{n+1,p}^{k+1} &= \eta_{2,p} \sum_{j=0}^{n-1} G_{1,p}^{j}\big(\eta_{4,p} e_{n-1-j,p}^{k-1} +  \eta_{3,p} d_{n-1-j,p}^{k-1}\big) 
+ \eta_{1,p}\big(\eta_{2,p} e_{n-1,p}^{k-1} + \eta_{1,p} d_{n-1,p}^{k-1}\big) \\
&= \eta_{2,p} \eta_{4,p} \sum_{j=1}^{n-1} G_{1,p}^{j} e_{n-1-j,p}^{k-1} + \eta_{2,p} (\eta_{1,p}+\eta_{4,p}) e_{n-1,p}^{k-1} \\
&\quad + \eta_{2,p} \eta_{3,p} \sum_{j=1}^{n-1} G_{1,p}^{j}d_{n-1-j,p }^{k-1} + (\eta_{2,p} \eta_{3,p}+\eta_{1,p}^{2}) d_{n-1,p}^{k-1}.
\end{align*}
We take the absolute norm on both sides and derive the bound for $d_{n+1,p}^{k+1}$, 
\begin{align*}
|d_{n+1,p}^{k+1}| &\leqslant \left( \frac{|G_{1,p} \eta_{2,p} \eta_{4,p}|}{1 - |G_{1,p}|} + |\eta_{2,p} (\eta_{4,p} + \eta_{1,p})| \right) \max_{0 \leqslant i\leqslant n-1} |e_{i}^{k-1}| \\
&\quad + \left( \frac{|G_{1,p} \eta_{2,p} \eta_{3,p}|}{1 - |G_{1,p}|} + |\eta_{2,p}  \eta_{3,p} + \eta_{1,p}^{2}| \right) \max_{0\leqslant i\leqslant n-1} |d_{i,p}^{k-1}|\\
&=: \gamma_{d,p}\max_{0\leqslant i\leqslant n-1}|e_i^{k-1}| + \gamma_{e,p} \max_{0 \leqslant i\leqslant n-1} |d_i^{k-1}|,
\end{align*}
where the convergence factors $\gamma_{d,p}$ and $\gamma_{e,p}$ are defined as 
\begin{align}
\gamma_{d,p} &:=\frac{|G_{1,p} \eta_{2,p}\eta_{4,p} |}{1-|G_{1,p} |} +|\eta_{2,p}\left( \eta_{4,p} +\eta_{1,p} \right) |,\label{eqn:g_d}\\
\gamma_{e,p} &:=\frac{|G_{1,p} \eta_{2,p} \eta_{3,p} |}{1-|G_{1,p} |} +|\eta_{2,p} \eta_{3,p} +\eta_{1,p}^{2} |.\label{eqn:g_e}
\end{align}
Note that the first term in \eqref{eqn:g_c} becomes unbounded when $z_p=0$ due to the denominator. 

\vspace{5pt}
Finally, we arrive at the iteration on index $k$, based on the estimates of the three terms: 
\begin{align*}
\left\{
\begin{aligned}
|e_{n+1,p}^{k+1}| &\leqslant \gamma_{a,p} \max_{0 \leqslant i \leqslant n-1} |e_{i,p}^{k-1}| 
+ \gamma_{b,p} \max_{0\leqslant i \leqslant n-3} |e_{i,p}^{k-3}| + \gamma_{c,p} \max_{0 \leqslant i\leqslant n-3} |d_{i,p}^{k-3}|, \\
|d_{n,p}^{k-1}|   &\leqslant \gamma_{d,p} \max_{0 \leqslant i\leqslant n-3} |e_{i,p}^{k-3}| + \gamma_{e,p} \max_{0\leqslant i \leqslant n-3} |d_{i,p}^{k-3}|,
\end{aligned}
\right.
\end{align*}
where the left-hand side of the first inequality is independent of $n$; then, we take the maximum over $n$. 
We recall that $x_j = \max_{0\leqslant i\leqslant n-k+j}|e_{i,p}^{j}|$ and $y_j = \max_{0\leqslant i\leqslant n-k+j} |d_{i,p}^j|$, and obtain
\begin{align}\label{eqn:bound}
\left\{
\begin{aligned}
x_{k+1,p} &\leqslant \gamma_{a,p} x_{k-1,p} + \gamma_{b,p} x_{k-3,p} + \gamma_{c,p} y_{k-3,p}, \\
y_{k-1,p} &\leqslant \gamma_{d,p} x_{k-3,p} + \gamma_{e,p} y_{k-3,p}.
\end{aligned}
\right.
\end{align}
The proof is  complete. 
\end{proof}

\begin{remark}\label{rmk:n_0}
Note that the summation in \eqref{eqn:g_a} involves infinitely many terms and is thus not directly computable. To obtain an accurate approximation of $\gamma_{a,p}$, we compute the first $n_0$ terms explicitly and bound the tail of the series as follows:
\begin{align*}
\gamma_{a,p}&= \Big(\sum_{m=0}^{n_0}+\sum_{m=n_0+1}^{\infty}\Big) |G_{1,p}|^{m}\left| (m+1) \eta_{4,p}^{2} +\eta_{3,p} \eta_{2,p}\right|\\
&\leqslant  \sum_{m=0}^{n_0} |G_{1,p}| ^{m}\left| (m+1) \eta_{4,p}^{2} +\eta_{3,p} \eta_{2,p}\right| + \frac{|G_{1,p}|^{n_0+1}|\eta_{3,p}\eta_{2,p}|}{1-|G_{1,p}|} \\
&\quad+\eta_{4,p}^2\frac{(n_0+2)|G_{1,p}|^{n_0+1}-(n_0+1)|G_{1,p}|^{n_0+2}}{(1-|G_{1,p}|)^2}.
\end{align*}
We retain the first $n_0$ terms to ensure a sharp estimate, while noting that $\gamma_{a,p}$ is defined independently of $n_0$. In the numerical examples presented in Section \ref{subsec:exam_inv}, we set $n_0 = 50$.
\end{remark}

Inequality \eqref{eqn:bound} leads to the following system of inequalities:
\begin{equation}\label{eqn:system}
\begin{pmatrix} 
        x_{2m,p} \\ 
        x_{2m-2,p} \\ 
        y_{2m-2,p} 
    \end{pmatrix} 
    \leqslant 
    \begin{pmatrix} 
        \gamma_{a,p} & \gamma_{b,p} & \gamma_{c,p} \\ 
        1 & 0 & 0 \\ 
        0 & \gamma_{d,p} & \gamma_{e,p} 
    \end{pmatrix} 
    \begin{pmatrix} 
        x_{2m-2,p} \\ 
        x_{2m-4,p} \\ 
        y_{2m-4,p} 
    \end{pmatrix},
\end{equation}
which can be compactly written as $\mathbf{S}_{m,p} \leqslant \mathbf{A}_p \mathbf{S}_{m-1,p}$. We next analyze the matrix $\textbf{A}_p$. 
Let the eigenvalues of $\mathbf{A}_p$ be denoted by $\xi_{\ell,p}$, for $\ell = 1, 2, 3$; they satisfy the cubic equation
\begin{equation}\label{eqn:cubic}
\xi_{\ell,p}^{3} -( \gamma_{a,p} +\gamma_{e,p}) \xi_{\ell,p}^{2} +( \gamma_{a,p} \gamma_{e,p} -\gamma_{b,p}) \xi_{\ell,p} 
+( \gamma_{b,p} \gamma_{e,p} -\gamma_{c,p} \gamma_{d,p} ) =0.
\end{equation}
Then, we define the \textit{convergence function} as  
\begin{equation}\label{eqn:gazJ}
\gamma ( z_{p},J ) =\max_{\ell \in \left\{ 1,2,3 \right\}} |\xi_{\ell,p}|
\end{equation}
and the \textit{convergence factor} as $\gamma^{\dag} ( J ) =\sup_{p\in \mathbb{N}^{+}} \gamma ( z_{p},J )$. 
Note that $\gamma(z_p, J)$ depends on two components involving $J$: the FPs $F_1^{(J)}(z_p)$ and $F_2^{(J)}(z_p)$, and the CPs $G_i$. 

\begin{theorem}\label{thm:main}
When $k$ is even, the following estimate holds, for some constant $C$ independent of $n,k$ and $\gamma^\dag$, 
\begin{equation}\label{eqn:specical_thm}
\| E_{n}^{k}\|_H\leqslant C(  {\gamma}^{\dag} ( J ) )^{k/2-2} \sum_{i=0}^{n} \left( \| E_{i}^{2}\|_H 
+\| E_{i}^{0}\|_H +\| E_{i,-1}^{0}\|_H \right).
\end{equation} 
\end{theorem}

\begin{proof} 
We assume that the matrix $\mathbf{A}_p$ defined in \eqref{eqn:system} has the decomposition $\mathbf{A}_p = \mathbf{P}_p^{-1} \mathbf{D}_p \mathbf{P}_p$, where the matrix $\mathbf{D}_p$ is the upper Jordan matrix with the largest absolute eigenvalue $ {\gamma} (z_p,J)$. Then, system \eqref{eqn:system} becomes 
    \begin{equation*}
    \mathbf{P}_p \mathbf{S}_{m,p}\leqslant \mathbf{D}_p \mathbf{P}_p \mathbf{S}_{m-1,p}\leqslant \mathbf{D}_p^{m-2}\mathbf{P}_p\mathbf{S}_{2,p}.
    \end{equation*}
Taking the 2-norm on both sides, we obtain
\begin{equation*}
\| \mathbf{S}_{m,p}\|_{2} \leqslant C\| \mathbf{P}_p\mathbf{S}_{m,p}\|_{2} \leqslant C\| \mathbf{D}_p^{m-2}\|_{2} \| \mathbf{P}_p\mathbf{S}_{2,p}\|_{2} 
\leqslant C\left( {\gamma}^{\dag} (J) \right)^{m-2}\| \mathbf{S}_{2,p}\|_{2}.
\end{equation*}
Then, we estimate the $H$-norm of the parareal error,
\begin{align*}
&\quad \| E_{n+2m-k}^{2m}\|_H^{2} 
= \sum_{p=1}^{\infty} \left( e_{n+2m-k,p}^{2m} \right)^{2} \leqslant \sum_{p=1}^{\infty} |x_{2m,p}|^{2} \leqslant \sum_{p=1}^{\infty} \| \mathbf{S}_{m,p} \|_{2}^{2}\\
&\leqslant C\left(  {\gamma}^{\dag} \left( J \right) \right)^{2m-4} \sum_{p=1}^{\infty} \| \mathbf{S}_{2,p} \|_{2}^{2} \leqslant C\left( {\gamma}^{\dag} \left( J \right) \right)^{2m-4} \sum_{p=1}^{\infty} \left( |x_{2,p}|^{2}+|x_{0,p}|^{2}+|y_{0,p}|^{2} \right) \\
&\leqslant C\left(  {\gamma}^{\dag} \left( J \right) \right)^{2m-4} \sum_{i=0}^{n} \left( \| E_{i}^{2}\|_H^{2} +3\| E_{i}^{0}\|_H^{2} +2\| E_{i,-1}^{0}\|_H^{2} \right). 
\end{align*}
By the elementary inequality $\sqrt{\sum_{i=0}^{n} a_{i}^{2}} \leqslant \sum_{i=0}^{n} |a_{i}|$, we obtain
\begin{equation*}
\| E_{n+2m-k}^{2m}\|_H \leqslant C\left(  {\gamma}^{\dag} \left( J \right) \right)^{m-2} \sum_{i=0}^{n} \left( \| E_{i}^{2}\|_H
+\| E_{i}^{0}\|_H+\| E_{i,-1}^{0}\|_H \right).
\end{equation*}
Since $k$ is even, we take $m = k/2$ and thus 
\begin{equation*}
\| E_{n}^{k}\|_H \leqslant C\left(  {\gamma}^{\dag} \left( J \right) \right)^{k/2-2} \sum_{i=0}^{n} 
\left( \| E_{i}^{2}\|_H +\| E_{i}^{0}\|_H +\| E_{i,-1}^{0}\|_H \right).
\end{equation*}
This completes the proof of the Theorem.
\end{proof}

\section{Discussion on the convergence properties}
\label{subsec:exam_inv}

To illustrate the convergence in \eqref{eqn:specical_thm}, we consider four CPs, defined in \eqref{eqn:cp}: the backward Euler method 
(BE), the two-stage Lobatto IIIC (LIIIC2), the double two-stage Lobatto IIIC (LIIIC2(2)), and the exact solver, given by
\begin{equation*}
\begin{aligned}
\text{(BE):}~R(s) &= \frac{1}{1+s} , &\quad \text{(LIIIC2):}~R(s)&=\frac{2}{s^2+2s+2},\\
\text{(LIIIC2(2)):}~R(s)&=\left(\frac{2}{({s/2})^2+s+2}\right)^2,&\quad \text{(Exact):}~R(s) &= \e^{-s}.
\end{aligned}
\end{equation*}
The nontation LIIIC2(2) denotes applying LIIIC2 twice with half the step size, which is a 
typical choice for a high-resolution CP to achieve fast convergence \cite{dai2013stable}.

Recall that we apply the spectral decomposition to \eqref{eqn:coupled} and define $z_p = \tau  \lambda_p$ in the discrete setting. Note that 
in practice, $\tau $ is small, while $\lambda_p \to \infty$ as $p \to \infty$. Thus, for convenience, we consider the continuous spectrum 
case, i.e., we take $z=z_p\in \mathbb{R}^+$ and define $\eta_1 (z,J) = F_1^{(J-1)} (z) -F_1^{(J)}(z) R_2(Jz) R_1 (Jz),$ according to \eqref{eqn:eta_1}. Similarly, we define $\eta_2,\eta_3$ and $\eta_4$, as well as $\gamma_a,\gamma_b,\gamma_c,\gamma_d$ and $\gamma_e$.

\subsection{Comparison of \texorpdfstring{${\gamma}_{a}$}{gamma	extunderscore a}}
In Fig.~\ref{fig:comparison_g_a}, we draw the graph of ${\gamma}_{a}$ for the four 
cases. When we apply the BE, for both Type (I) and Type (II), the supremum of ${\gamma}_a (z,J)$ increases as the coarsening factor $J$ increases. The values of $\sup_{z \in \mathbb{R}^+}{\gamma}_a (z,J)$ are nearly identical for both types. 

A key distinction lies in the behavior of the functions: Type (II) functions approach zero at $z=0$, whereas Type (I) functions do not. 
However, for other CP solvers, $\sup_{z\in \mathbb{R}^+} {\gamma}_a(z,J)$ for Type (II) is significantly smaller than for Type (I). 
This suggests that the Type (II) update is more effective, as CPs typically approximate the exact solver. This can be explained from 
the main component $ {\eta}_{4,p}$ of \eqref{eqn:g_0_p} in $ {\gamma}_{a,p}$ in \eqref{eqn:g_a} when the CP is the exact solver: 
\begin{align*}
\text{Type (I):}~& {\eta}_{4,p} ( z_p ) = F_{1,p}^{(J)}( z_p ) + F_{2,p}^{(J)}( z_p ) - \e^{ -J z_p }; \\
\text{Type (II):}~& {\eta}_{4,p} ( z_p ) = F_{1,p}^{(J)}( z_p ) \e^{ z_p } + F_{2,p}^{(J)}( z_p ) - \e^{ -J z_p }.
\end{align*}
Type (II) is more reasonable, as $F_{1,p}^{(J)}$ and $F_{2,p}^{(J)}$ should not occupy the same position. 
Specifically, $F_{1,p}^{(J)} \e^{z_p}$ precedes $F_{2,p}^{(J)}$ by one position. Moreover, as the CPs provide a better approximation to the exact solver, the value of $\sup_{z \in \mathbb{R}^+} \gamma_a(z, J)$ becomes smaller.

\begin{figure}[htbp!]
  \centering
  \includegraphics[width=0.98\textwidth,trim={0.7cm 1cm 1.5cm 1cm},clip]{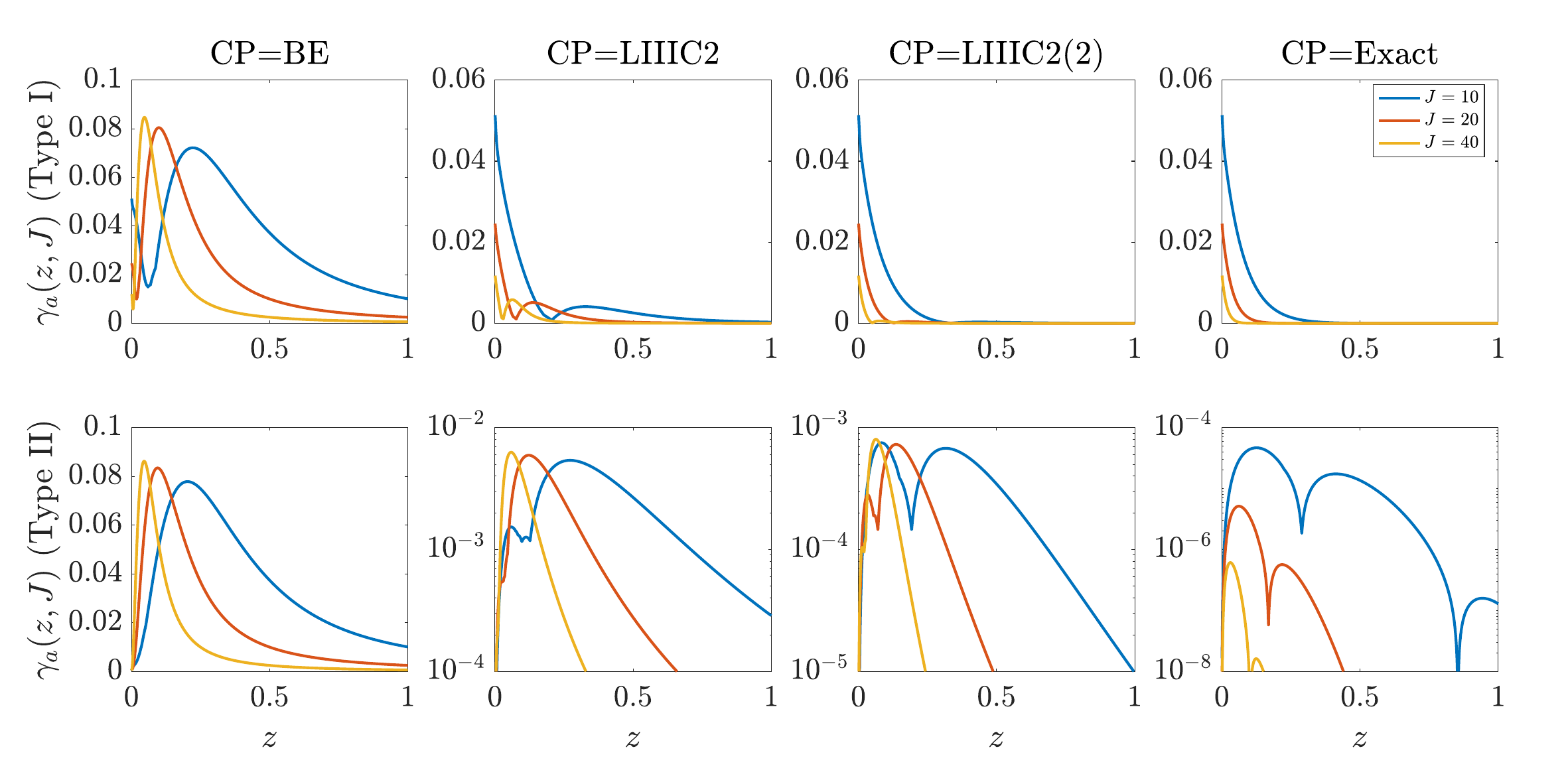}
    \caption{Comparison of the graph of ${\gamma}_a (z,J)$ in \eqref{eqn:g_a}  for four CPs of two types.}
    \label{fig:comparison_g_a}
\end{figure}

\subsection{Comparison of the convergence factor}

Fig.~\ref{fig:comparison_g} illustrates the graphs of ${\gamma}(z,J)$ for the same four solvers. Note that ${\gamma}(z,J)$ is slightly larger than $\gamma_a (z,J)$, but their behaviors are similar, as observed by comparing Fig.~\ref{fig:comparison_g} 
with Fig.~\ref{fig:comparison_g_a}. Table~\ref{tab:comparison} presents the values of ${\gamma}^\dag (J)$. As 
demonstrated in Section~\ref{sec:connection}, the convergence factor ${\gamma}^\dag$ approaches $\gl^2$ defined in \eqref{eqn:gamma}
as $J \to \infty$ for both Type (I) and (II). However, Type (II) already provides an excellent approximation to $\gl^2$ at relatively small 
values of $J$, which makes it particularly attractive for practical applications. 

\begin{figure}[htbp!]
  \centering
  \includegraphics[width=0.98\textwidth,trim={0.7cm 1cm 1.5cm 1cm},clip]{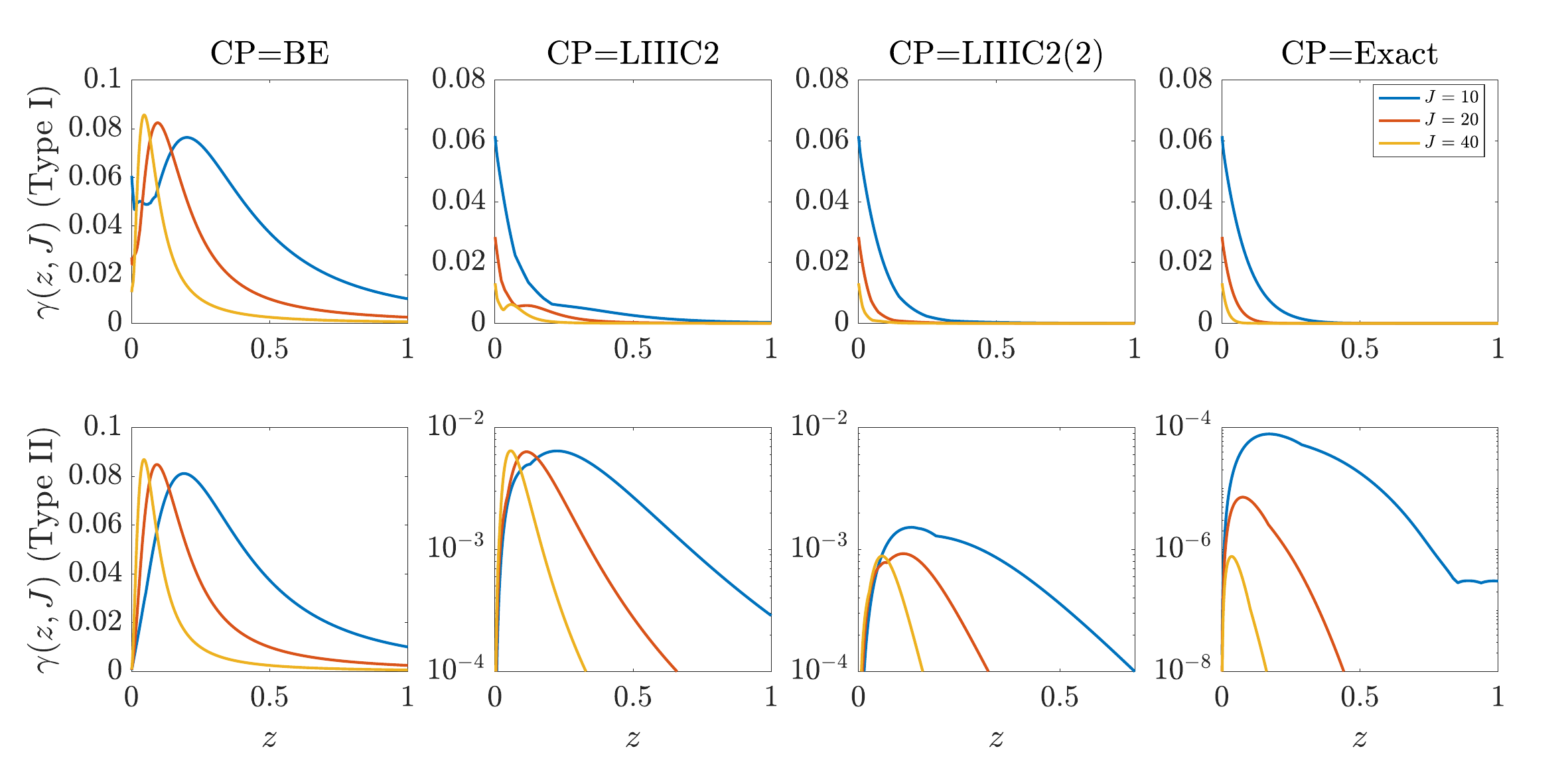}
    \caption{Comparison of the graph of ${\gamma} (z,J)$ in \eqref{eqn:gazJ} for four CPs of two types.}
    \label{fig:comparison_g}
\end{figure}

\begin{table}[htbp!]
\centering
\caption{ The value of ${\gamma}^\dag$ across different $J$ values for four solvers of two types.}
\label{tab:comparison}
\begin{tabular}{c|c|ccccc}
\toprule
\multirow{2}{*}{CP$(\gl^2)$} & \multirow{2}{*}{Type} & \multicolumn{5}{c}{$J$ } \\
\cmidrule{3-7}
 & & 10 & 20 & 40 & 60 & 120 \\
\midrule
{BE} & I &0.0764 &0.0824 &0.0857 &0.0868 & 0.0879 \\
\cmidrule{2-7}
 (0.0891)& II & 0.0812 &0.0849 &0.0869  & 0.0876 & 0.0883  \\
 \midrule
{LIIIC2} & I & 0.0625 & 0.0293 & 0.0140 & 0.00913 & 0.00649 \\
\cmidrule{2-7}
 (0.00668)& II & \textbf{0.00644} & \textbf{0.00634} & \textbf{0.00646}  & \textbf{0.00652} & \textbf{0.00660}   \\
  \midrule
{LIIIC2(2)} & I & 0.0625 & 0.0293 & 0.0140 & 0.00913 & 0.00444  \\
\cmidrule{2-7}
 (0.000906)& II & \textbf{0.00152} & \textbf{0.000926} & \textbf{0.000885}  & \textbf{0.000886} & \textbf{0.000893} \\
\midrule
{Exact} & I &0.0625 &0.0293 & 0.0140 & 0.00913 &0.00444   \\
\cmidrule{2-7}
 (0.)& II & \textbf{7.77e-5} & \textbf{7.22e-6} & \textbf{7.77e-7} & \textbf{2.15e-7} & \textbf{2.50e-8}  \\
\bottomrule
\end{tabular}
\end{table}

\section{Connection with plain parareal}\label{sec:connection}

In this section, we consider Dahlquist's test equation for $\lambda >0 $,
\begin{equation}\label{eqn:Test}
    u'(t) + \lambda u(t) = 0.
\end{equation}
The F-multistep parareal \ref{alg:m_para} employs the BDF2 method \eqref{eqn:fp} as FPs and the CPs defined in \eqref{eqn:cp}. We derive 
explicit expressions of the FPs $F_1^{(J)}$ and $F_2^{(J)}$. The fine solution $U_{n+1} = U_{n,J}$ satisfies the relation
\begin{equation*}
    U_{n,J}=-{(3+2\tau \lambda)}^{-1} U_{n,J-2}+{4}{(3+2\tau \lambda)}^{-1} U_{n,J-1}.
\end{equation*}
Using recursion techniques \eqref{eqn:fp}, we obtain
\begin{equation*}
U_{n,J}=\frac{r_{1}\left( z \right) r_{2}\left( z \right)}{r_{1}\left( z \right) -r_{2}\left( z \right)} \left( r_{2}^{J}\left( z \right) -r_{1}^{J}\left( z \right) \right) U_{n,-1}+\frac{r_{1}^{J+1}\left( z \right) -r_{2}^{J+1}\left( z \right)}{r_{1}\left( z \right) -r_{2}\left( z \right)} U_{n,0},
\end{equation*}
where $z=\tau \lambda$ and $r_1 (z), r_2 (z)$ are the characteristic roots,
\begin{equation*}
r_{1}\left( z \right) =\frac{2+\sqrt{1-2z}}{3+2z} ,\quad r_{2}\left( z \right) =\frac{2-\sqrt{1-2z}}{3+2z}.
\end{equation*}
Consequently, $F_1^{(J)}$ and $F_2^{(J)}$ are given by 
\begin{equation*}
    F_{1}^{\left( J \right)}\left( z \right) =\frac{r_{1}\left( z \right) r_{2}\left( z \right)}{r_{1}\left( z \right) -r_{2}\left( z \right)} \left( r_{2}^{J}\left( z \right) -r_{1}^{J}\left( z \right) \right) ,\quad F_{2}^{\left( J \right)}\left( z \right) =\frac{r_{1}^{J+1}\left( z \right) -r_{2}^{J+1}\left( z \right)}{r_{1}\left( z \right) -r_{2}\left( z \right)} .
\end{equation*}

Based on the error analysis in Section \ref{ssec:Resursion derivation} and equation \eqref{eqn:cubic}, we obtain the cubic equation 
for the ODE \eqref{eqn:Test} $(\ell=1,2,3)$,
\begin{equation}\label{eqn:cubic_ode}
\xi_{\ell}^{3} -\left( \gamma_{a} +\gamma_{e}\right) \xi_{\ell}^{2} +\left( \gamma_{a} \gamma_{e} -\gamma_{b}\right) \xi_{\ell} 
+\left( \gamma_{b} \gamma_{e} -\gamma_{c} \gamma_{d} \right) =0,
\end{equation}
where ${\gamma}_a(z,J), {\gamma}_b(z,J), {\gamma}_c(z,J), {\gamma}_d (z,J)$, and ${\gamma}_e 
(z,J)$ are defined through the combinations of ${\eta}_1 (z,J)$, ${\eta}_2 (z,J), 
{\eta}_3(z,J),{\eta}_4(z,J)$ in \eqref{eqn:g_a}, \eqref{eqn:g_b}, \eqref{eqn:g_c}, \eqref{eqn:g_d}, and \eqref{eqn:g_e}, 
respectively. We will show that ${\gamma}_e \rightarrow 0$, ${\gamma}_b \rightarrow 0$ and 
${\gamma}_c {\gamma}_d \rightarrow 0$ as $J \to \infty$ for all $z \in \mathbb{R}^+$ in Theorem~\ref{thm:lim}. Meanwhile, the 
principal convergence term satisfies $\gamma_a \rightarrow \left( \frac{\e^{-Jz}-R\left( Jz \right)}{1-|R\left( Jz \right) |} \right)^{2}$. Thus, in 
the limit $J \to \infty$, the cubic equation reduces to  
\begin{equation*}
    \xi_{\ell}^{3} - {\gamma}_{a} \xi_{\ell}^{2} =0,
\end{equation*}
and the convergence factor $\gamma^\dag (J)=\max_{\ell \in\{1,2,3\}}\sup_{z\in \mathbb{R}^+}|\xi_{\ell}(z,J)| \rightarrow \gl^2$ as $J 
\to \infty$. This indicates that as the coarsening factor $J$ increases, the convergence rate of the F-multistep parareal algorithm approaches that of 
the plain parareal.
Numerical experiments in Fig.~\ref{fig:comparison_g} and Table~\ref{tab:comparison} support this theoretical result, e.g., for the BE CPs, the convergence 
factor ${\gamma}^\dag (J) \rightarrow \gamma_{\text{lin}}^2 \approx 0.298^2$. For simplicity, we present the proof for CPs~(II) and the proof for CPs~(I) follows similarly. 

\subsection{Preliminary estimates}\label{sec:pre est}
We present several lemmas that will be used in the subsequent proofs. In practical applications, we focus on values of $J$ that are not too small to ensure high efficiency. Accordingly, we consider the estimates for $J\geqslant 10 $.

\begin{lemma}\label{lem:1}
There exists a constant $C>0$, independent of $J$ and $z$, such that, for $z \geqslant 0$ and $J\geqslant 10$, $\left| {R_2(Jz)}/{R_1(Jz)} - 1 \right| \leqslant {C}/{J}.$ Moreover, for any given $z_0 \in (0,1/2)$, there exists a constant $C_R>0$, independent of $J$ and $z$, such that, for all $z \leqslant z_0$ and $J\geqslant 10$,
\begin{equation}\label{eqn:e^z}
    \left| {R_2(Jz)}/{R_1(Jz)} - \e^z \right| \leqslant C_R z \quad \text{and} \quad \left| {R_1(Jz)}/{R_2(Jz)} - \e^{-z} \right| \leqslant C_R z.
\end{equation}
\end{lemma}

\begin{proof}
We first prove the first inequality. Recall that $R_i(s)=R\left( {(J-(i-1))} s/J \right)$ for $i=1,2$ and assume $R(z) = P(z)/ Q(z)$, with $\deg P = n \leqslant \deg Q = m$. Consider   ${R_2(Jz)}/{R_1(Jz)} = R(\alpha Jz)/{R(Jz)}$, where $\alpha = 1 -
{J}^{-1} \in [0.9,1)$. Let $w = Jz >0$. Let $r=r(\alpha,w)={R(\alpha w)}/{R(w)}$. Transform the domain
via $u = {w}/({1+w}) \in [0,1)$, so $w={u}/{(1-u)}$. We define $f(\alpha,u)=r\left(\alpha,{u}/({1-u})\right)$ and $f$ can be extended 
continuously to the compact set $[0.9,1]\times [0,1]$ with $f(\alpha,0)=1$, $f(\alpha,1)=\alpha^{n-m}$, $f(1,u)=1$. Now, consider $\partial f
/ \partial \alpha$. For $u<1$,
\begin{equation*}
\frac{\partial f}{\partial \alpha} =  \frac{wQ\left(w\right) \left[ P'\left(\alpha w\right) 
Q\left(\alpha w\right) - P\left(\alpha w\right) Q'\left(\alpha w\right) \right]}{P\left(w\right) Q\left(\alpha w\right)^2}.
\end{equation*}
This expression extends continuously to $[0.9,1]\times [0,1]$. As $u\rightarrow 0$, $|{\partial f}/{\partial \alpha}| \leqslant C u$, thus tends to 
$0$. As $u\rightarrow 1^-$, it tends to $(n-m)\alpha^{n-m-1}$, which is bounded for $\alpha \in [0.9,1]$. At $\alpha=1$, it is well-defined. Thus, 
$\partial f / \partial \alpha$ is continuous on the compact set, hence bounded: $\left| \partial f/{\partial \alpha} \right| \leqslant M$ for 
some $M$. By the mean value theorem, for each $u$,
\begin{equation*}
f(\alpha,u) - f(1,u) = (\alpha - 1) {\partial f (\xi_u,u) }/{\partial \alpha},
\end{equation*}
for some $\xi_u \in (\alpha,1)$. Thus, $\left| f(\alpha,u) - 1 \right| \leqslant M (1 - \alpha) = {M}/{J}$. Setting $C=M$ completes the proof for
the first inequality. {For the second inequality, we consider the auxiliary function $\phi(s) = -\log R(s)$, so that $\phi'(s) = -R'(s)/R(s)$. By Assumption~\ref{assum:1} (i) and the fact that $R$ is a rational function, we have $\phi'(0)=1$ and $\phi'(\infty)=0$. The continuity of $\phi'(s)$ implies $|\phi'(s)|\leq K$ for some constant $K>0$. The desired result follows from
\begin{align*}
    |R_2(Jz)/R_1(Jz)-\e^z|=|\e^{\phi(Jz)-\phi((J-1)z)}-\e^z| = \e^z |\e^{(\phi'(\xi_z )-1)z}-1|\leqslant \e^{z_0}(K+1) \e^{(K+1)z_0} z,
\end{align*}
where $\xi_z \in ((J-1)z,Jz)$. The third inequality follows similarly.
}
\end{proof}

\begin{remark}
    The constant $C_R$ in \eqref{eqn:e^z} becomes small when $R(z)$ provides a good approximation to $\e^{-z}$ and $z_0$ is small. 
\end{remark}

\begin{corollary}\label{coro:1}
    There exist constants $c$ and $C$, independent of $J$ and $z$, such that, for $J\geqslant 10$, $c \leqslant {|R_1(Jz)|}/{|R_2(Jz)|} \leqslant C.$
\end{corollary}

\begin{proof}
    By Lemma~\ref{lem:1}, there exists some $J_0\in \mathbb{N}^+$ such that for all $J\geqslant J_0$ and all $ z \in \mathbb{R}^+$, $ 1/2 \leqslant {|R_1(Jz)|}/{|R_2(Jz)|} \leqslant 3/2.$ Combined with the fact that $f(\alpha,w)={R(\alpha w)}/{R(w)}$ is a positive function on the compact set $[0.9,1-{J_0}^{-1}]\times [0,1]$, we obtain the desired result.
\end{proof}

\begin{lemma}\label{lem:2}
There exists a constant $C$, independent of $z$ and $J$, such that for $z \in \mathbb{R}^+$ and $J \geqslant 10$,  
\begin{equation*}
\frac{|F_{1}^{\left( J \right)}\left( z \right)R_2(Jz)R_1^{-1}(Jz) +F_{2}^{\left( J \right)}\left( z \right) -\e^{-Jz}|}{1-\e^{-Jz}} \leqslant 
\frac{C}{J},\quad \frac{|{\eta}_2\left( z,J \right) |}{1-\e^{-Jz}} \leqslant \frac{C}{J},\quad \frac{|{\eta}_1(z,J)|}{1-\e^{-Jz}} \leqslant 
\frac{C}{J}.
\end{equation*}
\end{lemma}

\begin{proof}
We denote the left-hand sides of the first, second, and third inequalities by $H_1(z,J)$, $H_2 (z,J)$ and $H_3(z,J)$, respectively. Recall that ${\eta}_2$ and ${\eta}_1$ are defined in \eqref{eqn:g_0}. Then, we consider two cases: (i) $z \in 
(1/4,\infty)$; (ii) $z \in (0,1/4)$.

\medskip
 \textbf{Step (i):} For $z\in (1/4,\infty)$, all three functions decay exponentially with $J$, thus can be bounded by polynomial decay. 
Combined with ${|R_2(Jz)|}/{|R_1(Jz)|} \leqslant C$ from Corollary~\ref{coro:1}, $H_1$ can be bounded as:
\begin{equation*}
    H_{1}\left( z,J \right) \leqslant 2 \left( C|F_{1}^{\left( J \right)}\left( z \right) |+|F_{2}^{\left( J \right)}\left( z \right) |+\e^{-Jz} \right) \leqslant \left( 0.8 \right)^{J} +C(|F_{1}^{\left( J \right)}\left( z \right) |+|F_{2}^{\left( J \right)}\left( z \right) |).
\end{equation*}
Since $|r_1(z)|,|r_2(z)|< 0.78$ when $z\geqslant 1/4$, we have
\begin{align*}
&\quad|F_{1}^{(J)}(z) + F_{2}^{(J)}(z)| \leqslant \frac{|r_{1}r_{2}|}{|r_{1}-r_{2}|} \cdot |r_{2}^{J}-r_{1}^{J}| + \frac{|r_{1}^{J+1}-r_{2}^{J+1}|}{|r_{1}-r_{2}|} \\
&\leqslant |r_{1}r_{2}|\sum_{i=0}^{J-1} |r_{2}^{J-1-i}r_{1}^{i}| + \sum_{i=0}^{J} |r_{1}^{J-i}r_{2}^{i}| \leqslant J0.78^{J+1} + (J+1)0.78^{J+1}\leqslant  30(0.8)^{J+1}.
\end{align*}
Finally, we obtain $H_1(z,J) \leqslant C\left( 0.8 \right)^{J} \leqslant C/J$, where the constant $C$ is independent of $z$ and $J$. Combined with $c \leqslant {|R_2(Jz)|}/{|R_1 (Jz)|}\leqslant C$ from Corollary~\ref{coro:1}, $H_2(z,J)$ can be bounded as 
\begin{align*}
 H_{2}(z,J)&\leqslant 2\left( |F_{1}^{\left( J-1 \right)}\left( z \right) |+C|F_{2}^{\left( J-1 \right)}\left( z \right) |+C|F_{1}^{\left( J \right)}\left( z \right) |+|F_{2}^{\left( J \right)}\left( z \right) | \right)\\ 
&\leqslant C(\left( 0.8 \right)^{J} +\left( 0.8 \right)^{J+1}) \leqslant {C}/{J}.
\end{align*}
$H_3(z,J) \leqslant C/J$ holds similarly.

 \textbf{Step (ii):}
 We consider the case when $z\in (0,1/4)$. By \eqref{eqn:e^z}, we obtain
 \begin{equation}\label{eqn:H1}
      \begin{aligned}
H_{1}\left( z,J \right) &\leqslant \frac{|F_{1}^{\left( J \right)}\left( z \right) |}{1-\e^{-Jz}} |\frac{R_{2}\left( Jz \right)}{R_{1}\left( Jz \right)} -\e^{z}|+\frac{|F_{1}^{\left( J \right)}\left( z \right) \e^{z}+F_{2}^{\left( J \right)}\left( z \right) -\e^{-Jz}|}{1-\e^{-Jz}} \\ 
&\leqslant \frac{C_R}{J}\cdot \frac{Jz|F_1^{(J)}(z)|}{1-\e^{-Jz}} + \frac{|F_{1}^{\left( J \right)}\left( z \right) \e^{z}+F_{2}^{\left( J \right)}\left( z \right) -\e^{-Jz}|}{1-\e^{-Jz}} =: \frac{C_R}{J}\cdot {\rm I_1}+ {\rm I_2}.
\end{aligned}
 \end{equation}

By \cite[Lemma 10.3]{Thomee:2006}, ${\rm I_1}$ can be bounded by a constant $C$ and simple calculus yields 
\begin{equation}\label{eqn:I_2}
    {\rm I_2} \leqslant {0.55}/{J^2}.
\end{equation}

The desired estimate for $H_1$ follows from the bounds for ${\rm I_1}$ and ${\rm I_2}$. For $H_2(z,J)$, similar to the estimate for $H_1$, we apply \eqref{eqn:e^z} again and obtain
\begin{align*}
 &\quad H_{2}\left( z,J \right) \leqslant \frac{|F_{2}^{\left( J-1 \right)}\left( z \right) |}{1-\e^{-Jz}} |\frac{R_{1}\left( Jz \right)}{R_{2}\left( Jz \right)} -\e^{-z}|+\frac{|F_{1}^{\left( J \right)}\left( z \right) |}{1-\e^{-Jz}} |\frac{R_{2}\left( Jz \right)}{R_{1}\left( Jz \right)} -\e^{z}|\\ 
 & +\frac{|F_{1}^{\left( J-1 \right)}\left( z \right) +F_{2}^{\left( J-1 \right)}(z)\e^{z}-F_{1}^{\left( J \right)}(z )\e^{-z}-F_{2}^{\left( J \right)}\left( z \right) |}{1-\e^{-Jz}} \leqslant \frac{C}{J} +\frac{C}{J} +{\frac{C}{J}} =\frac{C}{J}.
\end{align*}
Next, we estimate $H_3(z,J)$. Similar to the estimate for $H_1$, we obtain
\begin{align*}
H_{3}\left( z,J \right) &\leqslant \frac{|F_{1}^{\left( J \right)}\left( z \right) |}{1-\e^{-Jz}} |\frac{R_{2}\left( Jz \right)}{R_{1}\left( Jz \right)} -\e^{z}|+\frac{|F_{1}^{\left( J-1 \right)}\left( z \right) -F_{1}^{\left( J \right)}\left( z \right) \e^{z}|}{1-\e^{-Jz}}\\ 
&\leqslant \frac{C_{R}}{J}\cdot \frac{Jz|F_1^{(J)}(z)|}{1-\e^{-Jz}} +{\rm II}\leqslant \frac{CC_{R}}{J} + {\rm II}.
\end{align*}
Since $1-\e^{-z} \geqslant z\e^{-z}$  when $z>0$, after some basic calculus, we obtain
\begin{align*}
 {\rm{II}} &\leqslant \frac{|\left( 1-r_{2}\e^{z} \right) r_{2}^{J-1}-r_{1}^{J-1}\left( 1-r_{1}\e^{z} \right) |}{2\sqrt{1-2z}Jz\e^{-Jz}} \leqslant \frac{C}{Jz} \e^{Jz}\left( r_{1}^{J-1}|1-r_{1}\e^{z}|+r_{2}^{J-1}|1-r_{2}\e^{z}| \right)\\ &=: {C}{(Jz)}^{-1} \e^{Jz}\left( r_{1}^{J-1}{\rm II_{1}}+r_{2}^{J-1}{\rm II_{2}} \right).
\end{align*}
Since BDF2 is a second order scheme, we have ${\rm II_{1}}\leqslant \e^{z}|\e^{-z}-r_{1}|\leqslant Cz^{3}$ for some $C>0$ and $z \in (0,1/4)$. Also, ${\rm II_2}$ is bounded. Then, we obtain ${\rm II}\leqslant {Cz^{2}}{(Jr_{1})}^{-1} \left( r_{1} \e^{z} \right)^{J} +{C}{(Jz r_{2})}^{-1}\left( r_{2} \e^{z} \right)^{J}$. Since $0.76 \leqslant r_1\leqslant \e^{-z}$ and $ 0.3 \leqslant r_2 \leqslant 0.5\e^{-z}$ when $z\in(0, 1/4)$, ${\rm II}$ can be bounded by $C/J$.

\end{proof}

\begin{remark}
For the CPs (I)~\eqref{eqn:cp}, we have $R_2=R_1$ and the error term $H_1$ takes the form $H_1(z,J)={|F_{1}^{\left( J \right)}\left( z 
\right) +F_{2}^{\left( J \right)}\left( z \right) -\e^{-Jz}|}/({1-\e^{-Jz}})$, which can be bounded by $C/J$. In contrast, the error 
$H_1$ in~\eqref{eqn:H1} for CPs (II) is typically much smaller. Its first term is controlled by $C_R/J$ and can be made arbitrarily small by 
choosing $R(z)$ to approximate $\e^{-z}$ closely. The bound in \eqref{eqn:I_2} decays as $O(1/J^2)$ and $I_2$ is therefore negligible for large $J$. 
\end{remark}

\begin{corollary}\label{coro:2}
There exists a constant $C$, independent of $z$ and $J$, such that, for $z \in (0,\infty)$ and $J \geqslant 10$,  
\begin{equation*}
\frac{|{\eta}_4(z,J)|}{1-\e^{-Jz}} \leqslant C\quad \text{and}\quad \frac{|F_1^{(J)}(z)R_2(Jz)R_1^{-1}(Jz)+F_2^{(J)}(z)+\e^{-Jz}-2R(Jz)|}{1-\e^{-Jz}} \leqslant C.
    \end{equation*}
\end{corollary}

\begin{proof}
We define $H_4(z,J) = {|{\eta}_4(z,J)|}/({1-\e^{-Jz}}).$ Then, we obtain $H_{4}\left( z,J \right) \leqslant H_{1}\left( z,J \right) +|\gamma \left( Jz \right) |\leqslant C,$ with $H_1$ defined in the proof for Lemma~\ref{lem:2} and $\gamma(s)$ defined in~\eqref{eqn:gamma}. The proof for the second inequality follows similarly. 
\end{proof}

\subsection{Error estimate}\label{sec:error estimate}
In the following theorems, we clarify the connection between the F-multistep parareal \ref{alg:m_para} with the plain parareal. The quantity $\gamma_{a}$ is the main component among all coefficients of $\xi_\ell$ in \eqref{eqn:cubic_ode}. The other coefficient functions will decay to zero as $J\to \infty$.
\begin{theorem}\label{thm:lim}
 There exists a constant $C$, independent of $z$ and $J$, such that, for $z \in \mathbb{R}^+$ and $J \geqslant 10$, 
\begin{align*}
&|{\gamma}_{a} \left( z,J \right) -\gamma(Jz)^{2} |\leqslant {C}/{J} \quad {\gamma}_{e} \left( z,J \right) , {\gamma}_{b} \left( z,J \right) \leqslant {C}/{J},\quad {\gamma}_{c} \left( z,J \right) {\gamma}_{d} \left( z,J \right) \leqslant {C}/{J^{2}}.
\end{align*}
\end{theorem}

\begin{proof}
We prove these claims one by one.

\medskip   
 \textbf{First inequality:} Recall that ${\gamma}_a$ and $\gamma$ are defined in \eqref{eqn:g_a} and \eqref{eqn:gamma}, respectively. We observe that $\gamma \left( Jz \right) ^{2}=\sum_{m=0}^{\infty} |R\left( Jz \right) |^{m}\left( m+1 \right) |R\left( Jz \right) -\e^{-Jz}|^2.$ Then, 
\begin{align*}
&|{\gamma}_a(z,J)-\gamma(Jz)^2|\leqslant \sum_{m=0}^{\infty} |R\left( Jz \right) |^{m}\left( m+1 \right) |{\eta}_{4} \left( z,J\right) -\left( R\left( Jz \right) -\e^{-Jz} \right) |\\
&\quad \cdot |{\eta}_{4} \left( z,J \right) +\left( R\left( Jz \right) -\e^{-Jz} \right) | +\sum_{m=0}^{\infty} |R\left( Jz \right) |^{m}|{\eta}_{3} \left( z,J \right) {\eta}_{2} (z,J)| \\
&\leqslant \frac{|F_1^{(J)}(z)R_2(Jz)R_1^{-1}(Jz)+F_2^{(J)}(z)+\e^{-Jz}-2R(Jz)|}{1-\e^{-Jz}}\\
&\quad \cdot \frac{|F_1^{(J)}(z)R_2(Jz)R_1^{-1}(Jz)+F_2^{(J)}(z)-\e^{-Jz}|}{1-\e^{-Jz}} \cdot \left(\frac{1-\e^{-Jz}}{1-|R(Jz)|}\right)^2\\
&\quad + |{\eta}_3(z)|\cdot \frac{|{\eta}_2 (z,J)|}{1-\e^{-Jz}}\cdot \frac{1-\e^{-Jz}}{1-|R(Jz)|}.
\end{align*}
By Corollary~\ref{coro:2}, Lemma~\ref{lem:2}, and the boundedness of $|{\eta}_3|$ from Corollary~\ref{coro:1}, we obtain $|{\gamma}_a(z,J)-\gamma(Jz)^2| \leqslant C/ J$. 

\medskip
 \textbf{Second inequality:} Recall that ${\gamma}_e$ is defined in \eqref{eqn:g_e}, and $|{\eta}_3|$ is uniformly bounded for $z \in \mathbb{R}^+$ and $J\geqslant 10$; then, we have
\begin{equation*}
    {\gamma}_{e}(z,J)\leqslant \frac{C|{\eta}_2(z,J)|}{1-e^{-Jz}}\cdot \frac{1-e^{-Jz}}{1-|R(Jz)|} + |C{\eta}_{2}(z,J)|+|{\eta}_1^2(z,J)|.
\end{equation*}
By Lemma~\ref{lem:2}, we obtain ${\gamma}_e(z,J)\leqslant {C}/{J}+{C}/{J}+ {C}/{J^2} \leqslant {C}/{J}$. 

\medskip
 \textbf{Third inequality:} Recall that ${\gamma}_b $ is defined in \eqref{eqn:g_b}. The estimate in the first inequality implies that ${{\eta}_4 (z,J)}/{(1-|R(Jz)|)}$ is bounded. By the boundedness of ${\eta}_3$ and Lemma~\ref{lem:2}, we obtain
\begin{equation}\label{eqn:est_ga_b}
{\gamma}_b (z,J) \leqslant {C}/{J}({C}/{J}+C)^2 \leqslant {C}/{J}.
\end{equation}

\medskip
 \textbf{Fourth inequality:} Recall that ${\gamma}_c $ and ${\gamma}_d$ are defined in \eqref{eqn:g_c} and \eqref{eqn:g_d}. We first bound ${\gamma}_d(z,J)$ as in the previous steps:
\begin{equation*}
|{\gamma}_d(z,J)|\leqslant \frac{C|{\eta}_{2}(z,J)|}{1-\e^{-Jz}}\cdot\frac{1-\e^{-Jz}}{1-|R(Jz)|}+\frac{C}{J}\leqslant \frac{C}{J}. 
\end{equation*}
Then, we derive the bound for the product term ${\gamma}_d {\gamma}_c$. By the boundedness of ${\eta}_3$, we obtain
\begin{align*}
&{\gamma}_{d}(z,J) {\gamma}_{c}(z,J) \leqslant {\gamma}_{d}(z,J) \frac{C|{\eta}_2(z,J)||{\eta}_4(z,J)|}{(1-|R(Jz)|)^{3}} 
+ {\gamma}_{d}(z,J) \frac{C|{\eta}_2(z,J)||{\eta}_1(z,J)|}{(1-|R(Jz)|)^{2}} \\
&\quad + {\gamma}_{d}(z,J) \frac{C|{\eta}_4(z,J)||{\eta}_{1}(z,J)|^{2}}{(1-|R(Jz)|)^{2}} 
+ {\gamma}_{d}(z,J) \frac{C|{\eta}_1(z,J)|^{3}}{1-|R(Jz)|} =: \rm{I_1}+\rm{I_2} + \rm{I_3} +\rm{I_4}.
\end{align*}
By the boundedness of $|{\eta}_1|$ and $|{\eta}_4|$ from Lemma~\ref{lem:2} and Corollary~\ref{coro:2}, $\rm{I}_1$ can be bounded as 
\begin{equation*}
    {\rm{I_1}} \leqslant C\left(\frac{|{\eta}_2(z,J)|}{1-|R(Jz)|}\right)^{2} \cdot \left(\frac{|{\eta}_4(z,J)|}{1-|R(Jz)|}\right)^{2} + C\left(\frac{|{\eta}_2(z,J)|}{1-|R(Jz)|}\right)^{2} \cdot \left(\frac{|{\eta}_4(z,J)|}{1-|R(Jz)|}\right).
\end{equation*}
Furthermore, by Lemma~\ref{lem:2} and Corollary~\ref{coro:2}, we obtain ${\rm{I_1}} \leqslant {C}/{J^2}+{C}/{J^2} \leqslant {C}/{J^2}.$ By the estimate \eqref{eqn:est_ga_b} and Lemma~\ref{lem:2}, the term $\rm{I_2}$ can be bounded as 
\begin{equation*}
    {\rm{I_2}} \leqslant  \frac{C}{J} \frac{|{\eta}_2(z,J)|}{1-|R(Jz)|} \cdot \frac{|{\eta}_{1}(z,J)|}{1-|R(Jz)|} \leqslant \frac{C}{J^3}.
\end{equation*}
In analogy to $\rm{I}_1$ and $\rm{I}_2$, since $(1-|R(Jz)|)\leqslant 1$, the term ${\rm{I}_3}+\rm{I}_4$ can be bounded as 
\begin{align*}
    {\rm{I}_3+\rm{I}_4}&\leqslant {\gamma}_{d}(z,J)\frac{|{\eta}_4(z,J)|}{1-|R(Jz)|}\cdot \frac{|{\eta}_{1}(z,J)|^{2}}{(1-|R(Jz)|)^{2}} 
+ {\gamma}_{d}(z,J) \frac{C|{\eta}_1(z,J)|^{3}}{(1-|R(Jz)|)^3} \leqslant {C}/{J^3} + {C}/{J^4} \leqslant {C}/{J^3}.
\end{align*}
Finally, the product term ${\gamma}_d {\gamma}_c$ can be bounded by $C/J^2$. 
\end{proof}

Theorem~\ref{thm:lim} shows that $|{\gamma}_a + {\gamma}_e - \gamma(Jz)^2|\sim O(1/J)$, $|{\gamma}_{a} 
{\gamma}_{e} -{\gamma}_{b} |\sim O(1/J)$ and $|{\gamma}_{b} {\gamma}_{e} -
{\gamma}_{c} {\gamma}_{d} |\sim O(1/J^2)$ in \eqref{eqn:cubic_ode}. Then, we have the following connection between 
${\gamma}(z,J)$ and $\gamma^2(Jz)$, when $J$ is large.

\begin{theorem}\label{thm:conv1}
   Let $R$ and $\gamma$ of the plain parareal satisfy Assumption~\ref{assum:1}. For the  convergence factor 
   ${\gamma}^\dag (J)=\sup_{z\in \mathbb{R}^+}{\gamma}(z,J)$ with ${\gamma}$ defined in~\eqref{eqn:cubic_ode}, 
   there exists a constant $C$, independent of $J$ and $z$, such that
   \begin{equation*}
|{\gamma}^\dag (J) -\gl^2| \leqslant {C}/{J}.
   \end{equation*}
\end{theorem}

\begin{proof}
Denote $c_0(z,J)={\gamma}_{b} {\gamma}_{e} -{\gamma}_{c} {\gamma}_{d}$, $ c_1(z,J)={\gamma}_{a} 
{\gamma}_{e} -{\gamma}_{b}$, and $c_2 (z,J)=(\gamma(Jz))^2-\left( {\gamma}_{a} +{\gamma}_{e} \right)$. 
Then, \eqref{eqn:cubic_ode} becomes
\begin{equation}\label{eqn:cubic_3}
\xi_{\ell}^{3} +(c_2(z,J)-(\gamma (Jz))^2) \xi_{\ell}^{2} +c_1(z,J) \xi_{\ell} +c_0(z,J) =0. 
\end{equation}
Theorem~\ref{thm:lim} indicates that $|c_2(z,J)|\leqslant C/J$, $|c_1(z,J)|\leqslant C/J$, and $|c_0 (z,J)|\leqslant C/J^2$, where the constant $C$ is 
independent of $J$ and $z$. Recall that $\gl$ is the supremum defined in \eqref{eqn:gamma} and $|\gamma (s_0)| =\gl$. Consider $Jz_j=s_0$, 
i.e., $z_j=s_0/J$, and \eqref{eqn:cubic_3} becomes $\xi_{\ell}^{3} +(c_2(z_j,J)-(\gamma (s_0))^2) \xi_{\ell}^{2} + c_1(z_j,J) \xi_{\ell} +c_0(z_j,J) =0.$ Consider the disk $|\xi - \gl^2| \leqslant D/J$ and apply Rouch\'e's theorem on the circle $|\xi -\gl^2|=D/J$. Let 
$q(\xi)=\xi^3 - \gl^2 \xi^2$ and $h(\xi)=c_2 \xi^2 + c_1 \xi + c_0$. We can choose $D$ and $J_0$ sufficiently large such that $|h(\xi)| \leqslant |q(\xi)|$ and $D/J \leqslant \gl^2/2$ for $J\geqslant J_0$, since 
\begin{equation*}
|h\left( \xi \right) |\leqslant \frac{C}{J} \left( \gl^2+\frac{D}{J} \right)^{2} +\frac{C}{J} \left( \gl^2+\frac{D}{J} \right) +\frac{C}{J^{2}} 
\leqslant \frac{C}{J}  \leqslant \frac{D}{J} \left( \gl^2-\frac{D}{J} \right)^{2} \leqslant |q\left( \xi \right) |.
\end{equation*}
By Rouch\'e's theorem, there is exactly one root in $|\xi - \gl^2| \leqslant D / J$ when $J\geqslant J_0$. Through a similar argument, we deduce that the other two roots approach $0$ as $J \to \infty$. Thus, we obtain 
\begin{equation}\label{eqn:ga_zj}
|{\gamma} (z_j,J) -(\gamma(s_0))^2| \leqslant {C}/{J}.
\end{equation}
Now consider the general case. By applying Rouch\'e's theorem to $q$ and $h$ again on $|\xi|=D_1$ for some large $R$, we deduce that all roots are bounded in $|\xi| \leqslant D_1$. Then, \eqref{eqn:cubic_3} implies $|(\xi_{\ell} -\left( \gamma \left( Jz \right) \right)^{2}) \xi_{\ell}^{2} |\leqslant |h\left( \xi_{\ell} \right) |\leqslant {C}/{J} .$ Taking the supremum over $z \in \mathbb{R}^+$, the inequality holds as well: $\sup_{z\in \mathbb{R}^+}|(\xi_{\ell} -\left( \gamma \left( Jz \right) \right)^{2}) \xi_{\ell}^{2} |\leqslant {C}/{J}.$ We are interested in $\xi_{\ell}$ away from $0$, i.e., $\xi_{\ell}\geqslant c>0$ for a positive constant $c=\gl^2/2$. Such $\xi_{\ell}$ exists for sufficiently large $J$ since $|\gamma(s_0)|=\gl>0$. Then, we have 
\begin{equation*}
\max_{{\ell \in \{1,2,3\},|\xi_\ell|\geqslant c}}\sup_{z\in \mathbb{R}^{+}} |\xi_{\ell} |\leqslant {C}/{c^2J} +\sup_{z\in \mathbb{R}^{+}} |\gamma \left( Jz \right) |^{2} \leqslant {C}/{c^2J} + \gl^2. 
\end{equation*}
Combined with the lower bound provided by \eqref{eqn:ga_zj}, we obtain the desired result. 
\end{proof}

\begin{remark}
The results presented in Sections~\ref{sec:pre est} and~\ref{sec:error estimate} are for CPs (II) and can be readily extended to CPs (I) without any modifications to the statements.
\end{remark}

Fig.~\ref{fig:inv_fig} illustrates the result of Theorem~\ref{thm:conv1}. Recall that ${\gamma}^\dag (J) = \sup_{z\in
\mathbb{R}^+}|{\gamma}(z,J)|$. It can be verified that the stability functions of the BE, LIIIC2, and LIIIC2(2) schemes satisfy 
Assumption~\ref{assum:1}. Fig.~\ref{fig:inv_fig} demonstrates that, for both methods, the difference $|\gamma^\dag (J)-\gl^2|$ decays as 
${O}({J}^{-1})$. Moreover, the discrepancy for CPs (II) is smaller than that for CPs (I), indicating that Type (II) performs notably 
better than Type (I) when $\gl$ is small. Note that the difference between the two types for LIIIC2 is particularly pronounced around $J \in 
(10,80)$. This arises because, for Type (I), $\gl$ is small, and the leading error comes from the other coefficients in \eqref{eqn:cubic_3}. {We also observe sharp decreases in both LIIIC2 and LIIIC2(2) of Type (I), occurring around $J=80$ and $J=550$, respectively. These decreases arise because the dominant error source shifts from other terms in Theorem~\ref{thm:lim} to $|\gamma_a - \gamma(Jz)|^2$.
}

\begin{figure}[htbp!]
    \centering
    \includegraphics[width=0.98\textwidth,trim={0.7cm 1cm 1.5cm 1cm},clip]{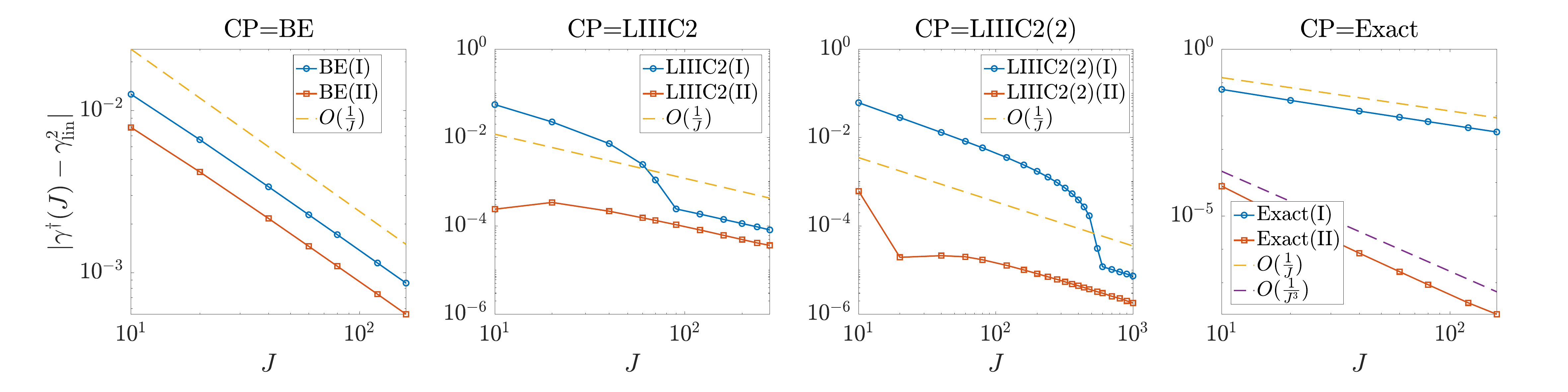}
    \caption{The function $|\gamma^\dag (J)-\gl^2|$ versus the coarsening factor $J$ for four CPs of two types. }
    \label{fig:inv_fig}
\end{figure}

\section{Further discussion and numerical experiments}\label{sec:numerical}
In this work, we propose a new update scheme \eqref{eqn:CP-II} for the F-multistep parareal and compare it with the original scheme 
\eqref{eqn:CP-I} in Section \ref{sec:conv}. Other possible structures exist that incorporate the multistep FPs and 
satisfy the consistency condition of the block iteration \cite{gander2023unified}. We also examined the following structure: 
\begin{align*}
\left\{
\begin{aligned}U_{n+1}^{k+1}&=G\left( U_{n}^{k+1} \right) +F_{1}^{\left( J \right)}\left( U_{n,-1}^{k} \right) +F_{2}^{\left( J \right)}\left(
U_{n}^{k} \right) -G\left( U_{n}^{k} \right)\\ U_{n+1,-1}^{k+1}&=G\left( U_{n,-1}^{k+1} \right) +F_{1}^{\left( J-1 \right)}\left( U_{n,-1}^{k}
\right) +F_{2}^{\left( J-1 \right)}\left( U_{n}^{k} \right) -G\left( U_{n,-1}^{k} \right).\end{aligned}
\right.
\end{align*}
In this update scheme, we use the term $G(U_{n,-i}^{k+1})-G(U_{n,-i}^k)$ to correct $U_{n+1,-i}^{k+1}$ for $i=0,1$. The finite 
convergence property is straightforward to verify. However, numerical experiments indicate that the parareal error grows with the iteration
$k$. The original Type (I) update scheme is convergent but introduces a consistency error, thus impeding a fast convergence rate. In contrast,
our newly proposed Type (II) scheme allows the computations of $U_{n+1,-i}^{k+1}$ for $i=1,\dotsc,q$ to remain parallel and achieves fast 
convergence when CPs are of high resolution. In the following numerical results, we illustrate the comparison between Type (I) and Type (II)
schemes on a linear model and a semilinear model.  

There is an alternative approach to employing the BDF$q$ method as the FP in the parareal algorithm. Within each coarse interval $[\Tau_n,\Tau_{n+1}]$, a $q$-th order L-stable Runge-Kutta scheme is utilized to initialize $q-1$ intermediate values, after which BDF$q$ is applied to compute the solution at $\Tau_{n+1}$. In this formulation, the FP serves as a single-step mixed scheme that propagates the solution from $\Tau_n$ to $\Tau_{n+1}$. This approach substantially simplifies the error analysis, as the mixed FP can be effectively approximated by the exact solver. However, the mixed FP is computationally much more expensive than the pure BDF$q$ scheme, particularly when the coarsening factor $J$ is small or the order $q$ is large. The convergence of this scheme is numerically assessed on the linear model illustrated in Fig.~\ref{fig:Ex1_nonsmooth}.

\subsection{Test on linear model}
We now numerically illustrate the two types of updates in the F-multistep parareal method, alongside the theoretical estimate established in 
Theorem \ref{thm:main}.
Consider the following initial and boundary value problem:\begin{equation}\label{eqn:diffusion}
\left\{	\begin{alignedat}{3}
&\partial_t u(x,t) - \partial_{xx} u(x,t) = f(x,t),\quad && x \in \Omega,\quad  &&0 < t < T, \\
&u(x,t) = 0, && x \in \partial\Omega,\quad && 0 < t < T, \\
&u(x,0) = u_0(x), && x \in \Omega, &&
\end{alignedat}
\right.
\end{equation}
with $\Omega=(0,1)$, and the following two sets of problem data: {\rm(a)}  $T=2$, $u_0 = \chi_{(0,1/2)},$ and $f\equiv 0$, where 
$\chi_{(0,1/2)}$ denotes the characteristic function of $(0,1/2)$; {\rm(b)} $T=2$, $u_0(x) = 2\sin(\pi x)$
and $f(x,t) =\sin ( \pi x ) \left( \pi^{2} \cos  t  -\sin  t  +\pi^{2} \cos ( 4t ) -4\sin ( 4t \right))$.

In the experiment, we divide the domain $\Omega$ into $1000$ equal subintervals, each of length $h=1/1000$,
apply the Galerkin FEM with piecewise linear FEM, and initialize $U_{n}^{0}$ with the CP. We fix the fine time step size $\tau  = 1/1000$ and 
consider different coarsening factors $J$. 
Throughout, we study the error between the iterative solution $U^k_n$ by the F-multistep parareal and the fine time-stepping solution $u_{nJ}$,
i.e.,
$\text{error}=\max_{1\leqslant n\leqslant N_{c}}\|U_{n}^{k}-u_{nJ}\|_{L^{2}\II} $. 

In Case (a), the problem data is nonsmooth and compatible with the zero Dirichlet boundary condition. Fig.~\ref{fig:Ex1_nonsmooth} illustrates the convergence rates of the F-multistep parareal employing three CPs, namely, BE, LIIIC2, and LIIIC2(2). The exact solver is omitted due to its computational expense in practical applications. Two multistep FPs are considered: the BDF2 and BDF4 schemes, with $J=10$ and $40$. The slopes of 
the reference lines for BE, LIIIC2 and LIIIC2(2) are derived by taking the square root of ${\gamma}^\dag$ from Table~\ref{tab:comparison}. These numerical results indicate that the error estimate in Theorem~\ref{thm:err_recursion} is sharp. When
BE serves as the CP, the error behaviors for both Type (I) and Type (II) updates are nearly identical, consistent with the similar behavior of
the convergence function ${\gamma}$ in Fig.~\ref{fig:comparison_g}. In contrast, when LIIIC2 or LIIIC2(2) is chosen as the CP, the F-multistep parareal with Type (II) update converges significantly faster than with Type (I), indicating that the Type (II) update is more effective. Furthermore, we evaluate the mixed FP discussed in the beginning of this Section. Specifically, the two-stage Lobatto IIIC is employed to initialize the BDF2 scheme, while the three-stage Lobatto IIIC is used to initialize the BDF4 scheme. The convergence behavior of the parareal with Type (II) update closely resembles that of the parareal incorporating the mixed FP. However, the computational cost of high-order L-stable Runge-Kutta methods substantially exceeds that of the BDF$q$ scheme.

In Case (b), the problem data is smooth, with the exact solution  $u(x,t)= \sin(\pi x) (\cos t + \cos(4t))$. Fig.~\ref{fig:Ex1_smooth} illustrates the convergence rate of the F-multistep parareal with four CPs. Here, the optimized coarse propagator (OCP), introduced in \cite{jin2025optimizing}, has the stability function
\begin{equation*}
    R(s)=\frac{1-0.17922s}{1+0.82078s+0.42444s^2}.
\end{equation*}
Note that the OCP is not always positive over $\mathbb{R}^+$; accordingly, two reference lines with empirical slopes of 0.22 and 0.10 are included for comparison. The OCP is specifically designed to minimize the convergence function within the standard parareal framework, providing a uniform approximation to the exact solver. Consequently, Type (II) updates with OCP outperform those with Type (I). Notably, the convergence is faster than in the nonsmooth scenario, owing to the 
higher regularity of the data \cite{dai2013stable}.  It is obvious that the convergence 
rate is stable if we consider the alternate iterations. This observation naturally motivates us to consider the relationship between $(k+1)$-th 
and $(k-1)$-th iterations in \eqref{eqn:k-1}, rather than focusing solely on the adjacent iterations.

\begin{figure*}[htbp!]
  \centering
  \includegraphics[width=0.9\textwidth,trim={0.1cm 0.5cm 1.3cm 1cm},clip]{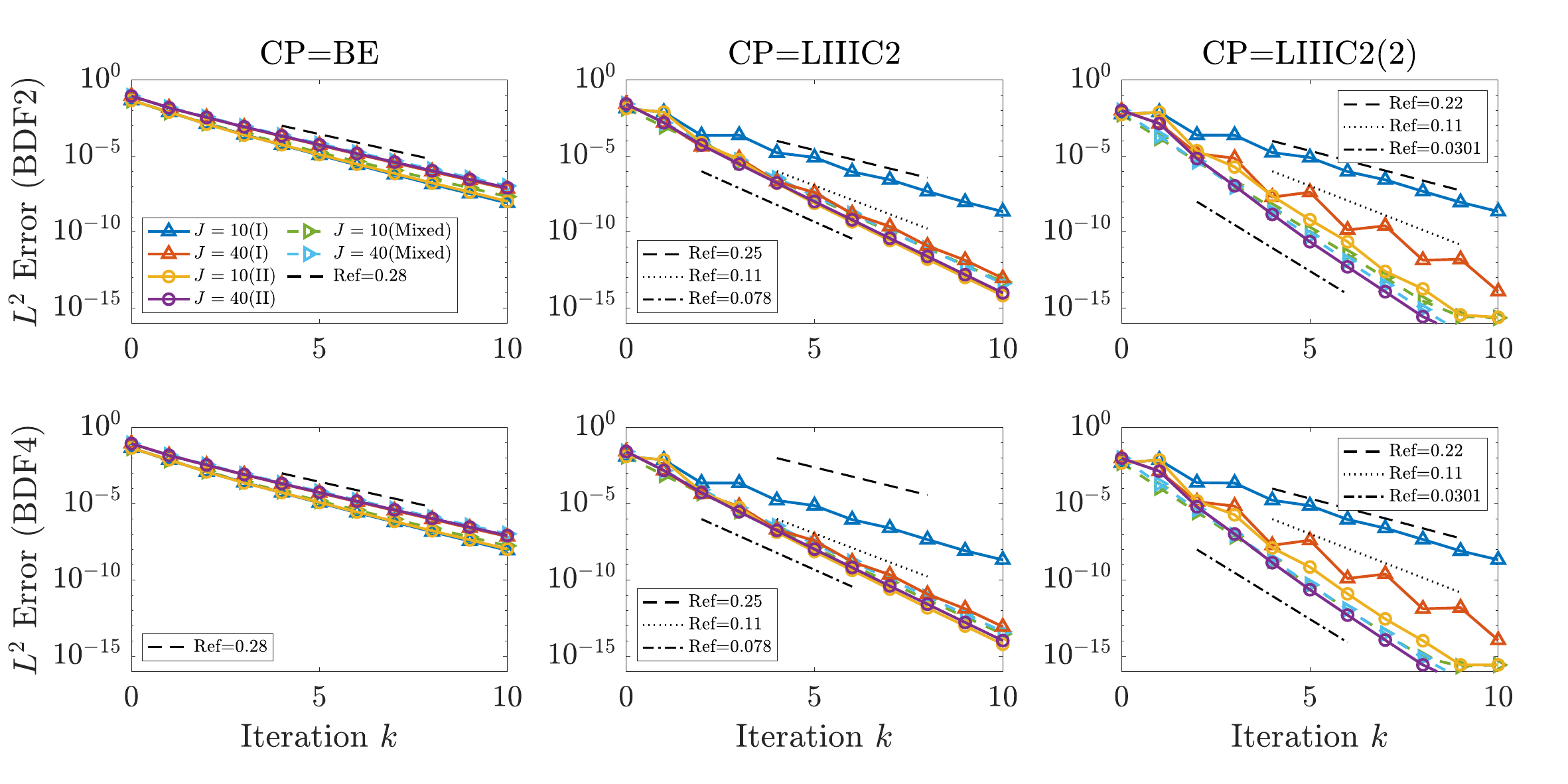}
  \caption{Case (a): the $L^2$ error of the F-multistep parareal using three CPs, BE, LIIIC2 and LIIIC2(2), for two FPs, the BDF2 and BDF4 schemes, with $J=10$ and $40$. The parareal with mixed FPs is also tested under the same configurations. }
  \label{fig:Ex1_nonsmooth}
\end{figure*}

\begin{figure*}[htbp!]
  \centering
  \includegraphics[width=0.98\textwidth,trim={0.1cm 0.5cm 1.3cm 1cm},clip]{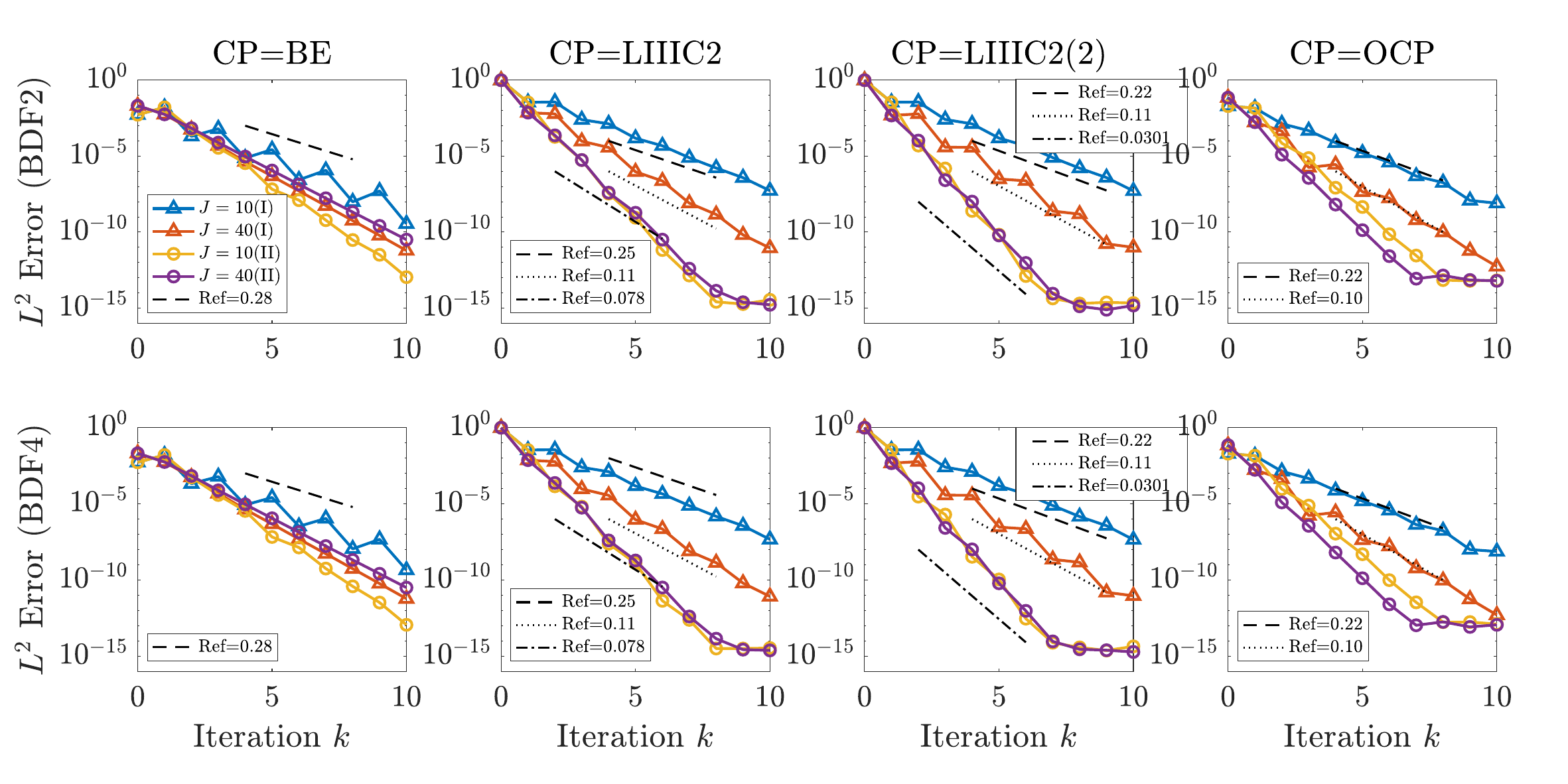}
  \caption{Case (b): the $L^2$ error of the F-multistep parareal using four CPs, BE, LIIIC2, LIIIC2(2) and OCP, for two FPs, the BDF2 and BDF4 schemes, with $J=10$ and $40$.}
  \label{fig:Ex1_smooth}
\end{figure*}

\subsection{Test on semilinear model}
Next, we test the performance of the proposed algorithm for the following semilinear parabolic equation:
\begin{equation}\label{eqn: fisher-kpp}
       \partial_{t} u(x,t)=\partial_{xx} u(x,t)+c_L u(x,t)\left( 1-u(x,t) \right) + g(x,t),\  \left( x,t \right) \in \left[ -1 ,1 \right] \times \left[ 0,T \right] ,
    \end{equation}
subject to homogeneous Dirichlet boundary conditions,  
with the solution $u(x, t) = \cos ({\pi x}/{2} ) \cos t$ and the force function $g(x,t)=\partial_t u - \partial_{xx}u-c_L u(1-u)$. We consider the final time $T=2$.

In this problem, we still apply the Galerkin FEM with piecewise linear functions with a mesh size $h = 1/1000$, and initialize $U_0^n$ with 
the initial value $U_0^0$. We still fix the fine time step size $\tau  = 1/1000$ and study the behavior of F-multistep parareal with $J \in \{10,40\}$ and  $c_L 
\in \{1,4 \}$. As demonstrated in \cite{linear_conv}, the convergence factor of the parareal algorithm, when applied to semilinear parabolic equations, approaches that of the linear case as the coarse time step size $\Tau$ becomes sufficiently small. Consequently, based on the test on the linear model, we can anticipate satisfactory performance from the F-multistep parareal with CPs (II).

In Fig.~\ref{fig:Ex2_nonlinear}, we set $c_L = 1$ in~\eqref{eqn: fisher-kpp} and consider four CPs: BE, LIIIC2, LIIIC2(2), and OCP. 
In all cases, independent of $J$ and of the chosen CP, Type~(II) F-multistep parareal consistently outperforms Type~(I). 
Moreover, the behavior of the two types is qualitatively similar when using BDF2 or BDF4 as FPs. 
When BE is used as the CP, the performance gap between Type~(I) and Type~(II) is small, in agreement with the linear experiments. 
As indicated by the error estimate in Theorem~\ref{thm:main}, there are two main contributions to the error: (i) the mismatch between the FP and CP, and 
(ii) the structural error intrinsic to the F-multistep parareal iteration.  
With BE as the CP, the FP--CP mismatch dominates the structural error, so the difference between the two types remains negligible. 
In contrast, when a more accurate CP such as LIIIC2(2) or OCP is employed, the structural error becomes the main obstacle to fast convergence for Type~(I), whereas the update in Type~(II) mitigates this effect and yields a clear improvement.

In Fig.~\ref{fig:Ex2_nonlinear_2}, we set $c_L = 4$ in \eqref{eqn: fisher-kpp} . As the nonlinearity $c_L$ increases, the discrepancy between 
FPs and CPs grows, leading to reduced convergence rates across all cases. Therefore, the difference between Type (I) and Type (II) for BE and OCP 
is small. However, the performance of the F-multistep parareal can be enhanced when OCP is the CP by incorporating additional terms to better 
approximate the nonlinear term. In the case when LIIIC2 or LIIIC2(2) serves as the CP, the CP still resolves the solution profile well, so 
that it can reflect the benefit of Type (II) update. The same as in the smooth linear case, the F-multistep parareal with Type (II) update converges
faster than the estimated rate owing to the high regularity. 
	
\begin{figure*}[htbp!]
  \centering
  \includegraphics[width=0.98\textwidth,trim={0.1cm 0.5cm 1.3cm 1cm},clip]{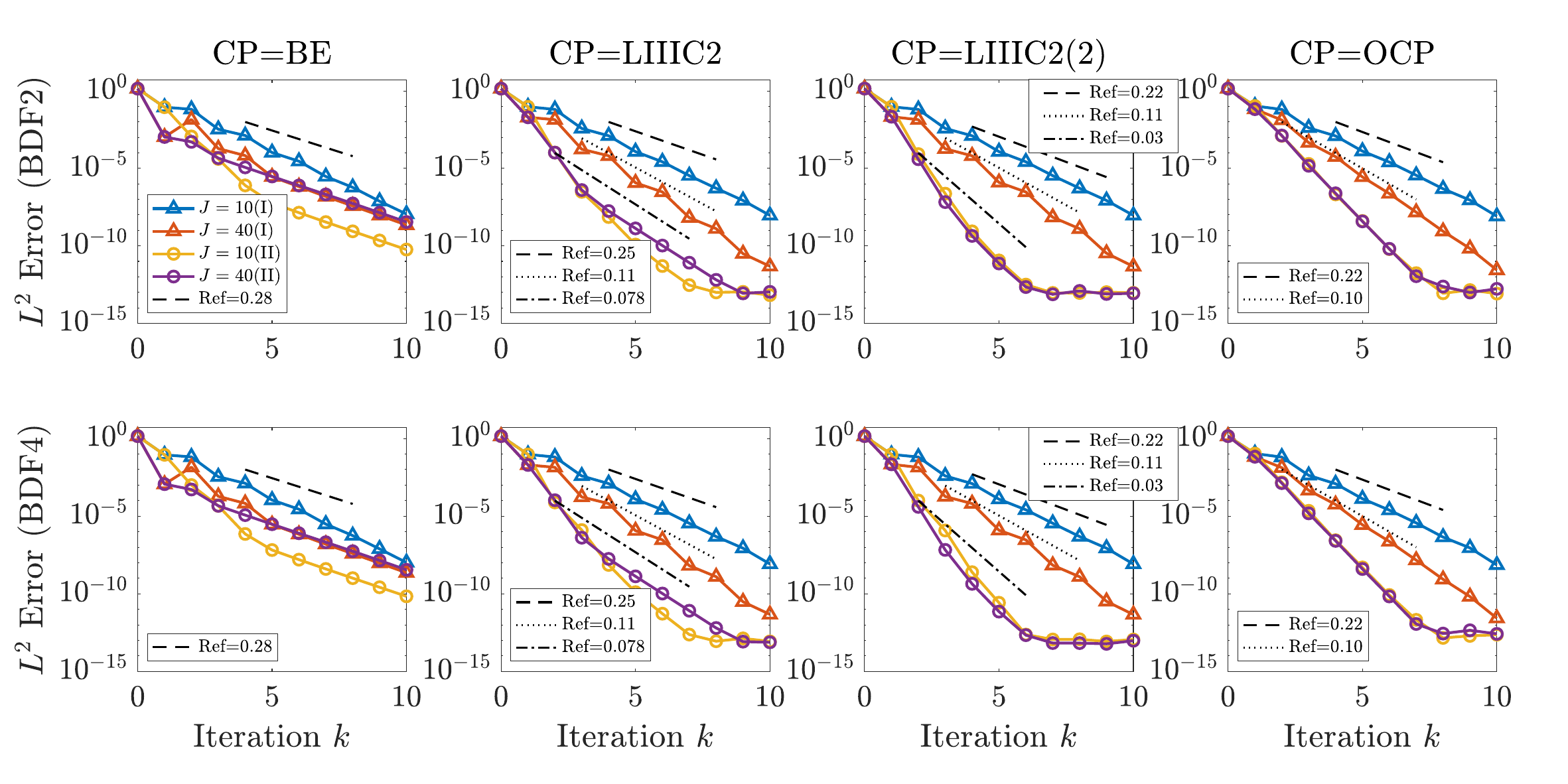}
  \caption{($c_L = 1$): the $L^2$ error of the F-multistep parareal using four CPs, BE, LIIIC2, LIIIC2(2) and OCP, for two FPs, the BDF2 and BDF4 schemes, with $J=10$ and $40$.}
  \label{fig:Ex2_nonlinear}
\end{figure*}

\begin{figure*}[htbp!]
  \centering
  \includegraphics[width=0.98\textwidth,trim={0.1cm 0.5cm 1.3cm 1cm},clip]{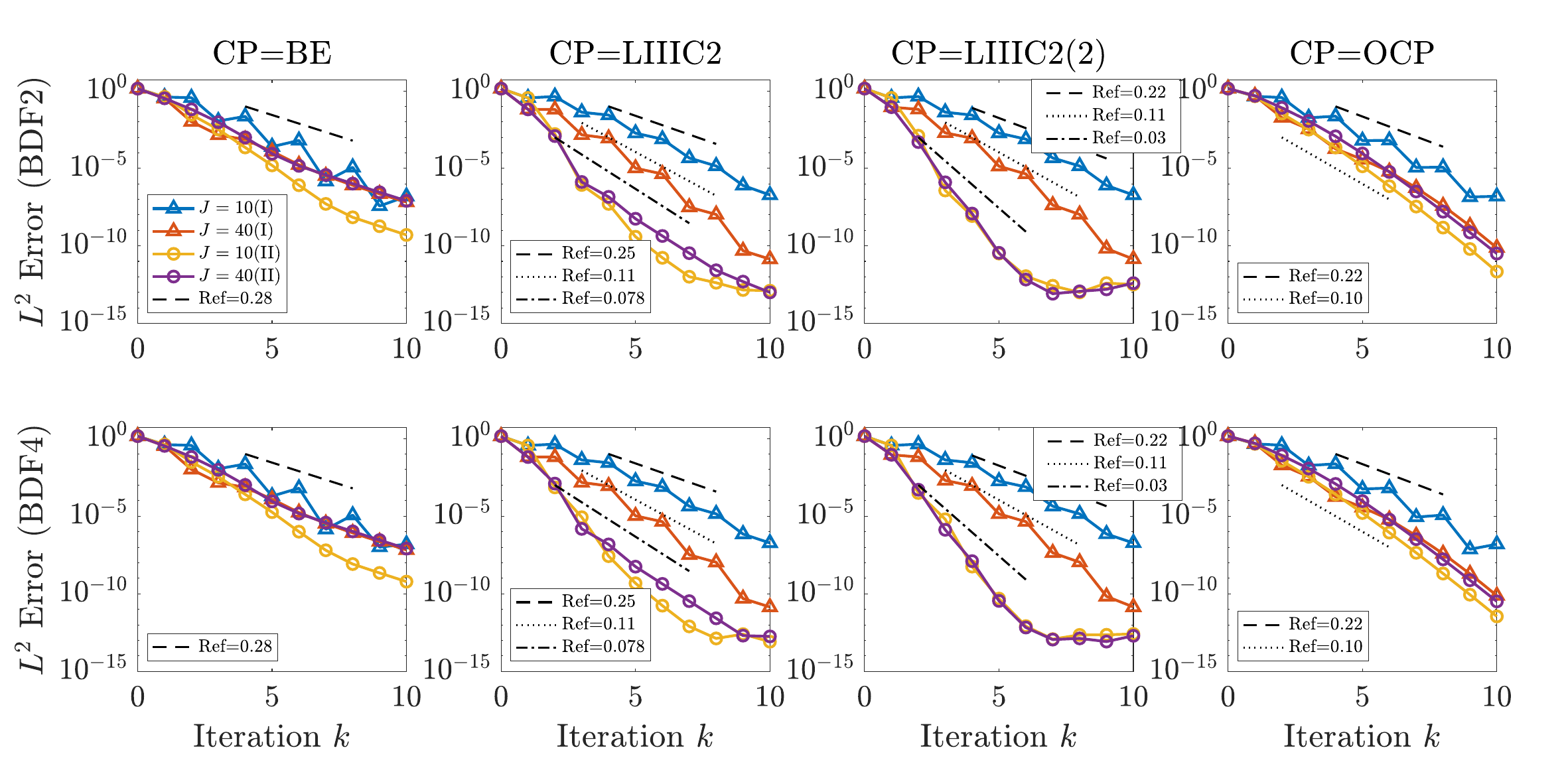}
  \caption{($c_L = 4$): the $L^2$ error of the F-multistep parareal using four CPs, BE, LIIIC2, LIIIC2(2) and OCP, for two FPs, the BDF2 and BDF4 schemes, with $J=10$ and $40$.}
  \label{fig:Ex2_nonlinear_2}
\end{figure*}

\section{Conclusion}
In this work, we propose a novel update scheme for the new variables in the multistep framework. The F-multistep parareal incorporating this scheme achieves faster convergence than the existing one introduced in \cite{ait2024multi}, while maintaining similar computational cost. We provide a detailed linear convergence analysis for the case when BDF2 serves as the FP for linear parabolic equations, and demonstrate the advantages of our update scheme both theoretically and numerically. Furthermore, we show that as the coarsening factor $J$ tends to infinity, the convergence factor of F-multistep parareal approaches that of the plain parareal. The proposed update also performs well for other high-order multistep methods, as illustrated by numerical experiments with BDF4 as the FP. 
A rigorous convergence analysis for higher-order BDF schemes appears substantially more involved and is left for future work. 
Finally, the linear analysis developed here serves as a key tool for studying nonlinear problems \cite{linear_conv}, for which a more refined treatment of the effect of linearization in parareal is required; we also leave this as future work.

\bibliographystyle{siam}
\bibliography{references}

\end{document}


\maketitle

\section{A detailed example}

Here we include some equations and theorem-like environments to show
how these are labeled in a supplement and can be referenced from the
main text.
Consider the following equation:
\begin{equation}
  \label{eq:suppa}
  a^2 + b^2 = c^2.
\end{equation}
You can also reference equations such as \cref{eq:matrices,eq:bb} 
from the main article in this supplement.

\lipsum[100-101]

\begin{theorem}
  An example theorem.
\end{theorem}

\lipsum[102]
 
\begin{lemma}
  An example lemma.
\end{lemma}

\lipsum[103-105]

Here is an example citation: \cite{KoMa14}.

\section[Proof of Thm]{Proof of \cref{thm:bigthm}}
\label{sec:proof}
\lipsum[106-112]

\section{Additional experimental results}
\Cref{tab:foo} shows additional
supporting evidence. 

\begin{table}[htbp]
{\footnotesize
  \caption{Example table}  \label{tab:foo}
\begin{center}
  \begin{tabular}{|c|c|c|} \hline
   Species & \bf Mean & \bf Std.~Dev. \\ \hline
    1 & 3.4 & 1.2 \\
    2 & 5.4 & 0.6 \\ \hline
  \end{tabular}
\end{center}
}
\end{table}

\newpage
\subsubsection{The case $k=0$}
We begin by eliminating the index $k+1$ on the right-hand side of the second equation in \eqref{eqn:coupled} through iterative substitution with respect to $n$:
\begin{align}\label{eqn:k=0}
E_{n+1}^{k+1}&=G_{1}^{2}E_{n-1}^{k+1}+\left( \bar{F}_{1} D_{n}^{k}+\gamma_{0} E_{n}^{k} \right) +G_{1}\left( \bar{F}_{1} D_{n-1}^{k}+\gamma_{0} E_{n-1}^{k} \right) \nonumber \\
&= G_{1}^{n+1} E_{0}^{k+1} + \sum_{i=0}^{n} G_{1}^{i} \left( \bar{F}_1 D_{n-i}^{k} + \gamma_{0} E_{n-i}^{k} \right),
\end{align}
where $E_{0}^{k+1}=U_0^{k+1} - U_0=0$. Setting $k=0$, we obtain:
\begin{align*}
E_{n+1}^{1} &= \sum_{i=0}^{n} G_{1}^{i} \left( F_{1}^{(J)} \left( E_{n-i,-1}^{0} - E_{n-i}^{0} \right) + \gamma_{0} E_{n-i}^{0} \right) \\
&= \sum_{i=0}^{n} G_{1}^{i} F_{1}^{(J)} E_{n-i,-1}^{0} + \sum_{i=0}^{n} G_{1}^{i} \left( \gamma_{0} - F_{1}^{(J)} \right) E_{n-i}^{0}.
\end{align*}
By taking the $L^2$ norm on both sides, we derive the error bound for the case $k=0$,
\begin{equation*}
\| E_{n+1}^{1}\|_{L^{2}} \leq \| F_{1}^{\left( J \right)}\|_{H} \sum_{i=0}^{n} \| E_{i,-1}^{0}\|_{L^{2}} +\| \gamma_{0} -F_{1}^{\left( J \right)}\|_{H} \sum_{i=0}^{n} \| E_{i}^{0}\|_{L^{2}},
\end{equation*}
since $\| G_1 \|_{H} \leq 1$.

\subsubsection{The case $k=1$}
We continue with \eqref{eqn:k=0} in the case $k=0$. Then using the first equation in \eqref{eqn:coupled} to eliminate $D_i^k$ on the right-hand side,
\begin{align*}
E_{n+1}^{k+1} &= \gamma_{0} \sum_{i=0}^{n} G_{1}^{i}E_{n-i}^{k} + \bar{F_{1}} \sum_{i=0}^{n} G_{1}^{i}D_{n-i}^{k} \\
&= \gamma_{0} \sum_{i=0}^{n} G_{1}^{i}E_{n-i}^{k} + \bar{F_{1}} \sum_{i=0}^{n-1} G_{1}^{i}\left( hE_{n-i-1}^{k-1} + \gamma_{1} D_{n-i-1}^{k-1} \right),
\end{align*}
where $D_0^k = E_{0,-1}^k - E_0^k=0$. 
Our next target is to further expand $E_i^k$ on the right-hand side. We have the expression for $E_{n-i}^k$,
\begin{align*}
E_{n-i}^{k} &= G_{1} E_{n-1-i}^{k} + \bar{F_{1}} D_{n-1-i}^{k-1} + \gamma_{0} E_{n-1-i}^{k-1} \\
&= G_{1}^{n-i} E_{0}^{k} + \sum_{j=0}^{n-1-i} G_{1}^{j} \left( \gamma_{0} E_{n-1-i-j}^{k-1} + \bar{F_{1}} D_{n-1-i-j}^{k-1} \right),
\end{align*}
with $E_0^k = U_0^k - U_0=0$. 
Then we continue our expression for the Parareal error $E_{n+1}^{k+1}$:
\begin{align}
E_{n+1}^{k+1} &= \gamma_{0} \sum_{i=0}^{n} G_{1}^{i} \sum_{j=0}^{n-1-i} G_{1}^{j} \left( \gamma_{0} E_{n-1-i-j}^{k-1} + \bar{F_{1}} D_{n-1-i-j}^{k-1} \right) \nonumber \\
&\quad + \bar{F_{1}} h \sum_{i=0}^{n-1} G_{1}^{i} E_{n-1-i}^{k-1} + \bar{F_{1}} \gamma_{1} \sum_{i=0}^{n-1} G_{1}^{i} D_{n-1-i}^{k-1} \nonumber \\
&= \gamma_{0}^{2} \sum_{i=0}^{n} \sum_{j=0}^{n-1-i} G_{1}^{i+j} E_{n-1-i-j}^{k-1} + \gamma_{0} \bar{F_{1}} \sum_{i=0}^{n} \sum_{j=0}^{n-1-i} G_{1}^{i+j} D_{n-1-i-j}^{k-1} \nonumber \\
&\quad + \bar{F_{1}} h \sum_{i=0}^{n-1} G_{1}^{i} E_{n-1-i}^{k-1} + \bar{F_{1}} \gamma_{1} \sum_{i=0}^{n-1} G_{1}^{i} D_{n-1-i}^{k-1} \nonumber \\
&= \sum_{m=0}^{n-1} (m+1) G_{1}^{m} \gamma_{0}^{2} E_{n-1-m}^{k-1} + \gamma_{0} \bar{F_{1}} \sum_{m=0}^{n-1} (m+1) G_{1}^{m} D_{n-1-m}^{k-1} \nonumber\\
&\quad + \bar{F_{1}} h \sum_{i=0}^{n-1} G_{1}^{i} E_{n-1-i}^{k-1} + \bar{F_{1}} \gamma_{1} \sum_{i=0}^{n-1} G_{1}^{i} D_{n-1-i}^{k-1} \nonumber\\
&= \sum_{m=0}^{n-1} G_{1}^{m} \left( (m+1) \gamma_{0}^{2} + \bar{F_{1}} h \right) E_{n-1-m}^{k-1}\nonumber \\
&\quad + \sum_{m=0}^{n-1} G_{1}^{m} \bar{F_{1}} \left( (m+1) \gamma_{0} + \gamma_{1} \right) D_{n-1-m}^{k-1}.\label{eqn:k=1}
\end{align}
Note that the right-hand side of the equality above is only related to the iteration $k-1$. Setting $k=1$ into the equality above, we then obtain
\begin{align*}
E_{n+1}^{2} &= \sum_{m=0}^{n-1} G_{1}^{m} \left( (m+1) \gamma_{0}^{2} + \bar{F}_{1} h \right) E_{n-1-m}^{0} \\
&\quad + \sum_{m=0}^{n-1} G_{1}^{m} \bar{F}_{1} \left( (m+1) \gamma_{0} + \gamma_{1} \right) \left( G_{1} G_{2}^{-1} E_{n-1-m,-1}^{0} - E_{n-1-m}^{0} \right) \\
&= \sum_{m=0}^{n-1} G_{1}^{m} \left( (m+1) \gamma_{0} \left( \gamma_{0} - \bar{F}_{1} \right) + \bar{F}_{1} \left( h - \gamma_{1} \right) \right) E_{n-1-m}^{0} \\
&\quad + \sum_{m=0}^{n-1} G_{1}^{m} \bar{F}_{1} \left( (m+1) \gamma_{0} + \gamma_{1} \right) G_{1} G_{2}^{-1} E_{n-1-m,-1}^{0}.
\end{align*}
After taking the $L^2$ norm on both sides, we then obtain
\begin{align*}
\| E_{n+1}^{2}\|_{L^{2}} &\leq \| \sum_{m=0}^{n-1} G_{1}^{m}\left( \left( m+1 \right) \gamma_{0} \left( \gamma_{0} -\bar{F}_{1} \right) +\bar{F}_{1} \left( h-\gamma_{1} \right) \right) E_{n-1-m}^{0}\|_{L^{2}} \\
&\quad +\| \sum_{m=0}^{n-1} G_{1}^{m}\bar{F}_{1} \left( (m+1)\gamma_{0} +\gamma_{1} \right) G_{1}G_{2}^{-1}E_{n-1-m,-1}^{0}\|_{L^{2}}.
\end{align*}
\red{details in the lemma}For the first term above, we have 
\begin{align*}
& \| \sum_{m=0}^{n-1} G_{1}^{m}\left( \left( m+1 \right) \gamma_{0} \left( \gamma_{0} -\bar{F}_{1} \right) +\bar{F}_{1} \left( h-\gamma_{1} \right) \right) E_{n-1-m}^{0}\|_{L^{2}} \\
& \leq \sup_{p\in \mathbb{N}^{+}} \sum_{m=0}^{n-1} |G_{1,p}^{m}\left( \left( m+1 \right) \gamma_{0,p} \left( \gamma_{0,p} -\bar{F}_{1,p} \right) +\bar{F}_{1,p} \left( h_{p}-\gamma_{1,p} \right) \right) |\sum_{i=0}^{n-1} \| E_{i}^{0}\|_{L^{2}}
\end{align*}
Similarly, for the second term, we have
\begin{align*}
& \left\| \sum_{m=0}^{n-1} G_{1}^{m}\bar{F}_{1} \left( (m+1)\gamma_{0} +\gamma_{1} \right) G_{1}G_{2}^{-1}E_{n-1-m,-1}^{0} \right\|_{L^{2}} \\
& \leq \sup_{p\in \mathbb{N}^{+}} \left| \sum_{m=0}^{n-1} G_{1,p}^{m}\bar{F}_{1,p} \left( \left( m+1 \right) \gamma_{0,p} +\gamma_{1,p} \right) G_{1,p}G_{2,p}^{-1} \right| \sum_{i=0}^{n-1} \left\| E_{i,-1}^{0} \right\|_{L^{2}}.
\end{align*}
Then for $k=1$, we have the following error estimate, 
\begin{align*}
\| E_{n+1}^{2} \|_{L^{2}} 
&\leq \sup_{p\in \mathbb{N}^{+}} \sum_{m=0}^{n-1} \left| G_{1,p}^{m}\left( (m+1) \gamma_{0,p} (\gamma_{0,p} -\bar{F}_{1,p}) + \bar{F}_{1,p} (h_{p}-\gamma_{1,p}) \right) \right| \sum_{i=0}^{n-1} \| E_{i}^{0} \|_{L^{2}} \\
&\quad + \sup_{p\in \mathbb{N}^{+}} \left| \sum_{m=0}^{n-1} G_{1,p}^{m}\bar{F}_{1,p} \left( (m+1)\gamma_{0,p} + \gamma_{1,p} \right) G_{1,p}G_{2,p}^{-1} \right| \sum_{i=0}^{n-1} \left\| E_{i,-1}^{0} \right\|_{L^{2}}.
\end{align*}

\subsubsection{The case $k=2$} 
We continue with \eqref{eqn:k=1} in the case $k=1$. Using the coupled equations \eqref{eqn:coupled}, 
\begin{align*}
E_{n+1}^{k+1} &= \sum_{m=0}^{n-1} G_{1}^{m}\left( \left( m+1 \right) \gamma_{0}^{2} + \bar{F}_{1} h \right) E_{n-1-m}^{k-1} \\
&\quad + \sum_{m=0}^{n-2} G_{1}^{m}\bar{F}_{1} \left( \left( m+1 \right) \gamma_{0} + \gamma_{1} \right) \left( \gamma_{1} D_{n-2-m}^{k-2} + h E_{n-2-m}^{k-2} \right).
\end{align*}
We then further expand $E_{n-1-m}^{k-1}$ and set $k=2$ to obtain
\begin{align*}
E_{n+1}^{3} &= \sum_{m=0}^{n-1} \sum_{j=0}^{n-2-m} G_{1}^{m+j}\left( \left( m+1 \right) \gamma_{0}^{2} +\bar{F}_{1} h \right) \left( \bar{F}_{1} D_{n-2-m-j}^{0}+\gamma_{0} E_{n-2-m-j}^{0} \right) \\
&\quad +\sum_{m=0}^{n-2} G_{1}^{m}\bar{F}_{1} \left( \left( m+1 \right) \gamma_{0} +\gamma_{1} \right) \left( \gamma_{1} D_{n-2-m}^{0}+hE_{n-2-m}^{0} \right) \\
&= \sum_{i=0}^{n-2} G_{1}^{i}\left( \sum_{m=0}^{i} \left( \left( m+1 \right) \gamma_{0}^{2} +\bar{F}_{1} h \right) \right) \left( \bar{F}_{1} D_{n-2-i}^{0}+\gamma_{0} E_{n-2-i}^{0} \right) \\
&\quad +\sum_{m=0}^{n-2} G_{1}^{m}\bar{F}_{1} \left( \left( m+1 \right) \gamma_{0} +\gamma_{1} \right) \left( \gamma_{1} D_{n-2-m}^{0}+hE_{n-2-m}^{0} \right) \\
&= \sum_{i=0}^{n-2} G_{1}^{i} \left( \bar{F}_{1} \left( \left( i+1 \right) \left( \frac{\gamma_{0}^{2} \left( i+2 \right)}{2} +\bar{F}_{1} h+\gamma_{0} \gamma_{1} \right) +\gamma_{1}^{2} \right) D_{n-2-i}^{0} \right. \\
&\quad \left. + \left( \left( i+1 \right) \frac{\gamma_{0}^{3} \left( i+2 \right)}{2} +2\left( i+1 \right) \bar{F}_{1} h\gamma_{0} +\bar{F}_{1} h\gamma_{1} \right) E_{n-2-i}^{0} \right).
\end{align*}

\newpage
\section{Linear Convergence for the Linear Parabolic Problem}

\subsection{Preliminaries}
Consider the following linear parabolic problem
\begin{equation*}
		\left \{
		\begin{aligned}
			u' (t)+ A u(t) &= f(t), \quad 0<t<T,\\
			 u(0)&=u_0 .
		\end{aligned}\right .
\end{equation*}
Thanks to Lemma 10.3 in \cite{thomee2007galerkin}, we are able to give the explicit formula of our FPs:
\begin{align*}
&F\left( T_{0}+n\Delta t,T_{0}+(q-1)\Delta t,u_{0},\cdots ,u_{q-1} \right) \\
    &= \sum_{s=0}^{q-1} \beta_{ns} \left( \Delta tA \right) u_{s} + \Delta t\sum_{j=q}^{n} \beta_{n-j} \left( \Delta tA \right) f\left( T_{0}+j\Delta t \right) \\
    &:= \sum_{s=0}^{q-1} F_{s}^{\left( n-q+1 \right)}(\Delta t A) u_{s} + N^{\left( n-q+1 \right)}\left( f \right) \left( T_{0}+(q-1)\Delta t \right),
\end{align*}
where the \(\beta_j(\lambda)\) and \(\beta_{ns}(\lambda)\) are defined by

$$\widetilde{\beta}(s)=\sum_{j = 0}^{\infty}\beta_j(\lambda)\zeta^j := (\widetilde{\alpha}(\zeta)+\lambda)^{-1}, \quad \beta_{ns}(\lambda)=-\sum_{j = q - s}^{q}\beta_{n - s - j}(\lambda)\alpha_j,$$
with $\tilde{\alpha} \left( \zeta \right) =\alpha_{j} \zeta^{j}$, where the $\alpha_j$ are the coefficients defined by $\tilde{\partial_{q}} U^{n}=\frac{1}{\Delta t} \sum_{j=0}^{q} \alpha_{j} U^{n-j}$.

For convenience, we define several rational functions,  
\begin{equation*}
    h_1(z) = \max_{0 \leq i \leq q-1} \left| \sum_{s=0}^{q-1} \Big( F_{s}^{(J-i)}(z) - F_{s}^{(J)}(z) \Big) \right|,~\gamma_1(z) = \max_{0 \leq i \leq q-1}  \left| F_{s}^{(J-i)}(z) - F_{s}^{(J)}(z) \right|,  
\end{equation*}
\begin{equation*}
    \gamma_0 (z) = \sum_{s=0}^{q-1} F_{s}^{\left( J \right)}\left( z \right) -R\left( Jz \right),~C^{(J)} (z)=\max_{1\leq s\leq q-1} |F_{s}^{\left( J \right)}\left( z \right) |.
\end{equation*}

\newpage
Without loss of generality, we consider the linear problem, $u'+ Au=0$ and the BDF2 scheme as FPs in the Parareal algorithm. Then the multi-step Parareal can be expressed as 
\begin{align} 
\left\{
\begin{aligned}
U_{n+1}^{k+1}&=G_{1}\left( U_{n}^{k+1} \right) +F_{1}\left( U_{n,-1}^{k} \right) +F_{2}\left( U_{n}^{k} \right) -G_{1}\left( U_{n}^{k} \right), \\
U_{n+1,-1}^{k+1}&=G_{2}\left( U_{n}^{k+1} \right) +F_{1}^{\prime}\left( U_{n,-1}^{k} \right) +F_{2}'\left( U_{n}^{k} \right) -G_{2}\left( U_{n}^{k} \right).
\end{aligned}
\right.
\end{align}
The exact solution also satisfies \begin{align*}
\left\{
\begin{aligned}
U_{n+1}      &= F_{1}\left( U_{n,-1} \right) + F_{2}\left( U_{n} \right), \\
U_{n+1,-1} &= F_{1}^{\prime}\left( U_{n,-1} \right) + F_{2}^{\prime}\left( U_{n} \right).
\end{aligned}
\right.
\end{align*} 
Thus, we define the Parareal error $E_n^k = U_n^k - U_n$ and $E_{n,-1}^k = U_{n,-1}^k - U_{n,-1}$, which satisfy
\begin{align}
\left\{
\begin{aligned}
E_{n+1}^{k+1}&=G_{1}\left( E_{n}^{k+1} \right) +F_{1}\left( E_{n,-1}^{k} \right) +F_{2}\left( E_{n}^{k} \right) -G_{1}\left( E_{n}^{k} \right), \\
E_{n+1,-1}^{k+1}&=G_{2}\left( E_{n}^{k+1} \right) +F_{1}^{\prime}\left( E_{n,-1}^{k} \right) +F_{2}'\left( E_{n}^{k} \right) -G_{2}\left( E_{n}^{k} \right).
\end{aligned}
\right.
\end{align}
If $G_2$ is invertible, the second equation in \eqref{eqn:para_err} is equal to 
\begin{equation*}
G_{1}G_{2}^{-1}\left( E_{n+1,-1}^{k+1} \right) =G_{1}\left( E_{n}^{k+1} \right) +G_{1}G_{2}^{-1}F_{1}^{\prime}\left( E_{n,-1}^{k} \right) +G_{1}G_{2}^{-1}F_{2}^{\prime}\left( E_{n}^{k} \right) -G_{1}\left( E_{n}^{k} \right).
\end{equation*}
Given that $F_1',~F_2'$, and $G_1 G_2^{-1}$ mutually commute, this can be rewritten as 
\begin{equation*}
    G_{1}G_{2}^{-1}\left( E_{n+1,-1}^{k+1} \right) =G_{1}\left( E_{n}^{k+1} \right) +F_{1}^{\prime}G_{1}G_{2}^{-1}\left( E_{n,-1}^{k} \right) +F_{2}^{\prime}G_{1}G_{2}^{-1}\left( E_{n}^{k} \right) -G_{1}\left( E_{n}^{k} \right).
\end{equation*}
We further introduce one new variable, $D_{n}^{k}=G_{1}G_{2}^{-1}\left( E_{n,-1}^{k} \right) -E_{n}^{k}$, which satisfies
\begin{align}
D_{n+1}^{k+1} &= \left( F_{1}^{\prime}G_{1}G_{2}^{-1}-F_{1} \right) \left( E_{n,-1}^{k} \right) +\left( F_{2}'{ G_{1}G_{2}^{-1}}-F_{2} \right) \left( E_{n}^{k} \right) \nonumber \\
&= \left( F_{1}^{\prime}-F_{1}G_{2}G_{1}^{-1} \right) \left( G_{1}G_{2}^{-1}\left( E_{n,-1}^{k} \right) -E_{n}^{k} \right) \nonumber \\
&\quad +\left( F_{1}^{\prime}+F_{2}^{\prime}G_{1}G_{2}^{-1}-F_{1}G_{2}G_{1}^{-1}-F_{2} \right) \left( E_{n}^{k} \right) \nonumber \\
&:= \gamma_{1} D_{n}^{k}+hE_{n}^{k}.
\end{align}
The Parareal error $E_{n+1}^{k+1}$ can also be related with $D_n^k$ through
\begin{align*}
E_{n+1}^{k+1} &= G_{1}\left( E_{n}^{k+1} \right) + F_{1}G_{2}G_{1}^{-1}\left( G_{1}G_{2}^{-1}\left( E_{n,-1}^{k} \right) - E_{n}^{k} \right) \\
&\quad + \left( F_{1}G_{2}G_{1}^{-1} + F_{2} - G_1 \right) \left( E_{n}^{k} \right) \\
&:= G_{1} E_{n}^{k+1} + \bar{F_1} D_{n}^{k} + \gamma_{0} E_{n}^{k}. 
\end{align*}
Then we derive the coupled equations between $E_n^k$ and $D_n^k$:
\begin{equation}
    \left\{
\begin{aligned}
D_{n+1}^{k+1} &= \gamma_{1} D_{n}^{k} + h E_{n}^{k}, \\
E_{n+1}^{k+1} &= G_{1} E_{n}^{k+1} + \bar{F}_{1} D_{n}^{k} + \gamma_{0} E_{n}^{k}.
\end{aligned}
\right.
\end{equation}
Our further estimate bases on these coupled equations. Now we start from eliminating the index $k+1$ of the right-hand side of the second equation in \eqref{eqn:coupled} by iterating itself on $n$,
\begin{align*}
E_{n+1}^{k+1} &= G_{1}^{2} E_{n-1}^{k+1}  + \sum_{i=0}^{1} G_{1}^{i}\left( \bar{F}_1 D_{n-i}^{k} + \gamma_{0} E_{n-i}^{k} \right) \\
&= G_{1}^{n+1} E_{0}^{k+1}  + \sum_{i=0}^{n} G_{1}^{i}\left( \bar{F}_1 D_{n-i}^{k} + \gamma_{0} E_{n-i}^{k} \right),
\end{align*}
where $E_{0}^{k+1}=U_{0}^{k+1}-U_{0}=0$. Then using the first equation in \eqref{eqn:coupled} to eliminate $D_i^k$ on the right-hand side,
\begin{align*}
E_{n+1}^{k+1} &= \gamma_{0} \sum_{i=0}^{n} G_{1}^{i}E_{n-i}^{k} + \bar{F_{1}} \sum_{i=0}^{n} G_{1}^{i}D_{n-i}^{k} \\
&= \gamma_{0} \sum_{i=0}^{n} G_{1}^{i}E_{n-i}^{k} + \bar{F_{1}} \sum_{i=0}^{n} G_{1}^{i}\left( hE_{n-i-1}^{k-1} + \gamma_{1} D_{n-i-1}^{k-1} \right).
\end{align*}
Our next target is to further expand $E_i^k$ on the right-hand side, because the numerical experiments indicate that we should derive staggered time steps, and if we don't follow this rule, our estimate is not accurate enough. We have the expression for $E_{n-i}^k$,
\begin{align*}
E_{n-i}^{k} &= G_{1} E_{n-1-i}^{k} + \bar{F_{1}} D_{n-1-i}^{k-1} + \gamma_{0} E_{n-1-i}^{k-1} \\
&= G_{1}^{n-i} E_{0}^{k} + \sum_{j=0}^{n-i} G_{1}^{j} \left( \gamma_{0} E_{n-1-i-j}^{k-1} + \bar{F_{1}} D_{n-1-i-j}^{k-1} \right),
\end{align*}
with $E_0^k = U_0^k - U_0=0$. 
Then we continue our expression for the Parareal error $E_{n+1}^{k+1}$:
\begin{align*}
E_{n+1}^{k+1} &= \gamma_{0} \sum_{i=0}^{n} G_{1}^{i} \sum_{j=0}^{n-i} G_{1}^{j} \left( \gamma_{0} E_{n-1-i-j}^{k-1} + \bar{F_{1}} D_{n-1-i-j}^{k-1} \right) \\
&\quad + \bar{F_{1}} h \sum_{i=0}^{n} G_{1}^{i} E_{n-1-i}^{k-1} + \bar{F_{1}} \gamma_{1} \sum_{i=0}^{n} G_{1}^{i} D_{n-1-i}^{k-1} \\
&= \gamma_{0}^{2} \sum_{i=0}^{n} \sum_{j=0}^{n-i} G_{1}^{i+j} E_{n-1-i-j}^{k-1} + \gamma_{0} \bar{F_{1}} \sum_{i=0}^{n} \sum_{j=0}^{n-i} G_{1}^{i+j} D_{n-1-i-j}^{k-1} \\
&\quad + \bar{F_{1}} h \sum_{i=0}^{n} G_{1}^{i} E_{n-1-i}^{k-1} + \bar{F_{1}} \gamma_{1} \sum_{i=0}^{n} G_{1}^{i} D_{n-1-i}^{k-1} \\
&= \gamma_{0} \sum_{m=0}^{n} (m+1) G_{1}^{m} \gamma_{0}^{2} E_{n-1-m}^{k-1} + \gamma_{0} \bar{F_{1}} \sum_{m=0}^{n} (m+1) G_{1}^{m} D_{n-1-m}^{k-1} \\
&\quad + \bar{F_{1}} h \sum_{i=0}^{n} G_{1}^{i} E_{n-1-i}^{k-1} + \bar{F_{1}} \gamma_{1} \sum_{i=0}^{n} G_{1}^{i} D_{n-1-i}^{k-1} \\
&= \sum_{m=0}^{n} G_{1}^{m} \left( (m+1) \gamma_{0}^{2} + \bar{F_{1}} h \right) E_{n-1-m}^{k-1} \\
&\quad + \sum_{m=0}^{n} G_{1}^{m} \bar{F_{1}} \left( (m+1) \gamma_{0} + \gamma_{1} \right) D_{n-1-m}^{k-1}.
\end{align*}
Note that the right-hand side of the equality above is only related to the iteration $k-1$.\red{$12345$} Now we further expand $D_{n-1-m}^{k-1}$ to iteration $k-3$ based on \eqref{eqn:coupled}: 
\begin{align*}
D_{n-1-m}^{k-1} &= h E_{n-2-m}^{k-2} + \gamma_{1} D_{n-2-m}^{k-2} \\
&= h \left( G_{1} E_{n-3-m}^{k-2} + \gamma_{0} E_{n-3-m}^{k-3} + \bar{F_{1}} D_{n-3-m}^{k-3} \right) + \gamma_{1} D_{n-2-m}^{k-2} \\
&= h \sum_{j=0}^{n-3-m} G_{1}^{j} \left( \gamma_{0} E_{n-3-m-j}^{k-3} + \bar{F_{1}} D_{n-3-m-j}^{k-3} \right) + \gamma_{1} D_{n-2-m}^{k-2}.
\end{align*}
Then we further expand $D_{n-2-m}^{k-2}$ to eliminate iteration $k-2$ on the right-hand side of the above equation, 
\begin{align*}
D_{n-1-m}^{k-1} &= h \sum_{j=0}^{n-3-m} G_{1}^{j} \left( \gamma_{0} E_{n-3-m-j}^{k-3} + \bar{F_{1}} D_{n-3-m-j}^{k-3} \right) + \gamma_{1} h E_{n-3-m}^{k-3} + \gamma_{1}^{2} D_{n-3-m}^{k-3} \\
&= h \gamma_{0} \sum_{j=0}^{n-3-m} G_{1}^{j} E_{n-3-m-j}^{k-3} + \gamma_{1} h E_{n-3-m}^{k-3} + h \bar{F_{1}} \sum_{j=0}^{n-3-m} G_{1}^{j} D_{n-3-m-j}^{k-3} + \gamma_{1}^{2} D_{n-3-m}^{k-3}. 
\end{align*}
Finally, we take the expression of $D_{n-1-m}^{k-1}$ back to the Parareal error $E_{n+1}^{k+1}$,  
\begin{align*}
E_{n+1}^{k+1} &= \sum_{m=0}^{n} G_{1}^{m} \left( (m+1) \gamma_{0}^{2} + \bar{F_{1}} h \right) E_{n-1-m}^{k-1} \\
&\quad + \sum_{m=0}^{n} G_{1}^{m} \bar{F_{1}} \left( (m+1) \gamma_{0} + \gamma_{1} \right) h \left( \gamma_{0} \sum_{j=0}^{n-3-m} G_{1}^{j} E_{n-3-m-j}^{k-3} + \gamma_{1} E_{n-3-m}^{k-3} \right) \\
&\quad + \sum_{m=0}^{n} G_{1}^{m} \bar{F_{1}} \left( (m+1) \gamma_{0} + \gamma_{1} \right) \left( h \bar{F_{1}} \sum_{j=0}^{n-3-m} G_{1}^{j} D_{n-3-m-j}^{k-3} + \gamma_{1}^{2} D_{n-3-m}^{k-3} \right).
\end{align*}
Now we take a close look at the equation above. The first row involves $E_i^{k-1}$, the second row involves $E_i^{k-3}$, and the third row involves $D_i^{k-3}$. To analyze the $L^2$ norm of $E_{n+1}^{k+1}$, we consider the $p$-th spectrum of the elliptic operator $A$ and take the inner product with the eigenfunction $\lambda_p$.  
\begin{align*}
(E_{n+1}^{k+1},\lambda_p) &= \sum_{m=0}^{n} G_{1}^{m} \left( (m+1) \gamma_{0}^{2} + \bar{F_{1}} h \right) (E_{n-1-m}^{k-1},\lambda_p)  \\
&\quad + \sum_{m=0}^{n} G_{1}^{m} \bar{F_{1}} \left( (m+1) \gamma_{0} + \gamma_{1} \right) h \left( \gamma_{0} \sum_{j=0}^{n-3-m} G_{1}^{j} (E_{n-3-m-j}^{k-3},\lambda_p) + \gamma_{1} (E_{n-3-m}^{k-3},\lambda_p) \right) \\
&\quad + \sum_{m=0}^{n} G_{1}^{m} \bar{F_{1}} \left( (m+1) \gamma_{0} + \gamma_{1} \right) \left( h \bar{F_{1}} \sum_{j=0}^{n-3-m} G_{1}^{j} (D_{n-3-m-j}^{k-3},\lambda_p) + \gamma_{1}^{2} (D_{n-3-m}^{k-3},\lambda_p)  \right).
\end{align*}
Let $e_n^k=(E_n^k,\lambda_p)$ and $d_n^k = (D_n^k,\lambda_p)$. We omit the explicit dependence of $e_n^k$, $d_n^k$ and all the operators appear above on $p$ for brevity. 
\begin{align*}
e_{n+1}^{k+1} &= \sum_{m=0}^{n} G_{1}^{m} \left( (m+1) \gamma_{0}^{2} + \bar{F_{1}} h \right) e_{n-1-m}^{k-1} \\
&\quad + \sum_{m=0}^{n} G_{1}^{m} \bar{F_{1}} \left( (m+1) \gamma_{0} + \gamma_{1} \right) h \left( \gamma_{0} \sum_{j=0}^{n-3-m} G_{1}^{j} e_{n-3-m-j}^{k-3} + \gamma_{1} e_{n-3-m}^{k-3} \right) \\
&\quad + \sum_{m=0}^{n} G_{1}^{m} \bar{F_{1}} \left( (m+1) \gamma_{0} + \gamma_{1} \right) \left( h \bar{F_{1}} \sum_{j=0}^{n-3-m} G_{1}^{j} d_{n-3-m-j}^{k-3} + \gamma_{1}^{2} d_{n-3-m}^{k-3} \right).
\end{align*}

\textbf{(Term I)} This term represents the primary error component of $E_{n+1}^{k+1}$, as it incorporates the complete error propagation from the $(k-1)$-th iteration. 
\begin{align*}
\left| \sum_{m=0}^{n} G_{1}^{m}\left( (m+1)\gamma_{0}^{2} +\bar{F_{1}} h \right) e_{n-1-m}^{k-1}\right|
&\leq \max_{i} |e_{i}^{k-1}|\sum_{m=0}^{n} |G_{1}|^{m}\left| (m+1) \gamma_{0}^{2} +\bar{F_{1}} h\right| \\
&\leq \gamma_{a,p} \max_{i} |e_{i}^{k-1}|.
\end{align*}
We have to make this bound clear: 
\begin{equation}\label{eqn:g_a}
\gamma_{a,p}:=\sum_{m=0}^{n} |G_{1}|^{m}\left| (m+1) \gamma_{0,p}^{2} +\bar{F_{1}} h_p\right|. 
\end{equation}

\begin{equation*}
\gamma_{a,p} :=\left( \frac{\gamma_{0,p}}{1-|G_{1}\left( Jz_{p} \right) |} \right)^{2} +\frac{|\bar{F}_{1,p} h_{p}|}{1-|G_{1}\left( Jz_{p} \right) |},
\end{equation*}
where $z_p = \Delta t\lambda_p$ and 
\begin{align*}
\gamma_{0,p} &= \left| \left( F_{1}(z_{p}) G_{2}(Jz_{p}) G_{1}^{-1}(Jz_{p}) + F_{2}(z_{p}) \right) - G_{1}(Jz_{p}) \right|, \\
\bar{F}_{1,p} &= |F_{1}(z_{p}) G_{2}(Jz_{p}) G_{1}^{-1}(Jz_{p})|, \\
h_{p} &= \left| F_{1}'(z_{p}) + F_{2}'(z_{p}) G_{1}(Jz_{p})G_2^{-1}(Jz_p) - F_{1}(z_{p}) G_{2}(Jz_{p}) G_{1}^{-1}(Jz_{p}) - F_{2}(z_{p}) \right|.
\end{align*}

\textbf{(Term II)} We bound for the second summation: 
\begin{align*}
&\left| h\gamma_{0} \bar{F_{1}} \sum_{m=0}^{n} \sum_{j=0}^{n-3-m} G_{1}^{m+j}\left( \left( m+1 \right) \gamma_{0} +\gamma_{1} \right) e_{n-3-m-j}^{k-3} \right. \\
&\quad \left. + h\gamma_{1} \bar{F_{1}} \sum_{m=0}^{n} G_{1}^{m}\left( \left( m+1 \right) \gamma_{0} +\gamma_{1} \right) e_{n-3-m}^{k-3} \right| \\
&\leq \max_{i} |e_{i}^{k-3}| \left( |h\gamma_{0} \bar{F_{1}} |\sum_{m=0}^{n} \sum_{j=0}^{n-3-m} |G_{1}|^{m+j}|\left( m+1 \right) \gamma_{0} +\gamma_{1} | \right) \\
&\quad + \max_{i} |e_{i}^{k-3}| \left( |h\gamma_{1} \bar{F_{1}} |\sum_{m=0}^{n} |G_{1}|^{m}|\left( m+1 \right) \gamma_{0} +\gamma_{1} | \right) \\
&\leq \left( \frac{|\bar{F_{1}} h\gamma_{0}^{2} |}{\left( 1-|G_{1}| \right)^{3}} + \frac{|\bar{F_{1}} h\gamma_{0} \gamma_{1} |}{\left( 1-|G_{1}| \right)^{2}} \right) \max_{i} |e_{i}^{k-3}| \\
&\quad + \left( \frac{|\bar{F_{1}} h\gamma_{0} \gamma_{1} |}{\left( 1-|G_{1}| \right)^{2}} + \frac{|\bar{F_{1}} h\gamma_{1}^{2} |}{1-|G_{1}|} \right) \max_{i} |e_{i}^{k-3}| \\
&\leq  \frac{|\bar{F_{1}} h|}{1-|G_{1}|} \left( |\gamma_{1} |+\frac{|\gamma_{0} |}{1-|G_{1}|} \right)^{2} \max_{i} |e_{i}^{k-3}|.
\end{align*}
We should make the bound clear: 
\begin{equation}\label{eqn:g_b}
    \gamma_{b,p} :=\frac{\bar{F}_{1,p} h_{p}}{1-|G_{1}\left( Jz_{p} \right) |} \left( \gamma_{1,p} +\frac{\gamma_{0,p}}{1-|G_{1}\left( Jz_{p} \right) |} \right)^{2},
\end{equation}
where 
\begin{equation*}
  \gamma_{1,p} =|F_{1}^{\prime}(z_{p})-F_{1}(z_{p})G_{2}(Jz_{p})G_{1}^{-1}(Jz_{p})|.
\end{equation*}

\textbf{(Term III)} 
We analyze the final summation term, specifically examining the misalignment-induced error $d_i^{k-3}$. Our approach involves expressing $d_i^{k+1}$ in terms of $d_i^{k-1}$ and $e_i^{k-1}$, which enables us to establish an upper bound for $d_i^{k+1}$. We start from \eqref{eqn:coupled},
\begin{align*}
d_{n+1}^{k+1} &= h e_{n}^{k} + \gamma_{1} d_{n}^{k} \\
&= h \left( G_{1} e_{n-1}^{k} + \gamma_{0} e_{n-1}^{k-1} + \bar{F_{1}} d_{n-1}^{k-1} \right) + \gamma_{1} d_{n}^{k} \\
&= h \sum_{j=0}^{n-1} G_{1}^{j} \left( \gamma_{0} e_{n-1-j}^{k-1} + \bar{F_{1}} d_{n-1-j}^{k-1} \right) + +\gamma_{1} d_{n}^{k}.
\end{align*}
Then we further expand $d_n^k$ with $d_{n}^{k}=he_{n-1}^{k-1}+\gamma_{1} d_{n-1}^{k-1}$,
\begin{align*}
d_{n+1}^{k+1} &= h\sum_{j=0}^{n-1} G_{1}^{j}\left(\gamma_{0} e_{n-1-j}^{k-1} + \bar{F_{1}} d_{n-1-j}^{k-1}\right) + \gamma_{1}\left(he_{n-1}^{k-1} + \gamma_{1} d_{n-1}^{k-1}\right) \\
&= h\gamma_{0} \sum_{j=1}^{n-1} G_{1}^{j}e_{n-1-j}^{k-1} + h(\gamma_{1}+\gamma_0) e_{n-1}^{k-1}+ h\bar{F_{1}} \sum_{j=1}^{n-1} G_{1}^{j}d_{n-1-j}^{k-1} + (h\bar{F}_1+\gamma_{1}^{2}) d_{n-1}^{k-1}.
\end{align*}
We take the absolute norm on both sides and derive the bound for $d_{n+1}^{k+1}$, 
\begin{align*}
    |d_{n+1}^{k+1}| &\leq \left( \frac{|G_{1} h \gamma_{0}|}{1 - |G_{1}|} + |h (\gamma_{0} + \gamma_{1})| \right) \max_{i} |e_{i}^{k-1}| 
+ \left( \frac{|G_{1} h \bar{F}_{1}|}{1 - |G_{1}|} + |h \bar{F}_{1} + \gamma_{1}^{2}| \right) \max_{i} |d_{i}^{k-1}|\\
&:= \gamma_{d,p}\max_{i}|e_i^{k-1}| + \gamma_{e,p} \max_{i} |d_i^{k-1}|,
\end{align*}
where we define the linear bounds, 
\begin{equation}\label{eqn:g_d,d_e}
    \gamma_{d,p} =\frac{|G_{1}\left( Jz_{p} \right) h_{p}\gamma_{0,p} |}{1-|G_{1}\left( Jz_{p} \right) |} +|h_{p}\left( \gamma_{0,p} +\gamma_{1,p} \right) |,~\gamma_{e,p} =\frac{|G_{1}\left( Jz_{p} \right) h_{p}\bar{F}_{1,p} |}{1-|G_{1}\left( Jz_{p} \right) |} +|h_{p}\bar{F}_{1,p} +\gamma_{1,p}^{2} |.
\end{equation}
Then we derive the bound for the last summation, 
\begin{align*}
&\left| \sum_{m=0}^{n} G_{1}^{m}\bar{F}_{1} \left( (m+1)\gamma_{0} + \gamma_{1} \right) \left( h\bar{F}_{1} \sum_{j=0}^{n-3-m} G_{1}^{j}d_{n-3-m-j}^{k-3} + \gamma_{1}^{2} d_{n-3-m}^{k-3} \right) \right| \\
&\leq \left| h\bar{F}_{1}^{2} \sum_{m=0}^{n} \sum_{j=0}^{n-3-m} G_{1}^{m+j} \left( (m+1)\gamma_{0} + \gamma_{1} \right) d_{n-3-m-j}^{k-3} \right| \\
&\quad + \left| \gamma_{1}^{2} \bar{F}_{1} \sum_{m=0}^{n} G_{1}^{m} \left( (m+1)\gamma_{0} + \gamma_{1} \right) d_{n-3-m}^{k-3} \right| \\
&\leq \left( \frac{|h\bar{F}_{1}^{2} \gamma_{0}|}{(1-|G_{1}|)^{3}} + \frac{|h\bar{F}_{1}^{2} \gamma_{1}|}{(1-|G_{1}|)^{2}} \right) \max_{i} |d_{i}^{k-3}|  + \left( \frac{|\bar{F}_{1} \gamma_{0} \gamma_{1}^{2}|}{(1-|G_{1}|)^{2}} + \frac{|\bar{F}_{1} \gamma_{1}^{3}|}{1-|G_{1}|} \right) \max_{i} |d_{i}^{k-3}|\\
&=  \left( \frac{|h\bar{F}_{1}^{2} \gamma_{0}|}{(1-|G_{1}|)^{3}} + \frac{|h\bar{F}_{1}^{2} \gamma_{1}|}{(1-|G_{1}|)^{2}} + \frac{|\bar{F}_{1} \gamma_{0} \gamma_{1}^{2}|}{(1-|G_{1}|)^{2}} + \frac{|\bar{F}_{1} \gamma_{1}^{3}|}{1-|G_{1}|} \right) \max_{i} |d_{i}^{k-3}|.
\end{align*}
Then we clarify the bound, but this is a little bit complicated, 
\begin{equation}\label{eqn:g_c}
    \gamma_{c,p} =\frac{h_{p}\bar{F}_{1,p}^{2} \gamma_{0,p}}{\left( 1-|G_{1}\left( Jz_{p} \right) | \right)^{3}} +\frac{h_{p}\bar{F}_{1,p}^{2} \gamma_{1,p}}{\left( 1-|G_{1}\left( Jz_{p} \right) | \right)^{2}} +\frac{\bar{F}_{1,p} \gamma_{0,p} \gamma_{1,p}^{2}}{\left( 1-|G_{1}\left( Jz_{p} \right) | \right)^{2}} +\frac{\bar{F}_{1,p} \gamma_{1,p}^{3}}{1-|G_{1}\left( Jz_{p} \right) |}.
\end{equation}

Finally, we arrive at the iteration on index $k$, based on the estimates of the three terms: 
\begin{align*}
\left\{
\begin{aligned}
|e_{n+1}^{k+1}| &\leq \gamma_{a,p} \max_{i} |e_{i}^{k-1}| + \gamma_{b,p} \max_{i} |e_{i}^{k-3}| + \gamma_{c,p} \max_{i} |d_{i}^{k-3}|, \\
|d_{n}^{k-1}|   &\leq \gamma_{d,p} \max_{i} |e_{i}^{k-3}| + \gamma_{e,p} \max_{i} |d_{i}^{k-3}|,
\end{aligned}
\right.
\end{align*}
where the left-hand side of the first equation is independent of $n$, then we take the maximum over $n$. We further introduce $x_k=\max_{i}|e_i^{k}|$ and $y_k = \max_{i} |d_i^k|$ and obtain
\begin{align}\label{eqn:bound}
\left\{
\begin{aligned}
x_{k+1} &\leq \gamma_{a,p} x_{k-1} + \gamma_{b,p} x_{k-3} + \gamma_{c,p} y_{k-3}, \\
y_{k-1} &\leq \gamma_{d,p} x_{k-3} + \gamma_{e,p} y_{k-3}.
\end{aligned}
\right.
\end{align}
Next we study the recursion, 
\begin{align*}
x_{k+1} &\leq \gamma_{a,p} x_{k-1}+\gamma_{b,p} x_{k-3}+\gamma_{c,p} \left( \gamma_{d,p} x_{k-5}+\gamma_{e,p} y_{k-5} \right) \\
&\leq \gamma_{a,p} x_{k-1}+\gamma_{b,p} x_{k-3}+\gamma_{c,p} \gamma_{d,p} x_{k-5}+\gamma_{c,p} \gamma_{e,p} y_{k-5} \\
&\leq \gamma_{a,p} x_{k-1}+\gamma_{b,p} x_{k-3}+\gamma_{c,p} \gamma_{d,p} x_{k-5}+\gamma_{c,p} \gamma_{e,p} \left( \gamma_{d,p} x_{k-7}+\gamma_{e,p} y_{k-7} \right) \\
&\leq \gamma_{a,p} x_{k-1}+\gamma_{b,p} x_{k-3}+\gamma_{c,p} \gamma_{d,p} \left( x_{k-5}+\gamma_{e,p} x_{k-7} \right) +\gamma_{c,p} \gamma_{e,p}^{2} y_{k-7} \\
&\leq \gamma_{a,p} x_{k-1}+\gamma_{b,p} x_{k-3}+\gamma_{c,p} \gamma_{d,p} \sum_{i=0}^{\lfloor \frac{k-5}{2} \rfloor} \gamma_{e,p}^{i} x_{k-5-2i}+\gamma_{c,p} \gamma_{e,p}^{\lfloor \frac{k-5}{2} \rfloor +1} y_{k-5-2\lfloor \frac{k-5}{2} \rfloor}.
\end{align*}

\subsubsection{Special case}
If $G_2$ is not invertible in \eqref{eqn:para_err}, we first take the inner product of both sides with $\lambda_p$,
\begin{align*}
\left\{
\begin{aligned}
e_{n+1}^{k+1} &= G_{1}(Jz_{p}) e_{n}^{k+1} + F_{1}(z_{p}) e_{n,-1}^{k} + F_{2}(z_{p}) e_{n}^{k} - G_{2}(Jz_{p}) e_{n}^{k}, \\
e_{n+1,-1}^{k+1} &= G_{2}(Jz_{p}) e_{n}^{k+1} + F_{1}'(z_{p}) e_{n,-1}^{k} + F_{2}'(z_{p}) e_{n}^{k} - G_{2}(Jz_{p}) e_{n}^{k}.
\end{aligned}
\right.
\end{align*}
We consider the case when $G_2(Jz_p)=0$, then the equations above reduce to XXX. \red{also consider G1=0}

\subsection{Comparison} 
The theoretical estimate is almost done, only the last step left. In this section, we compare the convergence rate between the multi-step Parareal by Maday and by our newly proposed update, through \eqref{eqn:bound}. In Maday's work, they choose $G_1 = G_2=R(\Delta TA)$ in \eqref{eqn:m-para}. We propose a more reasonable update: $G_1=R(\Delta TA)$ and $G_2 = R((\Delta T-\Delta t)A)$ in \eqref{eqn:m-para}. Next, we will illustrate these two algorithms through the recursion \eqref{eqn:bound}. 

\subsubsection{$\gamma_{a,p}$}
The expression of the convergence function $\gamma_{a,p}$ is given in \eqref{eqn:g_a}. This can be further bounded by 
\begin{equation*}
\gamma_{a,p} \leq\left( \frac{\gamma_{0,p}}{1-|G_{1}\left( Jz_{p} \right) |} \right)^{2} +\frac{|\bar{F}_{1,p} h_{p}|}{1-|G_{1}\left( Jz_{p} \right) |}.
\end{equation*}
However, this bound is not sharp when $\gamma_{0}^2(z_p)$ and $\bar{F}_1(z_p)h_p(z_p)$ have opposite signs. Then we plot the value of $\gamma_{a,p}$ when $n=1000$ to illustrate the behavior of different updates. In Fig.~\ref{fig:gamma_a_BE}, we plot the graph of $\gamma_a$ when CP is BE. As $J$ increases, the supremum of $\gamma_a$ also becomes larger. In this case, the benefit of the new update can not be observed since BE is not an accuracy solver. For comparison, in Fig.~\ref{fig:gamma_a_Exact}, we plot the graph of $\gamma_a (z)$ when we choose the exact solver as the CP, i.e., $R(z)=\text{exp}(-z)$. Different from the case for BE, as $J$ increases, the supremum of $\gamma_a$ decreases for both updates. Note that the supremum in the new update is much smaller than that one in the Maday's update, which implies that our new update is more reasonable. This can be explained through the important gradient in $\gamma_a$: $\gamma_0$ defined in \eqref{eqn:g_0},
\begin{align*}
    \gamma_{0} \left( z \right) &=F_{1}\left( z \right) +F_{2}\left( z \right) -\text{exp} \left( -Jz \right);&~\text{(Maday)}\\
    \gamma_{0} \left( z \right) &=F_{1}\left( z \right) \text{exp} \left( z \right) +F_{2}\left( z \right) -\text{exp} \left( -Jz \right);&~\text{(New)}.
\end{align*}
Obviously, the new update is more reasonable since $F_1$ and $F_2$ should not be in the same position. 

\begin{figure}[htbp]
    \centering
    \begin{subfigure}[b]{0.48\textwidth}
\includegraphics[width=\linewidth]{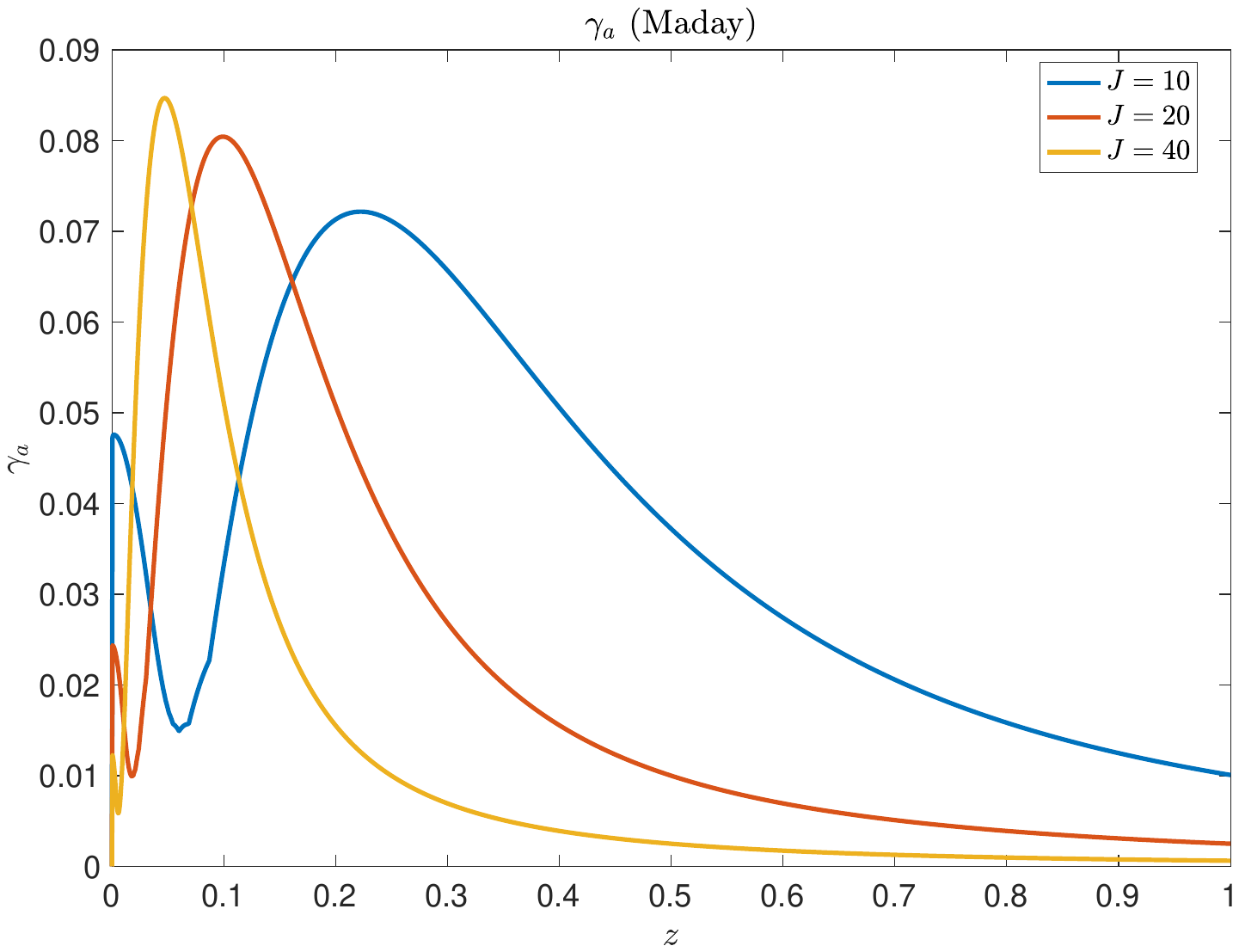}
    \end{subfigure}
    \begin{subfigure}[b]{0.48\textwidth}
        \includegraphics[width=\linewidth]{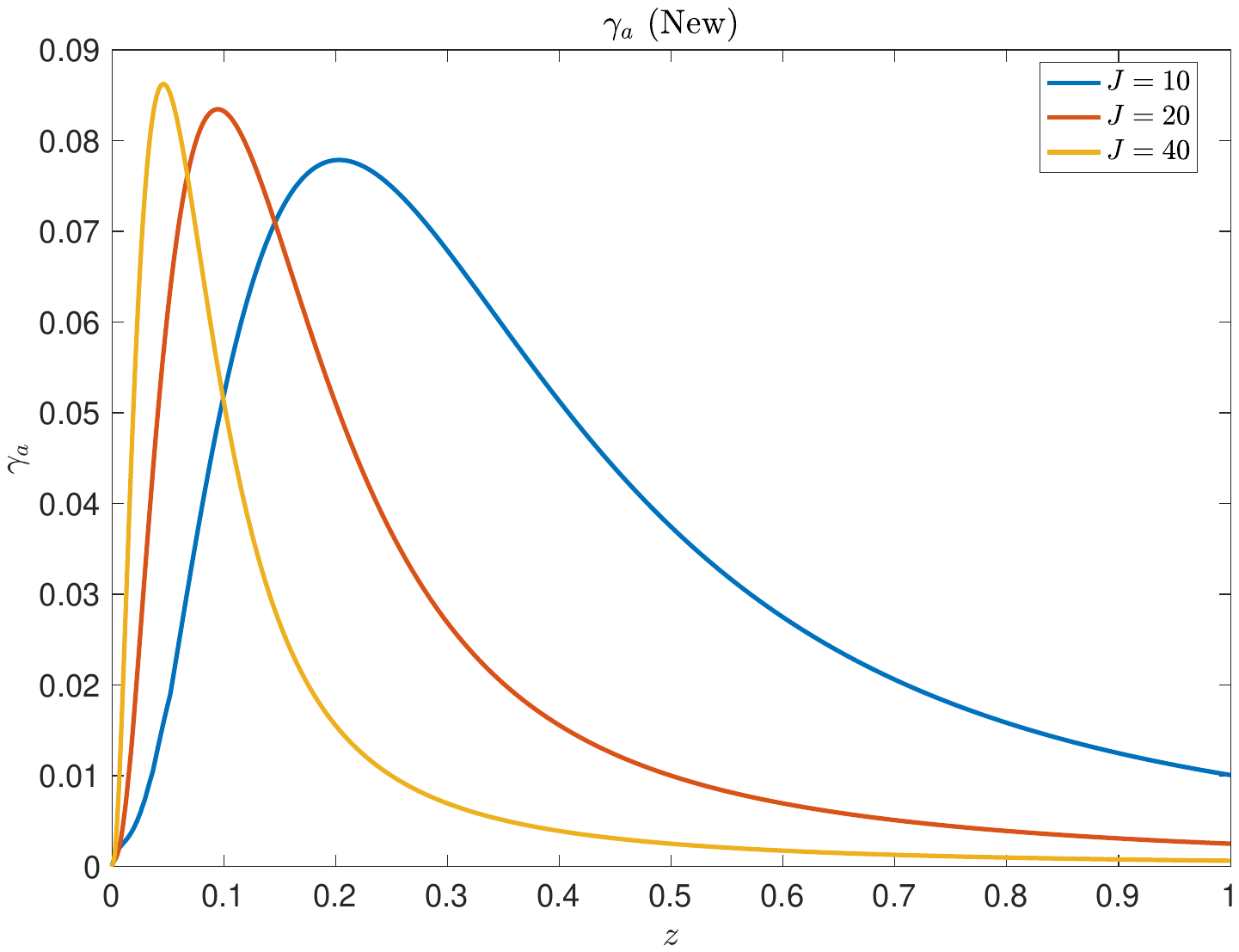}
    \end{subfigure}
    \caption{The graph of $\gamma_a(z)$ when BE is chosen as the CP. Left: Maday's update; Right: New update.}
    \label{fig:gamma_a_BE}
\end{figure}

\begin{figure}[htbp]
    \centering
    \begin{subfigure}[b]{0.48\textwidth}
\includegraphics[width=\linewidth]{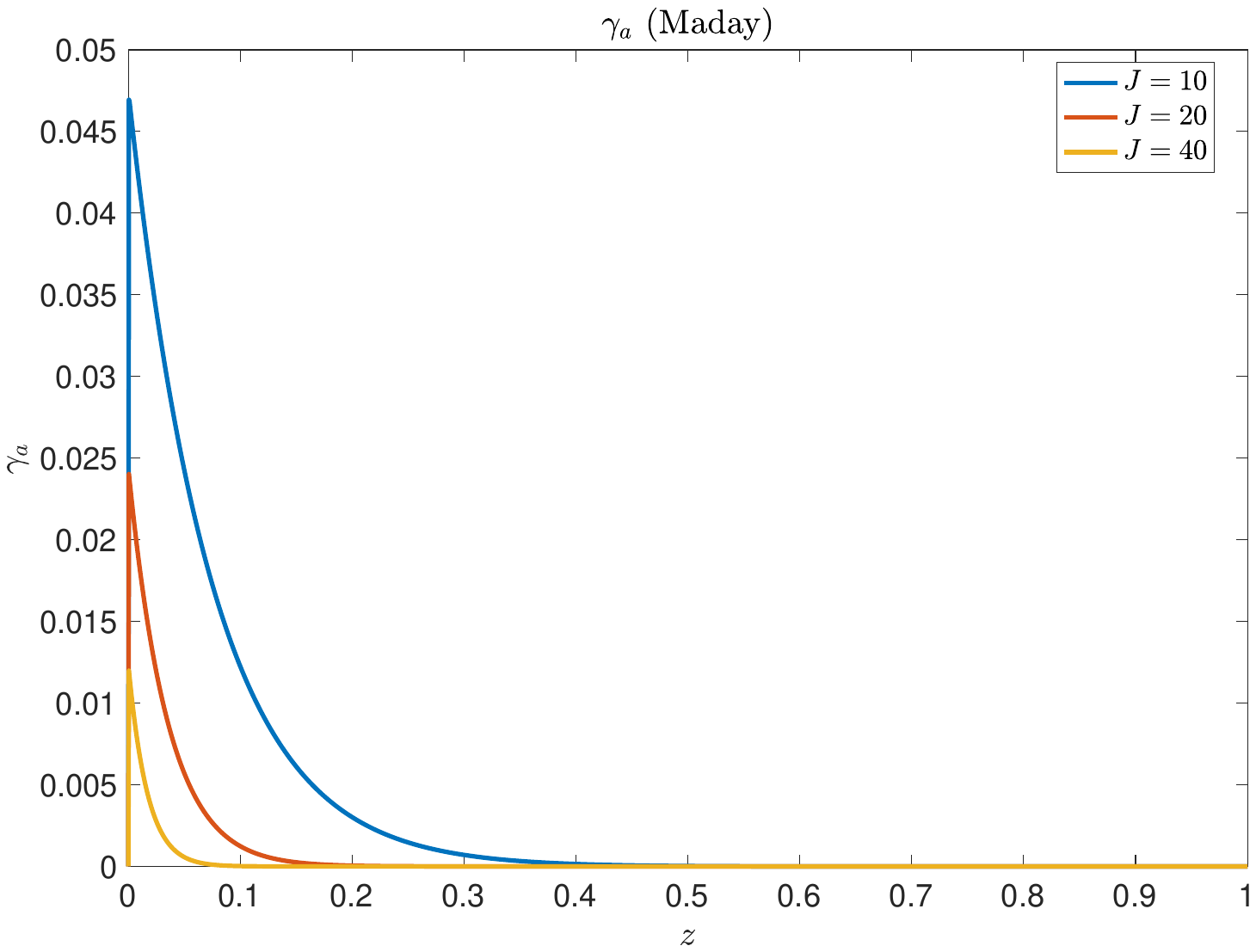}
    \end{subfigure}
    \begin{subfigure}[b]{0.48\textwidth}
        \includegraphics[width=\linewidth]{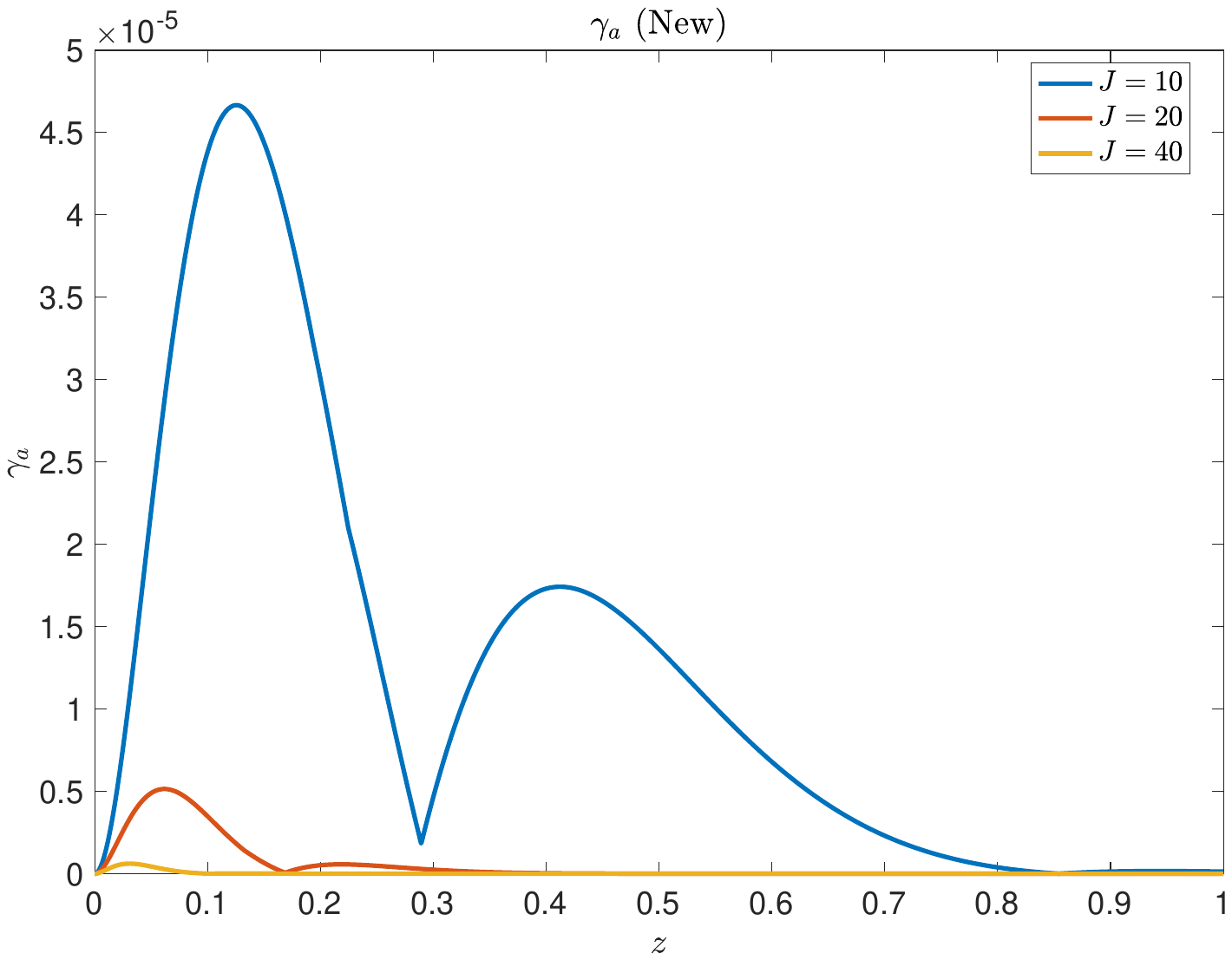}
    \end{subfigure}
    \caption{The graph of $\gamma_a(z)$ when the exact solver is chosen as the CP. Left: Maday's update; Right: New update.}
    \label{fig:gamma_a_Exact}
\end{figure}

\subsubsection{$\gamma_{b,p}$}
The expression of $\gamma_{b,p}$ is given in \eqref{eqn:g_b}. 
\begin{figure}[htbp]
    \centering
    \begin{subfigure}[b]{0.48\textwidth}
\includegraphics[width=\linewidth]{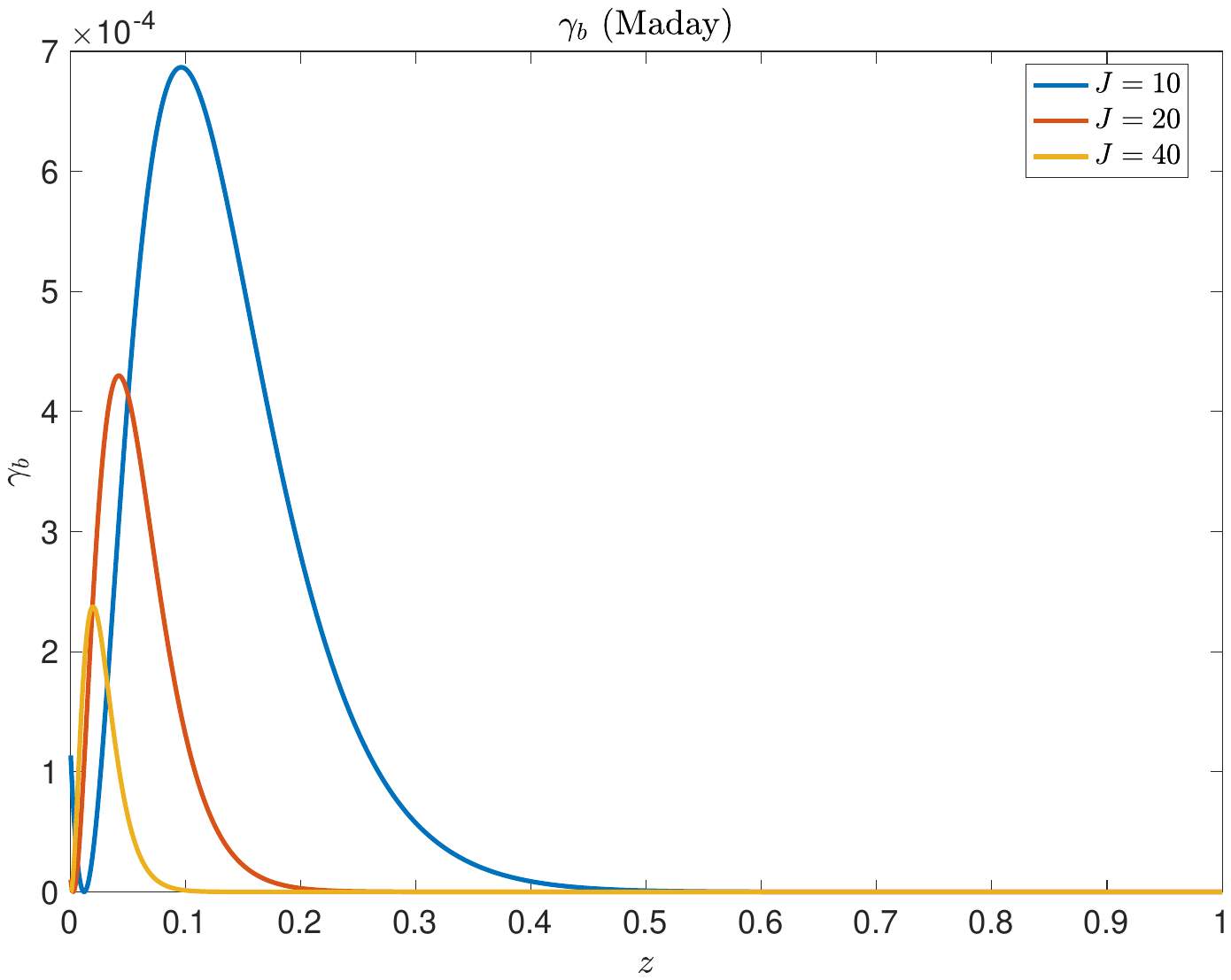}
    \end{subfigure}
    \begin{subfigure}[b]{0.48\textwidth}
        \includegraphics[width=\linewidth]{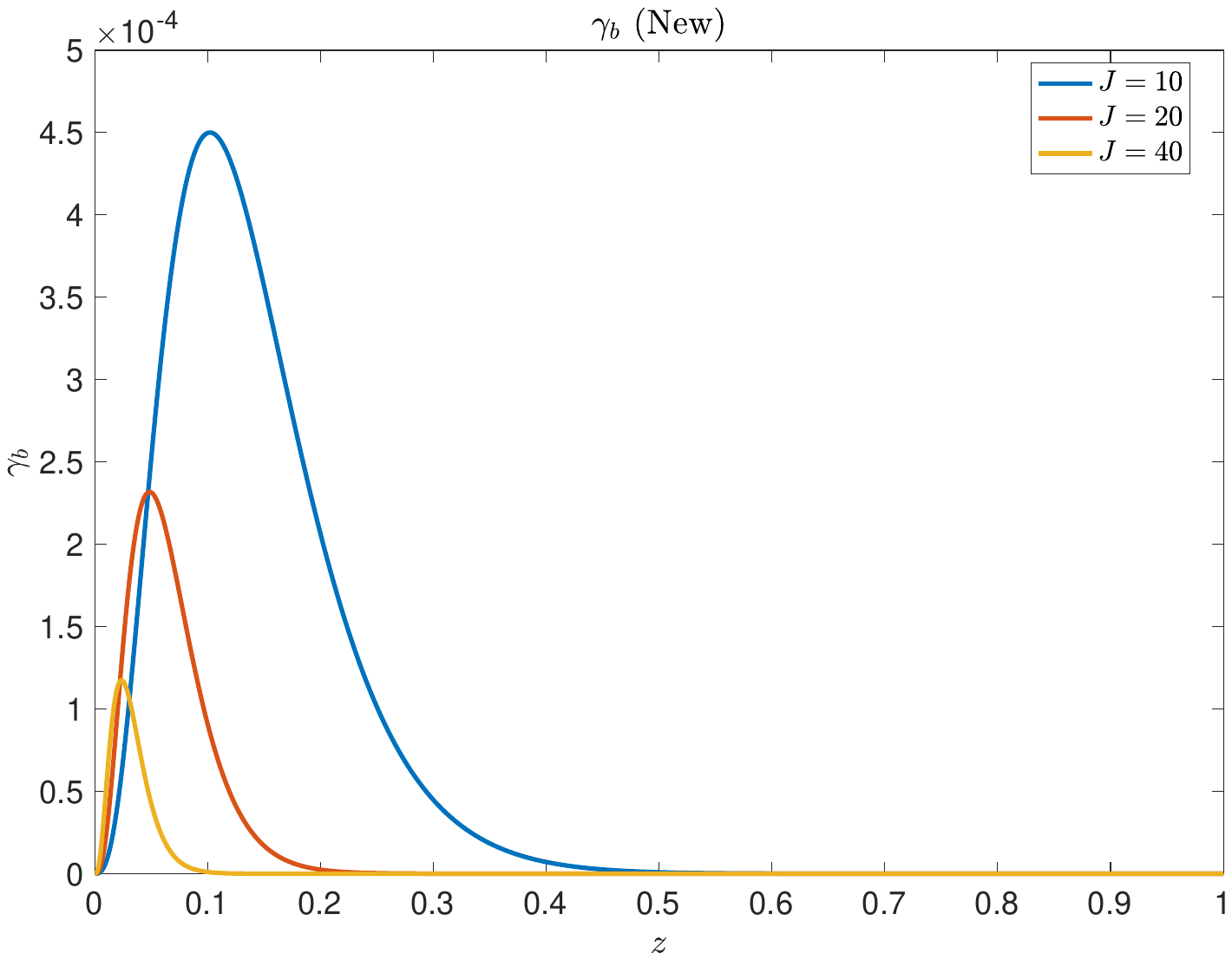}
    \end{subfigure}
    \caption{The graph of $\gamma_b(z)$ when BE is chosen as the CP. Left: Maday's update; Right: New update.}
    \label{fig:gamma_b_BE}
\end{figure}

\begin{figure}[htbp]
    \centering
    \begin{subfigure}[b]{0.48\textwidth}
\includegraphics[width=\linewidth]{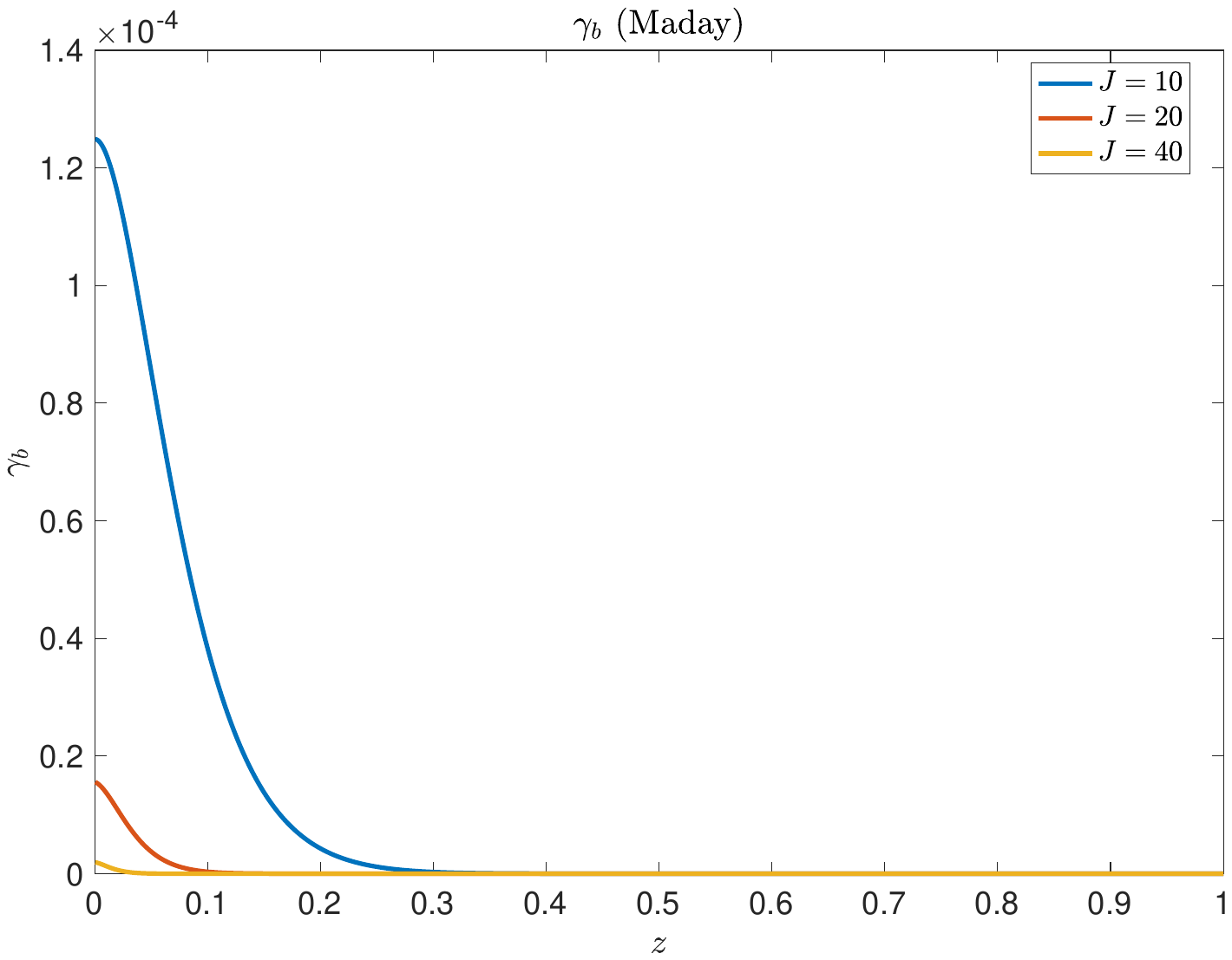}
    \end{subfigure}
    \begin{subfigure}[b]{0.48\textwidth}
        \includegraphics[width=\linewidth]{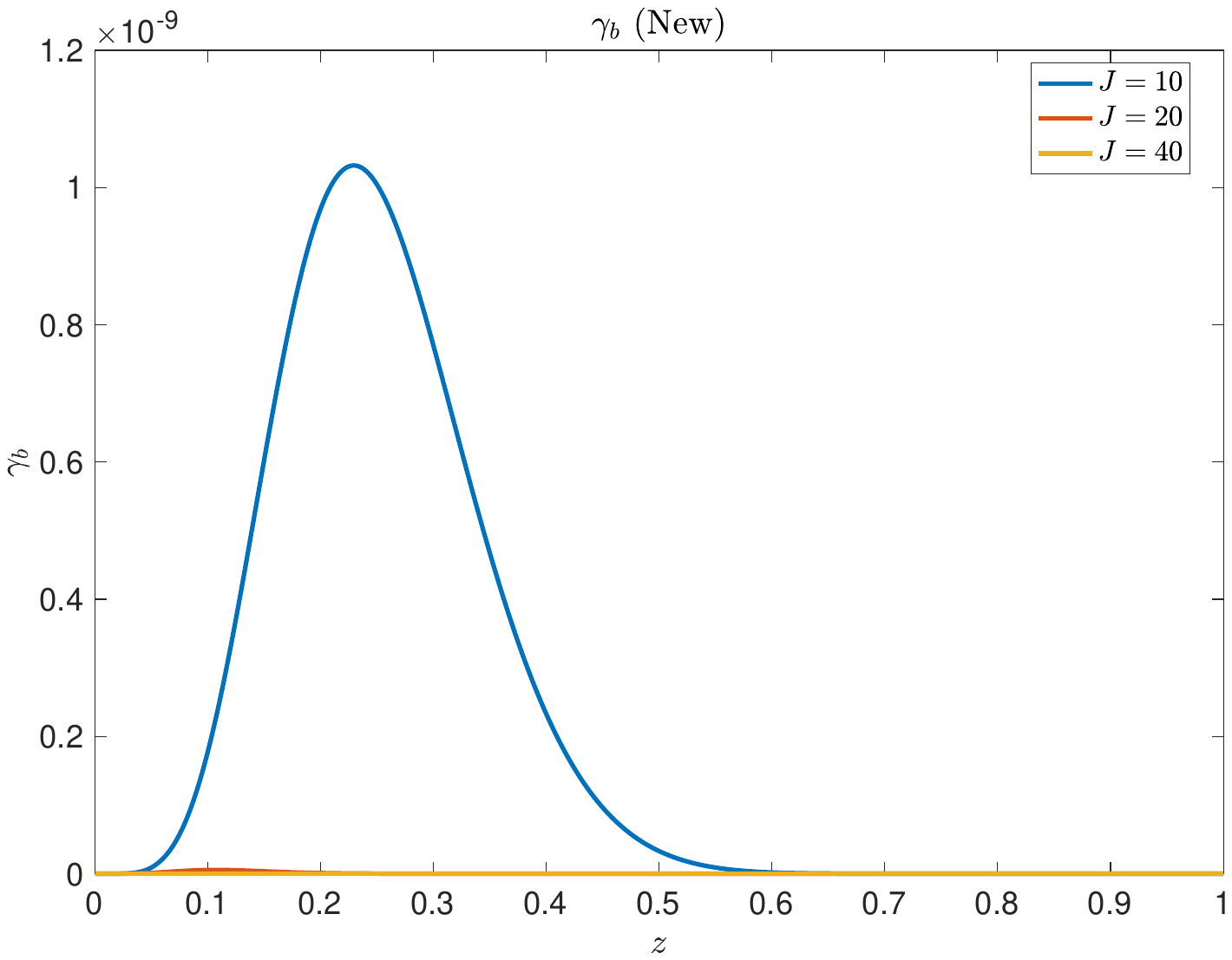}
    \end{subfigure}
    \caption{The graph of $\gamma_b(z)$ when the exact solver is chosen as the CP. Left: Maday's update; Right: New update.}
    \label{fig:gamma_b_Exact}
\end{figure}

\subsubsection{$\gamma_{c,p}\gamma_{d,p}$}
The expression of $\gamma_{c,p}\gamma_{d,p}$ is given in \eqref{eqn:g_c} and \eqref{eqn:g_d,d_e}.
\begin{figure}[htbp]
    \centering
    \begin{subfigure}[b]{0.48\textwidth}
\includegraphics[width=\linewidth]{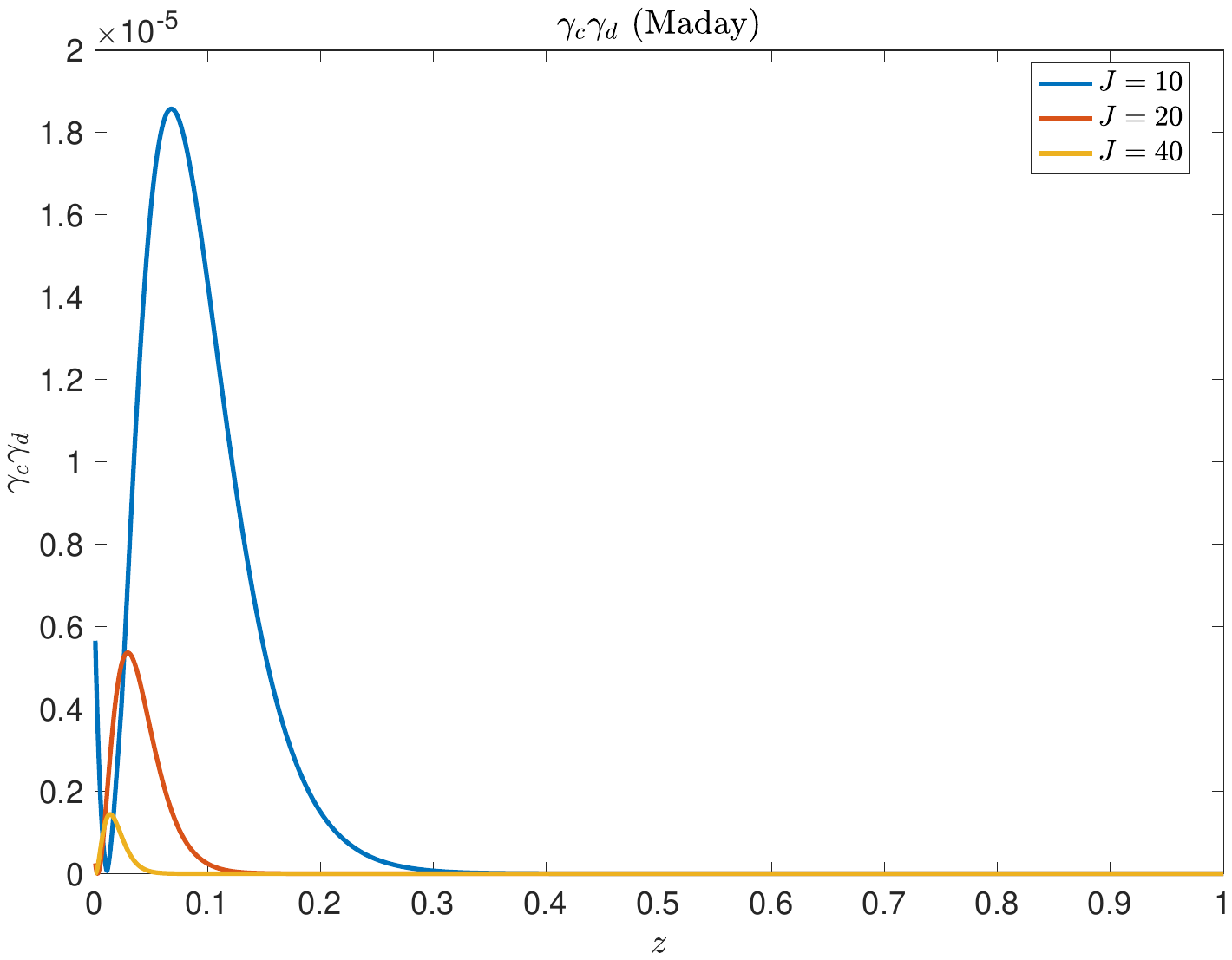}
    \end{subfigure}
    \begin{subfigure}[b]{0.48\textwidth}
    \includegraphics[width=\linewidth]{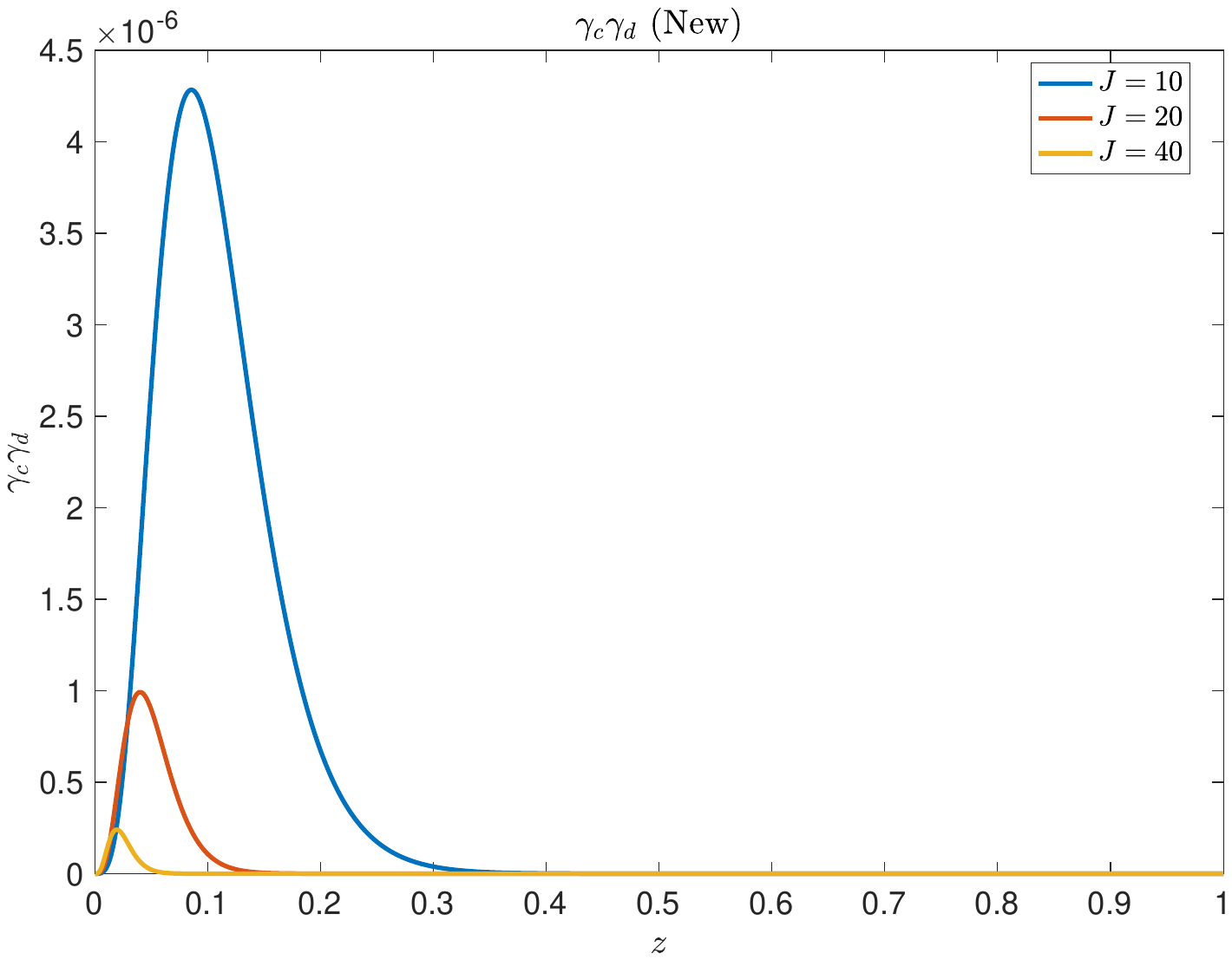}
    \end{subfigure}
    \caption{The graph of $\gamma_c(z)\gamma_d(z)$ when BE is chosen as the CP. Left: Maday's update; Right: New update.}
    \label{fig:gamma_cd_BE}
\end{figure}

\begin{figure}[htbp]
    \centering
    \begin{subfigure}[b]{0.48\textwidth}
\includegraphics[width=\linewidth]{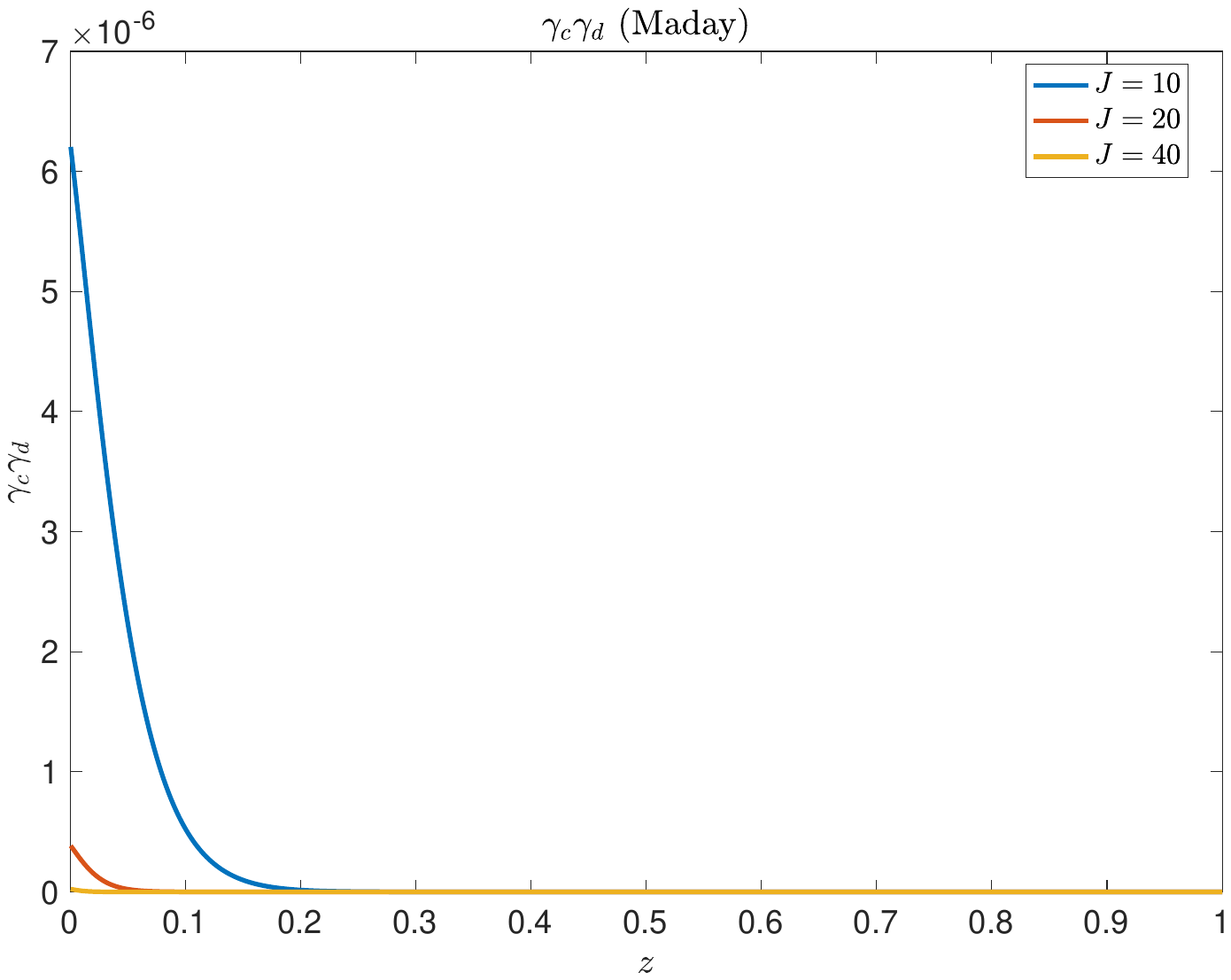}
    \end{subfigure}
    \begin{subfigure}[b]{0.48\textwidth}
        \includegraphics[width=\linewidth]{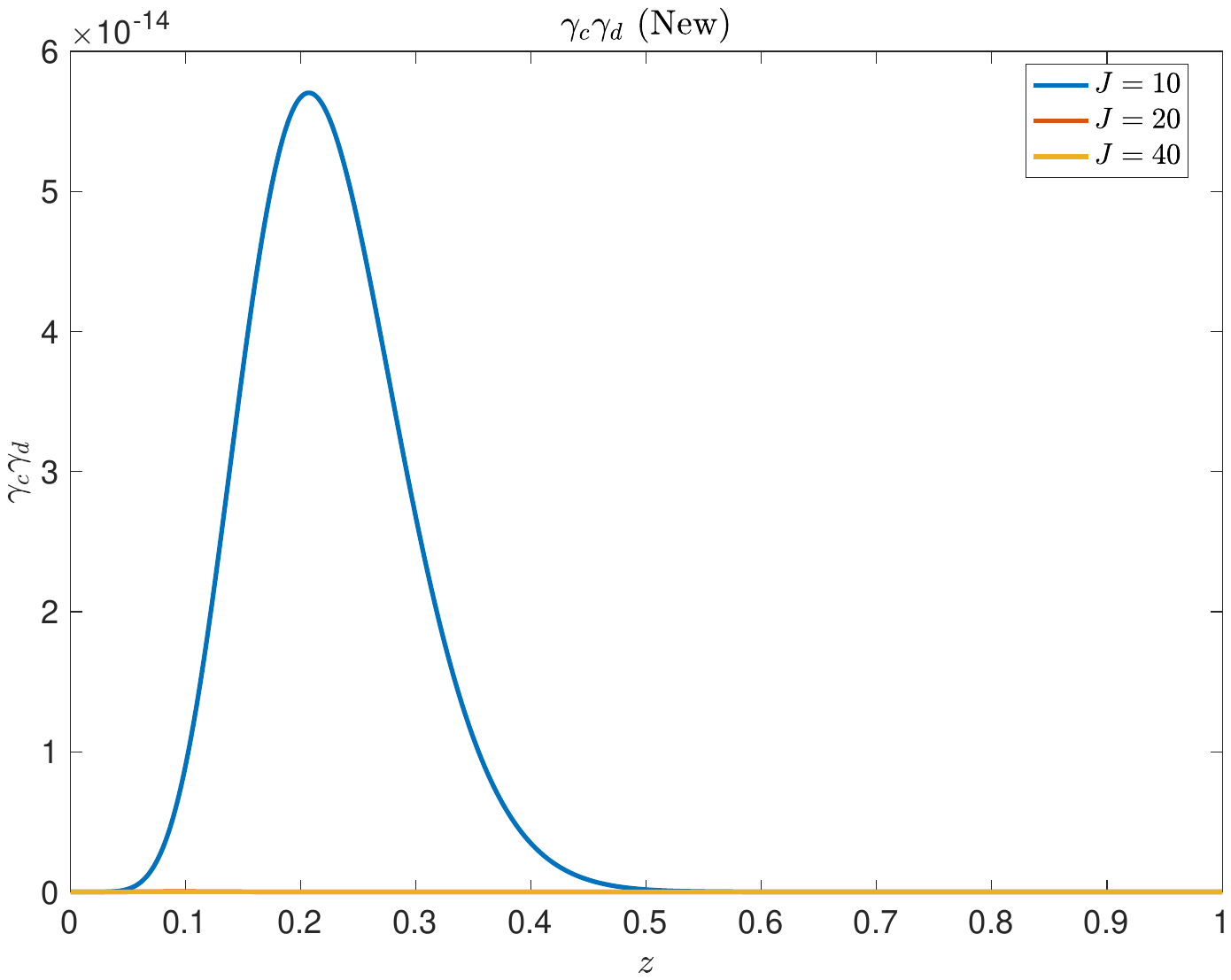}
    \end{subfigure}
    \caption{The graph of $\gamma_c(z)\gamma_d(z)$ when the exact solver is chosen as the CP. Left: Maday's update; Right: New update.}
    \label{fig:gamma_cd_Exact}
\end{figure}

\subsubsection{$\gamma_{c,p}\gamma_{d,p} \gamma_{e,p}$}
The expression of $\gamma_{c,p}\gamma_{d,p} \gamma_{e,p}$ is given in \eqref{eqn:g_c} and \eqref{eqn:g_d,d_e}.
\begin{figure}[htbp]
    \centering
    \begin{subfigure}[b]{0.48\textwidth}
\includegraphics[width=\linewidth]{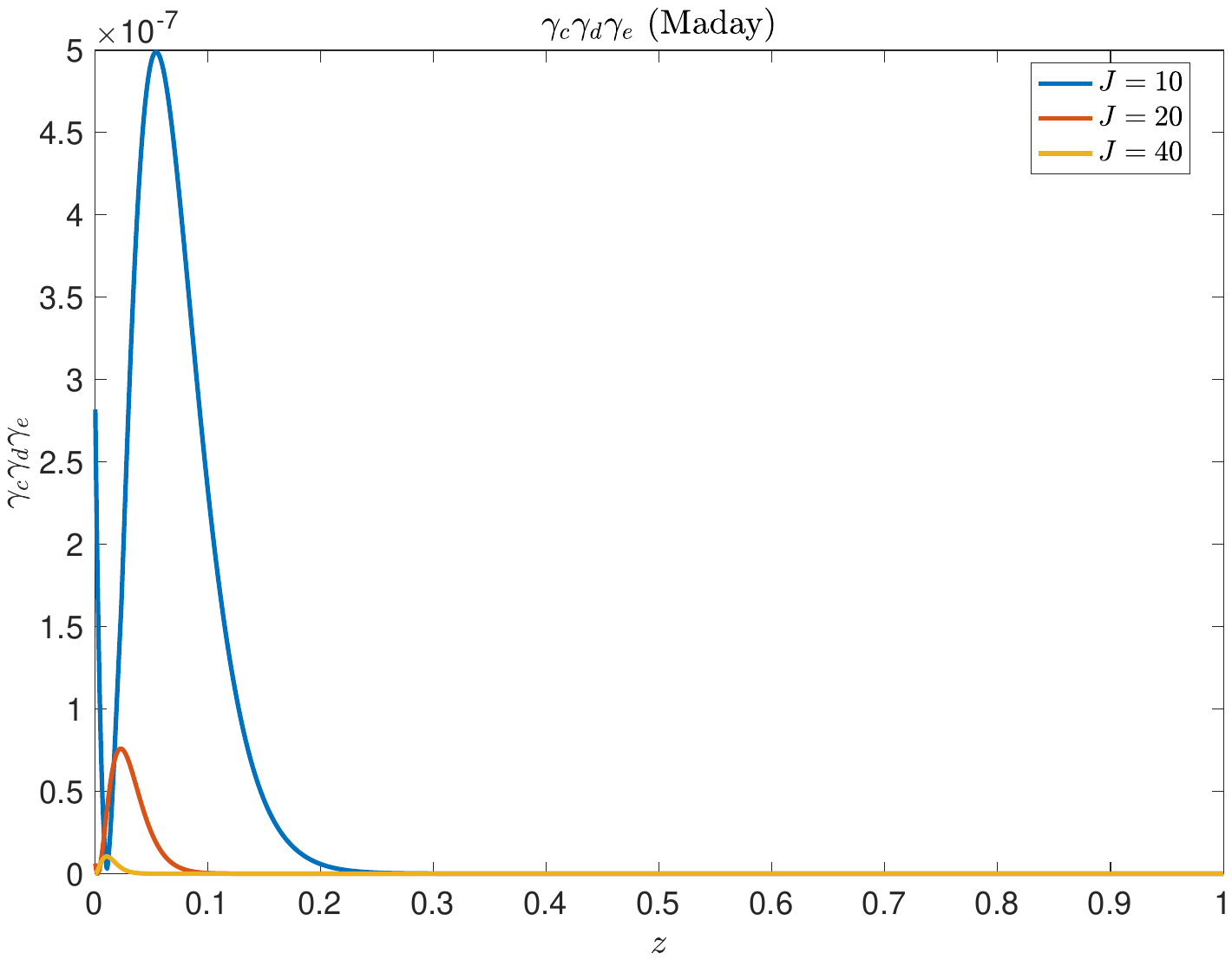}
    \end{subfigure}
    \begin{subfigure}[b]{0.48\textwidth}
        \includegraphics[width=\linewidth]{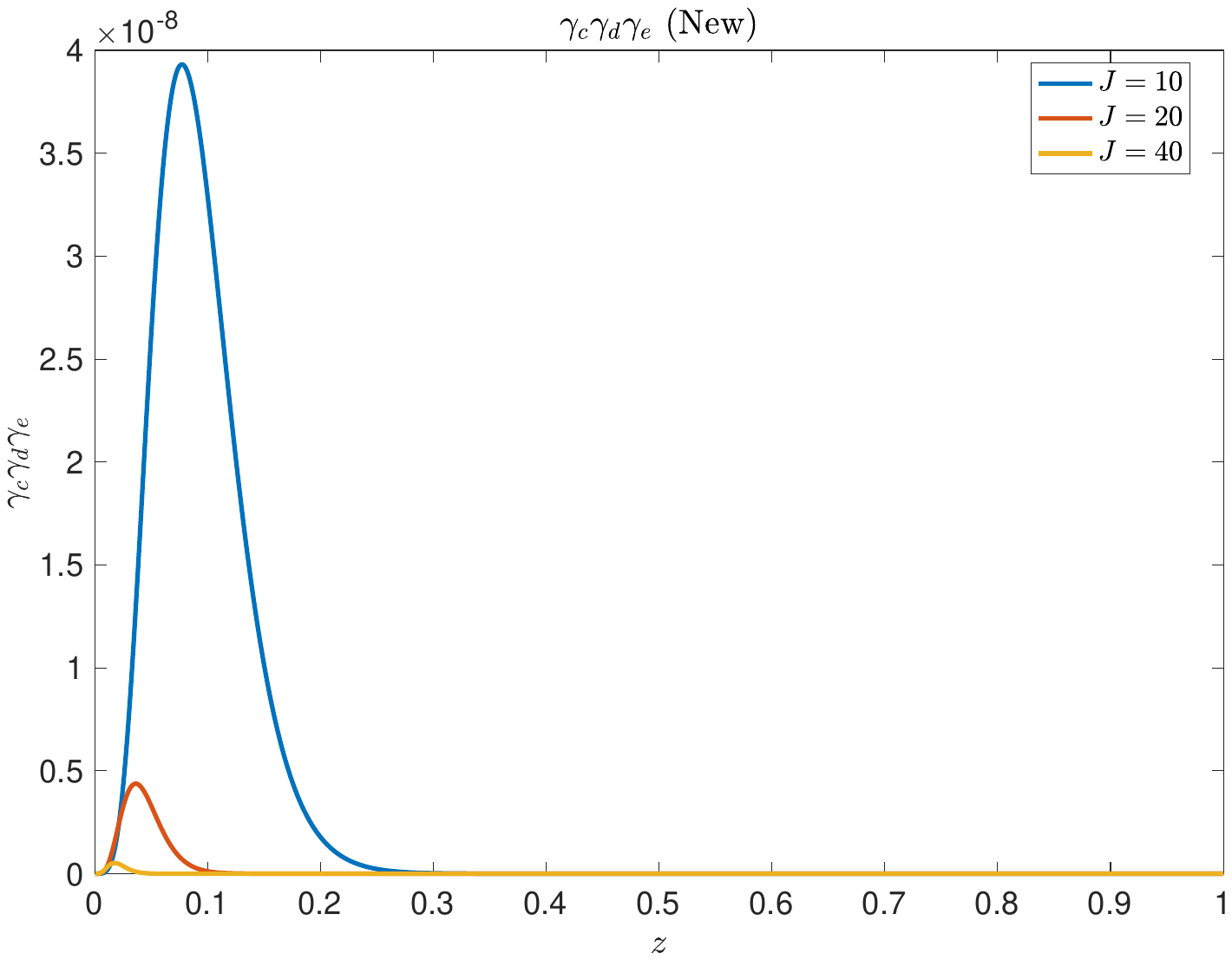}
    \end{subfigure}
    \caption{The graph of $\gamma_c(z)\gamma_d(z)\gamma_e(z)$ when BE is chosen as the CP. Left: Maday's update; Right: New update.}
    \label{fig:gamma_cde_BE}
\end{figure}

\begin{figure}[htbp]
    \centering
    \begin{subfigure}[b]{0.48\textwidth}
\includegraphics[width=\linewidth]{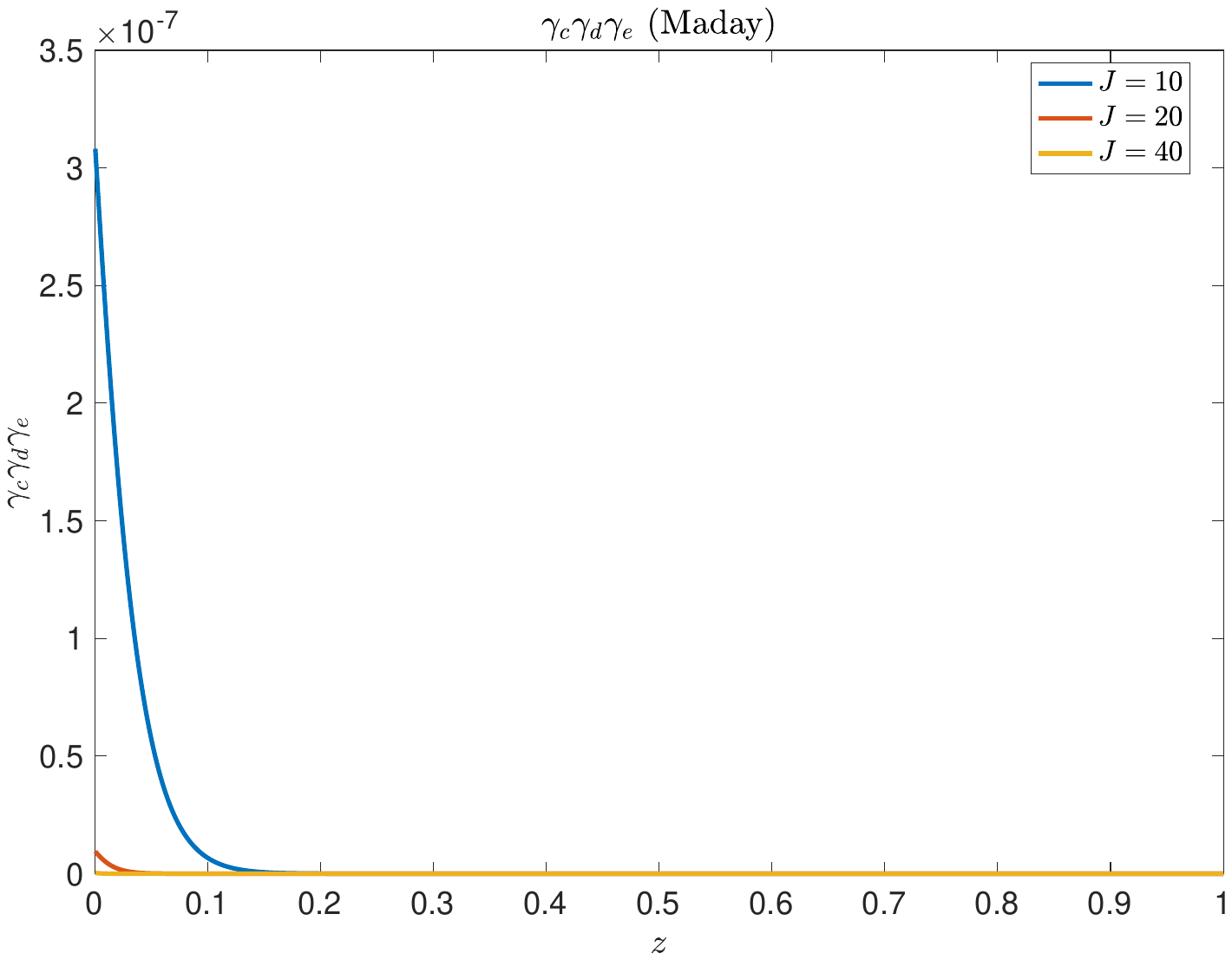}
    \end{subfigure}
    \begin{subfigure}[b]{0.48\textwidth}
        \includegraphics[width=\linewidth]{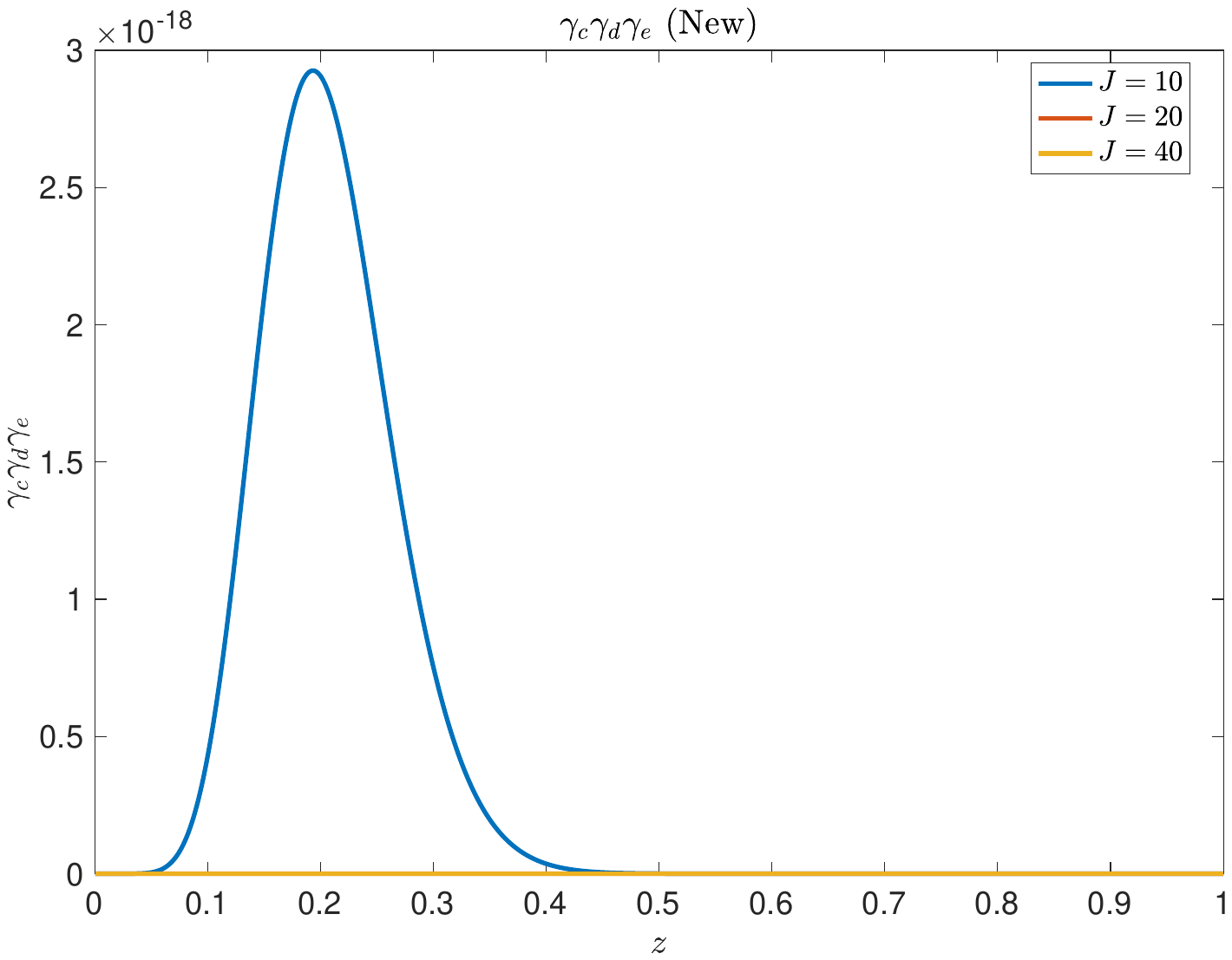}
    \end{subfigure}
    \caption{The graph of $\gamma_c(z)\gamma_d(z) \gamma_e (z)$ when the exact solver is chosen as the CP. Left: Maday's update; Right: New update.}
    \label{fig:gamma_cde_Exact}
\end{figure}

\bibliographystyle{siamplain}
\bibliography{references}